\pdfoutput=1 
\documentclass{amsart}
\usepackage{tikz-cd}
\usepackage{hyperref}
\usepackage{graphicx,color}
\usepackage{amssymb,amsmath, bm}
\usepackage{mathrsfs}
\usepackage{enumerate}
\usepackage{tkz-graph}
\usepackage{ragged2e}
\usepackage{enumitem}
\usepackage{adjustbox}
\usepackage{tikz}
\usepackage[all]{xy}
\usetikzlibrary{arrows,decorations.pathmorphing,automata,backgrounds}
\usetikzlibrary{backgrounds,positioning}
\usepackage{transparent}
\usepackage{graphicx,import}
\usepackage{subfigure}
\usepackage{enumitem}

\tikzset{plain/.style={circle,draw,very thick,
		inner sep=0pt,minimum size=6mm}}
\tikzset{empty/.style={circle,inner sep=0pt,minimum size=6mm}}
\tikzset{emptyvt/.style={circle,inner sep=0pt,minimum size=0mm}}

\newcommand{\pentatree}{\raisebox{-4pt}{\includegraphics[height=12pt]{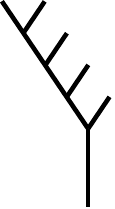}}}
\newcommand{\trivalenttree}{\raisebox{-4pt}{\includegraphics[height=12pt]{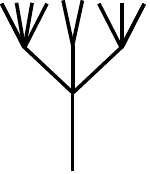}}}
\newcommand{\bluetree}{\raisebox{-4pt}{\includegraphics[height=14pt]{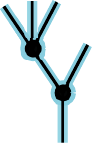}}}
\newcommand{\orangetree}{\raisebox{-4pt}{\includegraphics[height=14pt]{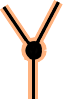}}}
\newcommand{\yellowtree}{\raisebox{-4pt}{\includegraphics[height=14pt]{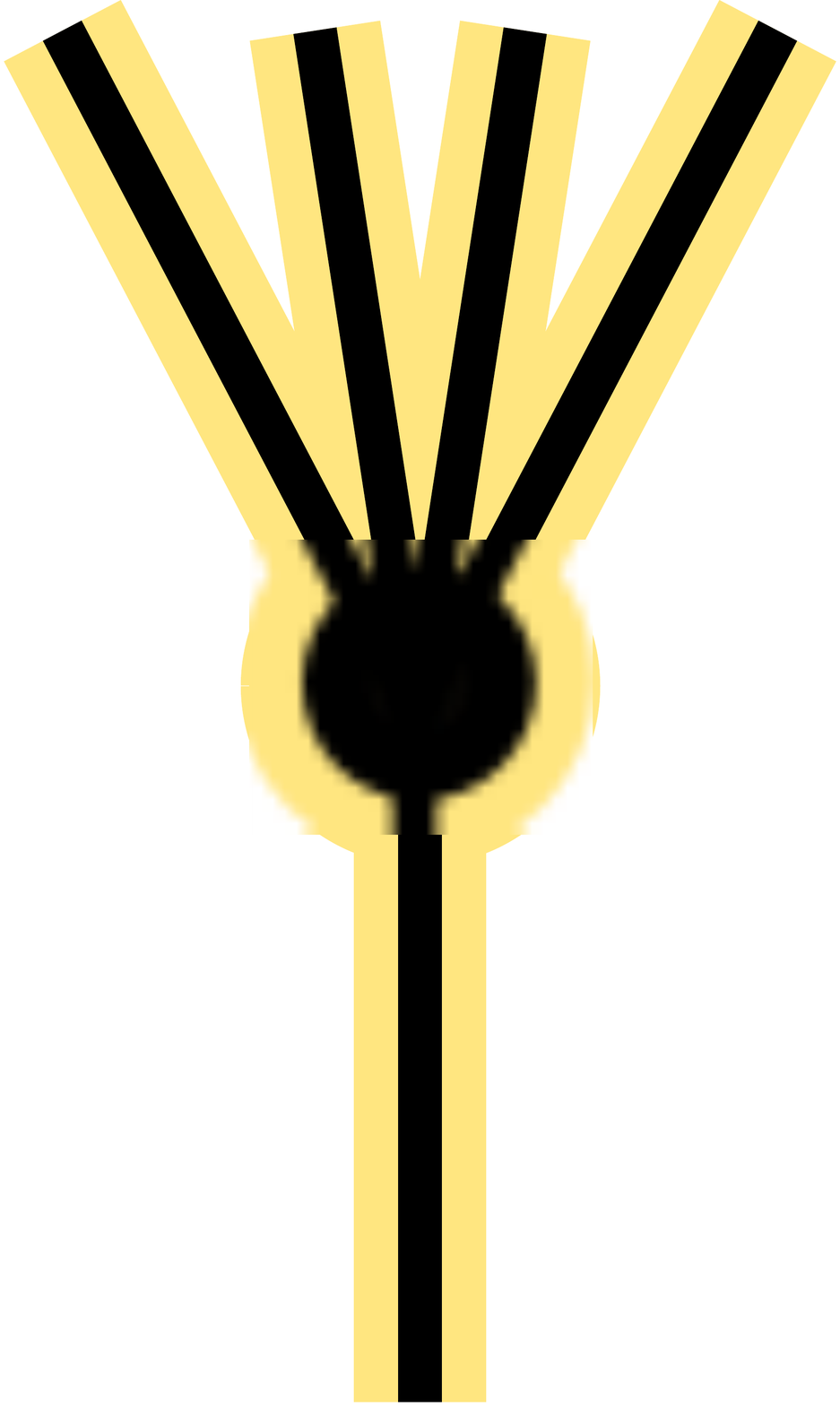}}}
\newcommand{\pinktree}{\raisebox{-4pt}{\includegraphics[height=14pt]{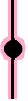}}}
\newcommand{\bluetreetwo}{\raisebox{-4pt}{\includegraphics[height=14pt]{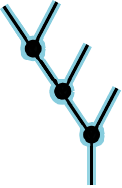}}}

\DeclareMathOperator{\cact}{\mathcal{C}act}

\newcommand{\set}{\mathsf{Set}}

\newcommand{\N}{\mathbb{N}}

\newcommand{\B}{\mathcal{B}}
\newcommand{\F}{\mathcal{F}}
\newcommand{\pp}{\mathcal{P}}
\newcommand{\OO}{\mathcal{O}}

\newcommand{\id}{\operatorname{id}}

\newcommand{\Hom}{\operatorname{Hom}}
\newcommand{\Map}{\operatorname{Map}}
\newcommand{\CoEnd}{\operatorname{CoEnd}}
\newcommand{\del}{\partial}

\newcommand{\ie}{iE}%{Ed_{i}}

\newcommand{\om}{\Omega}
\newcommand{\Path}{\text{Path}}

\newtheorem{theorem}{Theorem}[section]
\newtheorem{prop}[theorem]{Proposition}%[section]
\newtheorem{lemma}[theorem]{Lemma}%[section]
\newtheorem{cor}[theorem]{Corollary}%[section]
\newtheorem*{thm*}{Theorem}

\newtheorem{Th}{Theorem}

\theoremstyle{definition}
\newtheorem{definition}[theorem]{Definition}%[section]
\newtheorem{example}[theorem]{Example}%[section]
\newtheorem{remark}[theorem]{Remark}%[section]

\numberwithin{equation}{section}

\let\oldtocsection=\tocsection
\let\oldtocsubsection=\tocsubsection
\renewcommand{\tocsection}[2]{\hspace{0em}\oldtocsection{#1}{#2}}
\renewcommand{\tocsubsection}[2]{\hspace{1em}\oldtocsubsection{#1}{#2}}
\DeclareRobustCommand{\SkipTocEntry}[5]{}

\title{An infinity operad of normalized cacti}

\author[L. Basualdo Bonatto]{Luciana Basualdo Bonatto}
\address{Mathematical Institute \\ University of Oxford \\ Oxford, UK}
\email{luciana.bonatto@maths.ox.ac.uk}

\author[S. Chettih]{Safia Chettih}
\address{Department of Mathematics \\ Southwestern University \\ Georgetown, Texas, USA}
\email{chettihs@southwestern.edu}

\author[A. Linton]{Abigail Linton} \address{School of Mathematical Sciences \\ University of Southampton \\ UK} \email{a.linton@soton.ac.uk}

\author[S. Raynor]{Sophie Raynor}
\address{Department of Mathematics and Statistics \\ Macquarie University \\ NSW, Australia}
\email{sophie.raynor@mq.edu.au}

\author[M. Robertson]{Marcy Robertson}
\address{School of Mathematics and Statistics \\ The University of Melbourne \\ Melbourne, Victoria, Australia}
\email{marcy.robertson@unimelb.edu.au}

\author[N. Wahl]{Nathalie Wahl}
\address{Department of Mathematical Sciences \\ University of Copenhagen \\ Copenhagen, Denmark}
\email{wahl@math.ku.dk}

\date{\today}

\begin{document}

\begin{abstract}
We show that normalized cacti form an $\infty$-operad in the form of a dendroidal space satisfying a weak Segal condition. 
To do this, we introduce a new topological operad of bracketed trees and an enrichment of the dendroidal category $\om$.
\end{abstract}

\maketitle

\section{Introduction}
 Gluing surfaces along their boundaries allows to define composition laws
 that have been used to define cobordism categories, as well as operads and props associated to surfaces. These have played an important role in recent years, for example in  constructing topological field theory or computing the homology of the moduli space of Riemann surfaces. Of particular interest is the cobordism category whose morphism spaces are moduli spaces of Riemann surfaces. It has long been known that such moduli spaces admit a graph model: they have the homotopy type of spaces of metric fat graphs \cite{BowEps,Har84,Pen87}. The composition of moduli spaces induced by the gluing of surfaces was modeled using graphs in \cite[Construction 3.29]{egas_comparing}. Though the resulting composition is associative on the associated chain complex, it is not associative on the space level, and, at present, it is not known how to make it associative, or even coherently homotopy associative \cite[Remark 3.31]{egas_comparing}. In genus $0$, this graph model of the cobordism category includes normalized cacti (eg. \cite[Remark~ 2.8]{whal_westerland}), whose composition was also known not to be associative \cite[Remark~2.3.19]{K05}. The goal of our paper is to show that the composition of normalized cacti is associative up to all higher homotopies, in the precise sense that normalized cacti form an $\infty$-operad in the way detailed below. We expect that the technique presented here can be extended to likewise show that the composition in the graph model of the cobordism category is also associative up to all higher homotopies.

\medskip

\begin{figure}[ht]    \centering   
\includegraphics[width=0.3\textwidth]{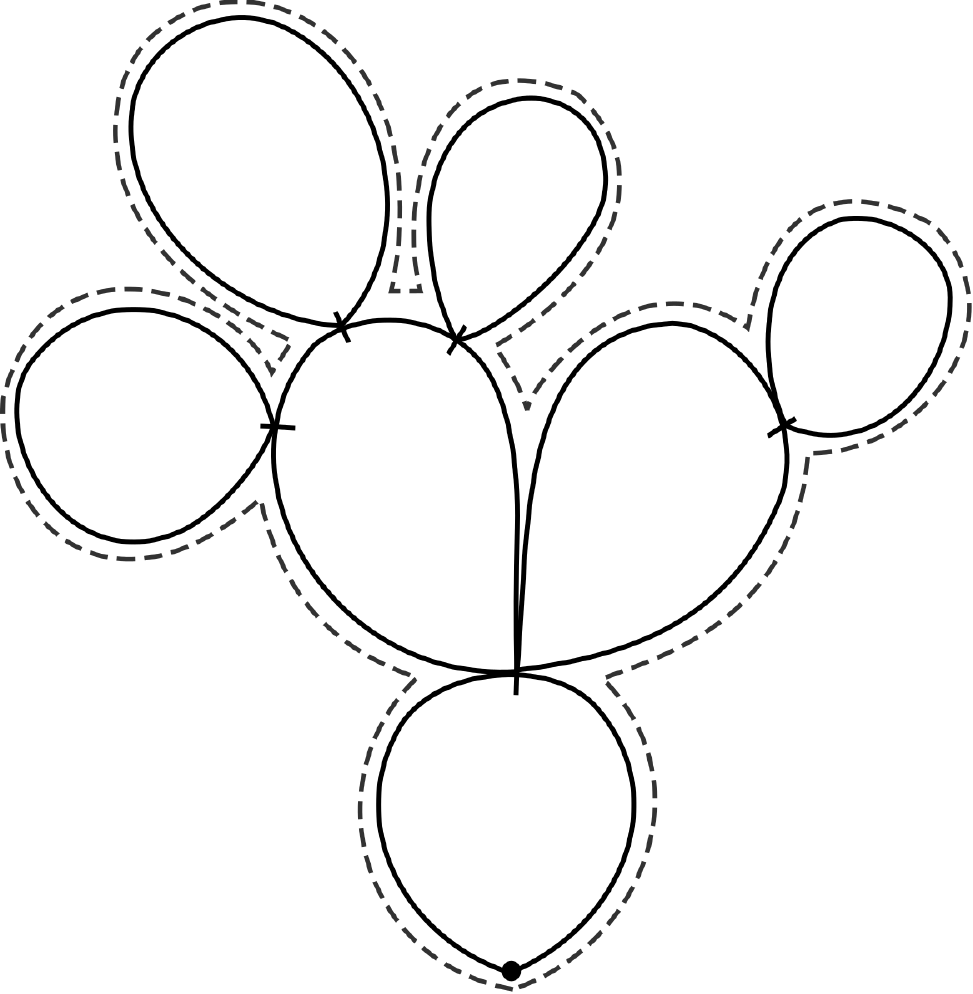}
\caption{Spineless cactus with 7 lobes, with its outside  the dotted line.}\label{fig:cactus_in_pot} \end{figure}
A cactus is a treelike configuration of circles (Figure~\ref{fig:cactus_in_pot}). The cactus operad, originally introduced by Voronov \cite[Section 2.7]{V05}, and its spineless version, introduced by Kaufmann \cite[Section 2.3]{K05}, are models for the framed and unframed little disc operads respectively \cite[Section 3.2.1]{K05}. Operadic composition is by insertion: identifying the outside contour of one cactus with the lobe of another cactus and scaling the inserted cactus appropriately. Here we  work with the spineless version for simplicity.

A cactus is \emph{normalized} if each circle in the cactus has circumference of length one. The space of all normalized cacti with $k$-lobes is denoted by $\cact^1(k)$ and these spaces assemble into the symmetric sequence $\cact^1=\{\cact^1(k)\}_{k\ge 0}$, with each $\cact^1(k)\subset \cact(k)$ a homotopy equivalent subspace, for $\cact(k)$ the space of all cacti with $k$ lobes.  (See \cite[Sec 2.3]{K05}.) Composition of normalized cacti is defined by insertion as for the cactus operad, but instead of scaling the 
inserted cactus to the size of the lobe it is inserted in, one scales the 
lobe to the size of the inserted cactus. Surprisingly, as illustrated in Figure~\ref{fig:CactCompUnassoc}, this new composition is not associative (\cite[Remark~2.3.19]{K05}). So, normalized cacti do \emph{not} form an operad. 
This non-associative composition is, however, the one relevant to the graph model of the cobordism category, as we explain in Remark~\ref{rem:cobcomposition}.

\medskip

Our main result is that this composition of normalized cacti is part of an $\infty$-operad structure. In this paper, an $\infty$-operad is a
dendroidal Segal space in the sense of \cite[Definition 8.1]{cm2}.\footnote{We use a slight variation of the original definition, for details see Definition~\ref{def: dendroidal space} Remark~\ref{remark: Segal map}}A dendroidal space is a space-valued $\om$-diagram, where the dendroidal category $\om$ is the full subcategory of colored 
operads freely generated by trees (Definition~\ref{def: segal condition}). Dendroidal spaces are closely related to operads
since there is an isomorphism of categories
between one-colored topological operads and reduced 
(i.e.~monochromatic) dendroidal spaces satisfying a strict Segal condition. 
Dendroidal spaces that satisfy a weak Segal 
condition are a model for $\infty$-operads, Quillen equivalent to all other known models for $\infty$-operads 
including: dendroidal sets satisfying an inner Kan condition \cite[Proposition 
6.3; Theorem 8.15]{cm2}, Lurie's $\infty$-operads \cite[Section 
2.5]{MR3545944}; \cite[Corollary 1.2]{Chu_Haugseng_Huets} and Barwick's 
complete Segal operads \cite[Theorem 1.1]{Chu_Haugseng_Huets}.

Operads can be described as algebras over the operad  of operads $\OO$, an operad whose elements can be represented by certain trees (Definition~\ref{def:operad of operads}). 
In Section~\ref{sec:BO}, we define a bracketing of a tree and use it to construct a new topological operad $B\OO$ (Definition~\ref{def: BO}) whose algebras are homotopy associative versions of operads: 
Any $B\OO$--algebra has an underlying symmetric sequence and a preferred composition, but the composition  is only associative up to coherent homotopy. 
The operad $B\OO$ is the realization of an operad whose operations lie in the poset of bracketings of the trees in $\OO$. 
Given a composition on a symmetric sequence, this operad gives a hands-on way to keep track of the homotopies required to show that it is coherently homotopy associative. We illustrate how to construct a $B\OO$--algebra in practice by showing:

\begin{Th}[Theorem~\ref{thm: cact is BO alg}]\label{ThCactisBO}
The symmetric sequence $\{\cact^1(k)\}_{k\ge 0}$ of normalized cacti, together with the $\cact^1$ composition described above, extends to a $B\OO$--algebra structure.  
\end{Th}

In Section~\ref{sec:thickening-omega}, we show that this hands-on notion of an operad up to homotopy is related to more well-known notions of $\infty$-operads. We achieve this by showing that any $B\OO$--algebra defines a dendroidal Segal space.  
First we construct a topological enrichment $\widetilde\om_0$ of the dendroidal category $\om$ whose objects are trees, as for $\om$, but whose morphisms are  the realisation of certain posets of bracketings in trees, defined in a similar fashion to the operad $B\OO$. 
Diagrams over this thickened dendroidal category $\widetilde\om_0$ are types of homotopy coherent dendroidal spaces.
In Proposition~\ref{prop:homotopy Segal implies Segal}, we show that a homotopy coherent $\widetilde\om_0$--diagram can be rectified to a strict $\om$--diagram that satisfies the Segal condition if the original diagram did. 
By defining a nerve functor that takes a $B\OO$-algebra to the category of strictly reduced $\widetilde\om_0$-diagrams that satisfy a strict Segal condition, we prove the following: 

\begin{Th}[Theorem~\ref{thm:W_0O-Omega_0} and %Theorem~\ref{thm: ThBO} 
Proposition~\ref{prop:homotopy Segal implies Segal}
]\label{ThBO}
There is an isomorphism of categories between $B\OO$-algebras and the category of $\widetilde\om_0$-diagrams that satisfy a strict Segal condition. In particular, as each $\widetilde\om_0$-diagram can be rectified, every $B\OO$-algebra is an $\infty$-operad. 
\end{Th}

By combining Theorem~\ref{ThCactisBO} and Theorem~\ref{ThBO}, normalized cacti are a rare example of an $\infty$-operad that does not arise via the application of a nerve construction to a known (discrete or topological) operad (Corollary~\ref{cor: cact is dendroidal segal space}). Indeed, to our knowledge, the only such examples include the weak operad of configuration spaces \cite[Corollary 5]{Hackney_config} and examples that arise as a result of completion as in \cite[Proposition 5.1]{BHR}. 
\medskip

The idea of using a resolution of the operad  of operads $\OO$ to model $\infty$-operads is not a new one. 
A classical way to resolve an operad is to apply 
the Boardman-Vogt $W$-construction. Applied to the operad $\OO$, one gets an operad $W\OO$ whose algebras are also $\infty$-operads: there exists a zig-zag of Quillen equivalences between the category of $W\OO$-algebras and reduced dendroidal Segal spaces. (For example this can be seen by combining Theorem 4.1 of \cite{bm_resolution} with either Theorem 1.1 of \cite{Bergner_Hackney_14} or Theorem 8.15 of \cite{cm3}.) 
However, the operad $W\OO$ is not easy to work with directly. Indeed, its elements are trees (from the $W$--construction) whose vertices are themselves decorated by trees (from the operad $\OO$), where the first trees compose by grafting and the second trees compose by vertex substitution. In Appendix~\ref{sec: BO and WO}, we show that the operad $B\OO$ is actually isomorphic to a quotient $W_0\OO$ of $W\OO$: 

\begin{Th}[Theorem~\ref{thm:BOWO}]
There exists an isomorphism of topological operads $W_0\OO\cong B\OO$.
\end{Th}

A $W\OO$--algebra is an operad up to homotopy, where the symmetric group action, the unit and associativity relation are all assumed to hold only up to coherent homotopy. ($W\OO$-algebras are called a lax operads in the Ph.D. thesis \cite{BrinkmeierThesis}.) On the other hand, a $W_0\OO$--algebra (or equivalently $B\OO$--algebra), is a homotopy operad where the composition is still only homotopy associative, but where the symmetric group action and unit are strict.

Theorem~\ref{ThBO} gives the relationship between the operad $B\OO\cong W_0\OO$ and the dendroidal category $\om$, showing a ``bracketed  version'' of the equivalence between $O$--algebras and appropriate $\om$--diagrams, i.e.~replacing $\OO$ and $\om$ by bracketed resolutions $B\OO$ and $\widetilde\om_0$. The operad $W\OO$ is a more complete resolution of $\OO$. 
For a category $\mathcal{K}$, there exists a resolution similar to the $W$--construction, namely the ``explosion'' $\widetilde{\mathcal{K}}$ of the category, as studied by Segal~\cite[Appendix B]{Segal74} and Leitch~\cite{Leitch}.
This ``explosion" has the property that $\widetilde{\mathcal{K}}$--diagrams are coherently homotopy $\mathcal{K}$--diagrams. 
Applying this construction to the category $\om$, one could expect that $W\OO$--algebras are related to $\widetilde\om$--diagrams in the same way that $B\OO=W\OO_0$--algebras are related to $\widetilde\om_0$--diagrams. 
We show in Theorem~\ref{thm:WO-Omega} that this does not quite hold, proving instead 
that there is an embedding of the category of $W\OO$-algebras as a full subcategory of the category of 
$\widetilde\om$--diagrams  %$X\in\mathcal{S}^{\widetilde\om^{op}}$ 
satisfying a strict Segal condition. 

\medskip

The results presented in this paper give a detailed infinity operad structure on normalized cacti. The input of the construction is a pre-given composition that we show to be associative up to coherent homotopy by using the operad $B\OO=W\OO/\sim$. The homotopies are constructed using the contractible space of basepoint preserving monotone reparametrizations of the circle (see the proof of Theorem~\ref{thm: cact is BO alg}). To extend the results to the cobordism category of graphs described above, one would need to replace $\OO$ by the operad $P\OO$, whose algebras are all symmetric properads~\cite[Section 14.1.2]{yj15}, define a resolution ``$BP\OO$'', as the appropriate quotient of the $W$-construction applied to $P\OO$. Our expectation is that these same reparametrisations of the circle will likewise provide all the necessary homotopies to provide an infinity composition in the cobordism category.

\addtocontents{toc}{\SkipTocEntry}
\subsection*{Acknowledgements} 
This work was done as part of the Women in Topology Workshop in August 2019, supported by
the Hausdorff Research Institute for Mathematics, NSF grant DMS 1901795, the AWM ADVANCE
grant NSF-HRD-1500481, and Foundation Compositio Mathematica. Additional work by L.B.B. and M.R. was carried out while in residence at MSRI in 2020. L.B.B. was supported by CNPq (201780/2017-8).   S.R. acknowledges the support of the Centre of Australian Category Theory and Australian Research Council grants DP160101519 and FT160100393. N.W. was supported by the Danish National Research Foundation through the Copenhagen Centre for Geometry and Topology (DNRF151) and the European Research Council (ERC) under the European
Union's Horizon 2020 research and innovation programme (grant agreement No. 772960). 
	
In addition, we would like to thank Philip Hackney, Gijs Heuts, Muriel Livernet, Claudia Scheimbauer and Bruno Vallette for helpful conversations, suggested references and comments.

\tableofcontents
\section{Preliminaries on Operads}

A \emph{symmetric sequence} in a symmetric monoidal category $\mathcal{S}$ is a collection $\mathcal{P}=\{\mathcal{P}(k)\}_{k\geq 0}$ of objects in $\mathcal{S}$ in which each $\mathcal{P}(k)$ comes equipped with an action of the symmetric group $\Sigma_k$. In this paper, our symmetric monoidal category $\mathcal{S}$ will either be the discrete category of sets, the category of simplicial sets, or the category of topological spaces with their standard Cartesian products.  

An \emph{operad} in $\mathcal{S}$ is a symmetric sequence $\mathcal{P}=\{\mathcal{P}(k)\}_{k\geq 0}$
 together with a distinguished element $\iota \in \mathcal P(1)$, called the \emph{unit}, and a collection of composition maps 
 \[\begin{tikzcd} \circ_{i}: \mathcal{P}(k)\times \mathcal{P}(j)\arrow[r] & 
 \mathcal{P}(k+j-1),\end{tikzcd}\]  $1\leq i \leq k$, which are associative, unital, and equivariant. For more complete details see, for example, \cite[Definition 11]{markl_operads_and_props}.  Given an operad $\mathcal{P}$, a symmetric sequence $\mathcal{Q}=\{\mathcal{Q}(k)\subseteq \mathcal{P}(k)\}_{k\geq 0}$ is a \emph{suboperad} of $\mathcal{P}$ if the restriction of the composition maps in $\mathcal{P}$ induce an operad structure on $\mathcal{Q}$.  A \emph{morphism of operads}  $f:\mathcal{P}\rightarrow\mathcal{Q}$ is a family of equivariant maps $$\{f(k): 
 \mathcal{P}(k)\rightarrow \mathcal{Q}(k)\}_{k\geq 0}$$ that are compatible with 
 composition and units. 

\begin{remark}\label{rem: simultaneous composition} It is equivalent to work with individual compositions $$\circ_i\colon \pp(k) \times \pp(j_i)\to \pp(k+j_i-1)$$ or with all $\circ_i$-compositions simultaneously. In the latter case, the simultaneous compositions are denoted by a map
\[\gamma_{\pp}\colon \pp(k)\times\pp(j_1)\times\ldots\times\pp(j_k)\to  \pp(\Sigma_{i=1}^{k}j_i).\] (eg:\cite[Proposition 13]{markl_operads_and_props}).

\end{remark}

More generally, we will use \textit{colored} operads. For any non-empty set $\mathfrak{C}$, a $\mathfrak{C}$-\emph{colored symmetric sequence} is a family of objects  $\mathcal{P}:=\{\mathcal{P}(c;c_1,\ldots,c_k)\}_{k\geq 0}$ in $\mathcal{S}$, where $(c;c_1,\dots,c_k)$ ranges over every list of colors in $\mathfrak{C}$ together with a map $\sigma^\star:\mathcal{P}(c;c_1,\ldots,c_k)\to \mathcal{P}(c;c_{\sigma(1)},\ldots,c_{\sigma(k)})$ for each $\sigma\in \Sigma_k$. A $\mathfrak{C}$-\emph{colored operad} is a $\mathfrak{C}$-colored symmetric sequence $\mathcal{P}$
%a $\mathfrak{C}$-\emph{coloured operad} consists of, for every list of colors $(c;c_1,\dots,c_n)$, a space of operations $\mathcal{P}(c;c_1,\dots,c_n)$, 
together with a family of partial composition maps \[ \circ_i \colon \mathcal{P}(c;c_1,\dots,c_k)\times \mathcal{P}(d;d_1,\dots,d_j) \to \mathcal{P}(c;c_1,\dots, c_{i-1},d_1,\dots,d_j, c_{i+1}\dots, c_k)\]
 defined only when $c_i=d$, together with an element $\iota_c \in \mathcal P(c;c)$ for each $c \in \mathfrak{C}$, which satisfies unit, equivariance and associativity conditions.  For more details see, for example, \cite[Definition 1.1]{bm_resolution}.  When the color set is $\mathfrak{C}=\{*\}$, a $\mathfrak{C}$--colored operad is a one-colored operad. In this paper  we will refer to both operads and colored operads as ``operads'', only mentioning the color set when necessary.

An \emph{algebra} over a ($\mathfrak{C}$--colored) operad $\mathcal{P}$ is a collection of objects $\{X(c)\}_{c\in \mathfrak{C}}$ in $\mathcal{S}$ together with evaluation maps 
$$\alpha\colon \mathcal{P}(c;c_1,\dots,c_k)\times X(c_1)\times\dots\times X(c_k) \longrightarrow X(c)$$
satisfying appropriate associativity, unit and equivariance conditions, see e.g.~\cite[Definition 1.2]{bm_resolution}. The category of $\pp$-algebras in $\mathcal{S}$ is denoted $\pp\mathrm{-Alg}_{\mathcal{S}}$. 

Our main example of a colored operad will be the $\mathbb{N}$-colored operad $\OO$, whose algebras are the (non-colored) operads, see Definition~\ref{def:operad of operads}. In Section~\ref{sec:cacti}, we will also make use of the following operad: 

\begin{example}\label{Example: CoEnd}
Let $X$ be a fixed space in $\mathcal{S}$. The \emph{coendomorphism operad} of $X$, $\CoEnd(X)$, has an underlying symmetric sequence with arity $k$ spaces \[\CoEnd(n)(X):=\Map(X,X^{\times k}).\] The symmetric groups act by permuting the factors of $f=(f_1,\ldots,f_k)\in\CoEnd(k)$. If $f=(f_1,\ldots,f_k)\in\CoEnd(k)(X)$ and $g=(g_1,\ldots, g_j)\in\CoEnd(j)(X)$ the partial compositions 
\[\begin{tikzcd}\circ_i:\CoEnd(k)(X)\times \CoEnd(j)(X)\arrow[r] & \CoEnd(k+j-1)(X) \end{tikzcd}\]
are given by 
\[f\circ_ig = (f_1\,,\,\ldots\,,\, f_{i-1}\,,\, g_1\circ f_i\,,\, \ldots\,,\, g_j\circ f_{i}\,,\, f_{i+1}\,,\, \ldots\,,\, f_k).\]

\end{example}

\medskip

\subsection{Trees}\label{sec:trees}
Throughout this paper, we use trees to model operad compositions and as the basis of our main constructions.
A graph $G$ is a tuple $(V(G), H(G), s, i)$ where $V(G)$ is a set of vertices,
$H(G)$ a set of half-edges, $s:H(G)\rightarrow V(G)$ is the source map and $i:H(G)\rightarrow H(G)$ is an involution. %Fixed points of the involution are called \emph{leaves}. 
Orbits of the involution $i$ are called \emph{edges} of $G$ and the set of edges is denoted by $E(G)$. An edge represented by a pair $\{h, i(h)\}$ with $i(h)\not=  h$ is called an \emph{internal edge}, and the set of internal edges is denoted $\ie(G)$. Edges corresponding to orbits of fixed points of the involution are \emph{external}.

A {\em tree} is a simply connected graph. 
All our trees will be \emph{rooted}, i.e. they come with a distinguished ``outgoing'' external edge called the \emph{root}. All other external edges are ``incoming'' and called \emph{leaves}. The set of leaves is denoted $L(T)$. The \emph{arity} of $T$ is the number of leaves $|L(T)|$. The root of the tree is denoted $R(T)$.

Note that a rooted tree can be canonically made into a directed graph by setting all the edges to point towards the root. Then note that the set of edges incident to a vertex always has a unique {\em outgoing edge}, the one closest to the root, and all other edges are \emph{incoming edges}. The number of incoming edges of a vertex $v$ is called the \emph{arity} of the vertex and denoted by $|v|$, with $|v|\ge 0$ any natural number. 

We allow the special tree $\eta=|$, with no vertices and a single edge. The trees with a single vertex and $n$ leaves are called \emph{$n$-corollas} and denoted $C_n$. A rooted tree $S$ is a \emph{subtree} of $T$ if $V(S)\subseteq V(T)$, $H(S)\subseteq H(T)$, and the structure maps for $S$ are restrictions of the structure maps for $T$, defining $i(h)=h$ in $S$ if $i(h)=h'$ in $T$ with $h'\notin H(S)$, so that for every  $v\in V(S)$, the arity of $v$ in $S$ and $T$ is the same.
A \emph{planar tree} is a rooted tree together with a preferred embedding into the plane. Note that for a planar tree, we get an induced canonical ordering of the incoming edges at each vertex.

\medskip

We use planar trees to model operad compositions via an operation called grafting.  Given trees $T$ and $T'$, of arity $n$ and $m$ respectively, and a leaf $i\in L(T)$, the {\em grafting} of $T'$ onto $T$ along the leaf $i$ is defined to be the tree $T \circ_i T'$ obtained by attaching the root of $T'$ to the leaf $i$ of $T$ so that they form a new internal edge in the grafted tree (Figure \ref{fig: grafting example}). Grafting of trees is also used to model the free operad generated by a symmetric sequence, as we will explain now. To avoid confusion later, when we will have to decorate vertices of trees by other trees, we will use blackboard fonts for the trees in the free operad construction (and later the associated $W$-construction in Section~\ref{sec: W-construction}), as we will soon apply this construction to a symmetric sequence of trees, which will give (blackboard) trees of (plain) trees. 

\begin{figure}[h]
	\centering
\subfigure[Example trees $T$ and $T'$]{
			\centering
			\begin{tikzpicture} [scale=1, inner sep=2mm]
			\begin{scope}[scale=0.5]
			    \coordinate (0) at (0,0);
			    \coordinate (r) at (0,-1);
			    \coordinate (1) at (-1, 1);
			    \coordinate (v) at (1, 1);
			    \coordinate (2) at (0, 2);
			    \coordinate (3) at (1, 2.4);
			    \coordinate (4) at (2, 2);
			    
			    \draw[line width=1pt] (r)--(0)--(1); 
			    \draw[line width=1pt ] (0)--(v)--(2);
			    \draw[line width=1pt ] (v)--(3);
			    \draw[line width=1pt ] (v)--(4);
			    
			    \foreach \i in {0, v} {\fill (\i) circle (4.5pt);}; 
			    
			    % Label the vertices
	            \node[right] at (v){$v$};
	            \node[below] at (1){$i$};
	            \node[left] at (-1, -1) {$T$};
			\end{scope}
			\begin{scope}[scale=0.5, xshift=5cm, yshift=-2cm]
			    \coordinate (0) at (0,0);
			    \coordinate (r) at (0,-1);
			    \coordinate (u) at (-1, 1);
			    \coordinate (7) at (1, 1);
			    \coordinate (5) at (-2, 2);
			    \coordinate (6) at (0, 2);
			    \coordinate (8) at (0, 1.2);
			    
			    \draw[line width=1pt ] (r)--(0)--(u)--(5); 
			    \draw[line width=1pt ] (0)--(7);
			    \draw[line width=1pt ] (u)--(6);
			    \draw[line width=1pt ] (0)--(8);

			    \foreach \i in {0, u, 6} {\fill (\i) circle (4.5pt);}; 
			     % Label the vertices
	            \node[left] at (-1, 0) {$T'$};
			\end{scope}
			\end{tikzpicture} 
			%\caption{}
		} \qquad \qquad \qquad
		\subfigure[The grafting $T \circ_i T'$]{
		%\begin{minipage}{0.3\textwidth}
			\centering
			\begin{tikzpicture} [scale=0.6, inner sep=2mm]
			    \coordinate (0) at (0,0);
			    \coordinate (r) at (0,-1);
			    \coordinate (w) at (-1, 1);
			    \coordinate (v) at (1.2, 1);
			    \coordinate (2) at (0.2, 2);
			    \coordinate (3) at (1.2, 2.4);
			    \coordinate (4) at (2.2, 2);
			    \coordinate (u) at (-2, 2);
			    \coordinate (7) at (-0.1, 1.9);
			    \coordinate (5) at (-3, 3);
			    \coordinate (6) at (-1, 3);
			    \coordinate (8) at (-1, 2.2); 
			    
			    \draw[line width=1pt] (r)--(0)--(w); 
			    \draw[line width=1pt] (0)--(v)--(2);
			    \draw[line width=1pt] (v)--(3);
			    \draw[line width=1pt] (v)--(4);
			    \draw[line width=1pt] (w)--(u)--(5); 
			    \draw[line width=1pt] (w)--(7);
			    \draw[line width=1pt] (u)--(6);
			    \draw[line width=1pt] (w)--(8);
			
			    \foreach \i in {0, v, w, u, 6} {\fill (\i) circle (4pt);}; 

    		    % Label the vertices
	            \node[right] at (v){$v$};
			%\draw[white] (top)--(bottom);	

			\end{tikzpicture} 
			%\caption{The grafting $T_1 \circ_i T_2$}
		%\end{minipage} \qquad\qua%	
		}
		\vspace{-1em}\caption{Grafting of trees.}
		\label{fig: grafting example}
\end{figure}
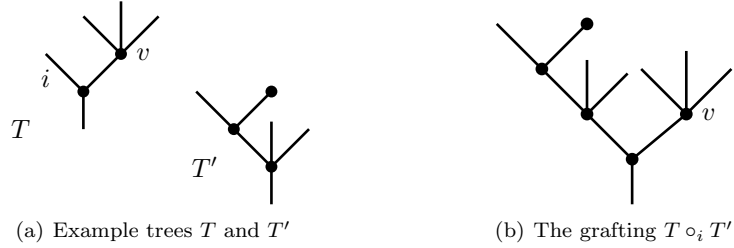

\begin{definition}\label{def: free operad} 
Let $\pp=\{\pp(c;c_1,\dots,c_k)\}_{c_i,c\in \mathfrak{C}}$ be a $\mathfrak{C}$-colored symmetric sequence in $\mathcal{S}$. 
A planar tree $\mathbb{T}$ is $\mathfrak{C}$-colored if it is equipped with a map $f:E(\mathbb{T})\rightarrow \mathfrak{C}$, we refer to $f(e)$ as the color of the edge $e$. A $\mathfrak{C}$-colored planar tree $\mathbb{T}$ is \emph{decorated} by $\pp$ if each vertex $v\in V(\mathbb{T})$ is labeled by an operation in $p_v\in \pp(out(v);in(v))$, where $out(v)$ is the color of the outgoing edge of $v$, and $in(v)$ is the list of colors of the incoming edges of $v$, ordered by the planar structure.
The \emph{free operad} $F(\pp)$ on $\pp$ is the $\mathfrak{C}$-colored operad whose $k$-ary operations are the $\mathfrak{C}$-colored, $\pp$ decorated, planar trees $\mathbb{T}$ of arity $k$ with leaves labeled by a bijection $\lambda:\{1,\dots,k\}\rightarrow L(\mathbb{T})$.

Explicitly, for each $c,c_1,\ldots,c_k\in\mathfrak{C}$,  
$$F(\pp)(c;c_1,\ldots,c_k):=\Big(\coprod\limits_{(\mathbb{T},f,\lambda)}\prod_{v\in V(\mathbb{T})}\pp(out(v);in(v))\Big)/\sim,$$ where $(\mathbb{T},f,\lambda)$ 
runs over all isomorphism classes of leaf-labeled $\mathfrak{C}$-colored planar trees with $k$ leaves such that $f(\lambda(i))=c_i$, $f(R(\mathbb{T}))=c$, and where the equivalence relation is generated by the following:

\begin{itemize}
    \item[($\ast$)]\label{free_operad_relation} two labeled trees $(\mathbb{T},f,\lambda,(p_v)_{v\in V(\mathbb{T})})$ and $(\mathbb{T}',f',\lambda',(p'_w)_{w\in V(\mathbb{T}')})$ are equivalent if there exists a non-planar isomorphism $\alpha:\mathbb{T}\to \mathbb{T}'$ such that 
    $f\circ \alpha=f'$, $\alpha\circ \lambda=\lambda'$, and $\sigma_v(\alpha)p_v=p_{\alpha(v)}$, for $\sigma_v(\alpha)$ the permutation on $in(v)$ induced by $\alpha$. 
\end{itemize}

The symmetric group acts on $F(\pp)$ by permuting the labels of the leaves, acting on $\lambda$, and composition in $F(\pp)$ is given by grafting of trees, with $\circ_i$ grafting at the leaf $\lambda(i)$. 
For full details see, for example, the construction under Corollary 3.3 \cite{bm_resolution}.
 
\end{definition}

We now employ the free operad construction to define a class of free operads $\om(T)$ generated by a planar tree $T$. 
This will play a fundamental role in the definition of the dendroidal category (Section~\ref{sec:omega}), which describes a model for $\infty$-operads.

\begin{example}\label{example: Omega(T)}
A planar tree $T$ generates a free colored operad $\Omega(T)$ as follows. 
The set of colors of $\Omega(T)$ is the set of edges $\mathfrak{C}=E(T)$. 
We define a discrete $E(T)$-coloured symmetric sequence $X(T)$ by
\[ X(T)(e;e_{\sigma(1)}, \dots, e_{\sigma(n)}) = \begin{cases}
     \{\sigma v\}& \text{if $(e;e_1, \dots, e_n) =(\text{out}(v); \text{in}(v))$, \ \text{for}\ $v \in V(T)$}, \\
    \emptyset & \text{otherwise,} 
\end{cases} \]with its built in free symmetric group action.  
Then $\Omega(T):=FX(T)$ is the free operad on the collection $X(T)$. 
Explicitly, $\Omega(T)$ is an $E(T)$-coloured operad with 
\[ \Omega(T)(e;e_{\sigma(1)}, \dots, e_{\sigma(n)}) =\begin{cases}
     \sigma S & \text{ if $(e;e_1, \dots, e_n)=(R(S);L(S))$, $S \subset T$}, \\
    \emptyset & \text{ otherwise.} 
\end{cases} \] for $S\subset T$  a subtree of $T$. Composition, as in the free operad, is given by grafting of subtrees. 
For further details, see Section 2.2 and just above Definition 2.3.1 in \cite{Moerdijk_Lecture_Notes}.
\end{example}

\subsection{The dendroidal category $\Omega$}\label{sec:omega}

The model we use for $\infty$-operads is that of dendroidal Segal spaces that satisfy the \textit{weak Segal condition}. Dendroidal spaces are diagrams of the dendroidal category. 

The dendroidal category $\Omega$ is the full subcategory of colored operads whose objects are the free operads $\Omega(T)$ generated by trees (as in Example~\ref{example: Omega(T)}). 
In other words, objects of $\om$ are planar isomorphism classes of planar rooted trees and morphisms in $\om$ are defined to be operad maps \[ \Hom_{\om}(S,T)=\Hom_{\mathsf{Op}}(\om(S),\om(T)).\] 
Morphisms in $\Omega$ can be described as a composition of four types of elementary morphisms: isomorphisms, degeneracies, inner and outer face maps. 
In terms of trees, {\em isomorphisms} are non-planar tree isomorphisms,  \emph{inner face maps} are of the form $\partial_{e}:T/e\rightarrow T$, where $T/e$ is the tree obtained from $T$ by contracting an inner edge $e\in \ie(T)$.  If $v$ is a vertex of $T$ with only one inner edge attached to it then $T/v$ is the tree obtained from $T$ by chopping off the vertex $v$ and the inclusion $\partial_{v}:T/v\rightarrow T$ is an \emph{outer face map}. A \emph{degeneracy} is a map $s_v:T/v\rightarrow T$ where $T/v$ is obtained from $T$ by deleting a vertex $v$, with $|v|=1$, in $T$.

In the opposite category $\Omega^{op}$, outer face maps correspond to restriction to certain allowed subtrees, while inner face maps correspond to edge collapses. 
For more details and plenty of examples see \cite{mw07,Moerdijk_Lecture_Notes}.

\begin{remark}
Our definition of $\om$ differs slightly from the usual definition in that we have chosen our objects to be planar trees. Technically, what we have described here is the equivalent category $\om'$ from \cite[2.3.2]{Moerdijk_Lecture_Notes}. 
\end{remark}

\begin{definition}\label{def: dendroidal space}
A \emph{dendroidal space} $X$ is an $\om$-diagram $X:\Omega^{op}\rightarrow\mathcal{S}$, where $\mathcal{S}$ is either the category of simplicial sets or topological spaces. 
\end{definition} 
The evaluation of $X$ at a tree $T$ is denoted $X(T)$. A dendroidal space is called \emph{reduced} if $X(\eta)\simeq*$, where $\eta=|$. We will write $\mathcal{S}^{\om^{op}}$ for the category of dendroidal spaces. \footnote{In the literature a dendroidal space is usually called reduced if $X(\eta)=*$ but we vary this slightly and say that a dendroidal space is reduced if $X(\eta)$ is contractible as in \cite[Definition 4.1]{BHR}.}

\medskip

For any vertex $v$ in a tree $T\in\om$, we have an associated outer face map in $\om$
$$C_v \longrightarrow T$$
taking the unique vertex of the corolla to $v\in V(T)$, where $C_v$ is the corolla with $|v|$ leaves. Likewise, for any internal edge between vertices $u$ and $v$ in $T$, there is a commuting diagram in $\om$
$$\xymatrix{\eta \ar[r] \ar[d] & C_u \ar[d] \\
C_v \ar[r] & T.}$$
Let ${\rm Sk}_1(T)$ be the category whose objects are the edges and vertices of $T$, thought of copies of $\eta$ and corollas $C_v$, and whose morphisms are associated to edge inclusions in $T$, as in the top left corner of the above diagram. 

For a dendroidal space $X$, the \emph{Segal map} 
is the unique map from $X(T)$ to the limit $\lim_{{\rm Sk}_1(T)^{op}}X$ induced by the corolla inclusions.  When $X(\eta)=*$, this limit becomes a product over the value of $X$ at the corollas, and the Segal map becomes the map 
\begin{equation*}%\label{eq: Segal} 
\begin{tikzcd} \chi\colon X(T)\arrow[r] & \prod_{v\in V(T)}X(C_v)\end{tikzcd}%\tag{$\star$}
\end{equation*}
with components the restriction to the value of $X$ at each corolla. 

The category of $\om$-diagrams admits two Quillen model category structures: the Reedy model structure and the projective model structure which are Quillen equivalent (eg:\cite[Remark 2.5]{BdBM}). Throughout, we take the \emph{projective model structure} in which a morphism of $\om$-diagrams is a weak equivalence or fibration if it is entrywise a weak equivalence or fibration. 

\begin{definition}\label{def: segal condition}
A dendroidal space $X\in\mathcal{S}^{\om^{op}}$ %with $X(\eta)\simeq*$ 
satisfies a \emph{strict Segal condition} if the Segal map is an isomorphism for each $\eta\not=T\in\Omega$. If $X$ is fibrant and the map $\chi$ is only a homotopy equivalence for each $\eta\not=T\in\om$ then we say that $X$ satisfies a \emph{weak} Segal condition. 
\end{definition}

\begin{remark}\label{remark: Segal map}
We briefly comment that in the original definition in \cite[Definition 8.1]{cm2} a dendroidal spaces satisfies the weak Segal condition if the Segal map is a trivial fibration. Our assumption that $X$ is fibrant allows us to only require that the Segal map is a weak equivalence as in \cite[Definition 3.1]{BdBM} or \cite[Definition 4.1]{BHR}.
\end{remark}

\subsection{The operad of operads}
One of the main constructions in this paper is the operad $B\OO$. This operad builds on an $\mathbb{N}$-colored operad $\OO$ called the \textit{operad of operads}, whose algebras are one-colored operads.

Let $T$ be a planar tree. For a vertex $v\in V(T)$ with arity $|v|=m$ and a planar tree $T'$ with $m$ leaves, the \emph{substitution} $T\bullet_{v} T'$ is obtained by removing the vertex $v$ from $T$ and identifying the incoming and outgoing edges of $v$ with the leaves and root of $T'$, respectively. 
An example is shown in Figure~\ref{fig: substitution example}.

\begin{figure}[h]
	\centering 
\subfigure[Example trees $T$ and $T'$]{
%	\begin{minipage}{0.27\textwidth}
			\centering
			\begin{tikzpicture} [scale=1, inner sep=2mm]
			\begin{scope}[scale=0.5]
			    \coordinate (0) at (0,0);
			    \coordinate (r) at (0,-1);
			    \coordinate (1) at (-1, 1);
			    \coordinate (v) at (1, 1);
			    \coordinate (2) at (0, 2);
			    \coordinate (3) at (1, 2.4);
			    \coordinate (4) at (2, 2);
			    
			    \draw[line width=1pt ] (r)--(0)--(1); 
			    \draw[line width=1pt ] (0)--(v)--(2);
			    \draw[line width=1pt ] (v)--(3);
			    \draw[line width=1pt ] (v)--(4);
			    
			    \foreach \i in {0, v} {\fill (\i) circle (4.5pt);}; 
			    
			    % Label the vertices
	            \node[right] at (v){$v$};
	            \node[below] at (1){$i$};
	            \node[left] at (-1, -1) {$T$};
			\end{scope}
			\begin{scope}[scale=0.5, xshift=5cm, yshift=-2cm]
			    \coordinate (0) at (0,0);
			    \coordinate (r) at (0,-1);
			    \coordinate (u) at (-1, 1);
			    \coordinate (7) at (1, 1);
			    \coordinate (5) at (-2, 2);
			    \coordinate (6) at (0, 2);
			    \coordinate (8) at (0, 1.2);
			    
			    \draw[line width=1pt ] (r)--(0)--(u)--(5); 
			    \draw[line width=1pt ] (0)--(7);
			    \draw[line width=1pt ] (u)--(6);
			    \draw[line width=1pt ] (0)--(8);

			    \foreach \i in {0, u, 6} {\fill (\i) circle (4.5pt);}; 
			     % Label the vertices
	            \node[left] at (-1, 0) {$T'$};
			\end{scope}
			\end{tikzpicture} 
			%\caption{}
		} \qquad \qquad \qquad
		\subfigure[The substitution $T \bullet_v T'$]{
			\centering
			\begin{tikzpicture} [scale=0.6, inner sep=2mm]
			    \coordinate (0) at (0,0);
			    \coordinate (r) at (0,-1);
			    \coordinate (1) at (-1, 1);
			    \coordinate (v) at (1, 1);
			    \coordinate (2) at (0, 2);
			    \coordinate (3) at (1, 2.4);
			    \coordinate (4) at (2, 2);
			    \coordinate (u) at (0, 2);
			    \coordinate (7) at (2, 2);
			    \coordinate (5) at (-1, 3);
			    \coordinate (6) at (1, 3);
			    \coordinate (8) at (1, 2.2);
			    
			    \draw[white] (-3,0)--(3,0);
			    
			    \draw[line width=1pt] (r)--(0)--(1); 
			    \draw[ line width=1pt] (0)--(v)--(u)--(5);
			    \draw[ line width=1pt] (v)--(7);
			    \draw[ line width=1pt] (u)--(6);
			    \draw[ line width=1pt] (v)--(8);

			    \foreach \i in {0, v, u, 6} {\fill (\i) circle (4pt);}; 

			    % Label the vertices
	            \node[below] at (1){$i$};
			\end{tikzpicture} 
			} %\hfill \hfill
					%\end{minipage}\qquad \quad
% 		\subfigure[The grafting $T \circ_i T'$]{
% 		%\begin{minipage}{0.3\textwidth}
% 			\centering
% 			\begin{tikzpicture} [scale=0.6, inner sep=2mm]
% 			    \coordinate (0) at (0,0);
% 			    \coordinate (r) at (0,-1);
% 			    \coordinate (w) at (-1, 1);
% 			    \coordinate (v) at (1.2, 1);
% 			    \coordinate (2) at (0.2, 2);
% 			    \coordinate (3) at (1.2, 2.4);
% 			    \coordinate (4) at (2.2, 2);
% 			    \coordinate (u) at (-2, 2);
% 			    \coordinate (7) at (-0.1, 1.9);
% 			    \coordinate (5) at (-3, 3);
% 			    \coordinate (6) at (-1, 3);
% 			    \coordinate (8) at (-1, 2.2); 
			    
% 			    \draw (r)--(0)--(w); 
% 			    \draw (0)--(v)--(2);
% 			    \draw (v)--(3);
% 			    \draw (v)--(4);
% 			    \draw (w)--(u)--(5); 
% 			    \draw (w)--(7);
% 			    \draw (u)--(6);
% 			    \draw (w)--(8);
			
% 			    \foreach \i in {0, v, w, u, 6} {\fill (\i) circle (3pt);}; 

%     		    % Label the vertices
% 	            \node[right] at (v){$v$};
% 			%\draw[white] (top)--(bottom);	

% 			\end{tikzpicture} 
% 			%\caption{The grafting $T_1 \circ_i T_2$}
% 		%\end{minipage} \qquad\qua%	
% 		}
		\vspace{-1em}\caption{Tree substitution (Compare with grafting in Figure \ref{fig: grafting example}).}
		\label{fig: substitution example}
\end{figure}
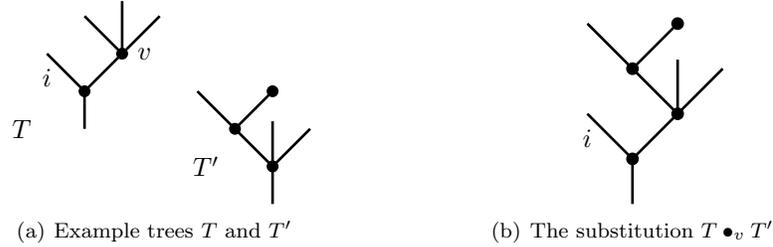

A \emph{labelled} planar tree is a triple $(T,\sigma,\tau)$, consisting of a planar tree $T$ equipped with bijections $\sigma:|V(T)|\rightarrow V(T)$ and $\tau:|L(T)|\rightarrow L(T)$. Two such triples $(T,\sigma,\tau)$ and $(T',\sigma',\tau')$ are isomorphic if there is a planar tree isomorphism $T\to T'$ that respects the labelling $\sigma, \tau$. We represent a labelled planar tree $(T,\sigma,\tau)$ by writing above each leaf $\ell\in L(T)$ the number $\tau^{-1}(\ell)$, and writing by each vertex $v\in V(T)$ the number $\sigma^{-1}(v)$, as depicted in Figure~\ref{fig:element of O}.
\begin{figure}[ht]
    \centering\def\svgwidth{0.2\columnwidth}
    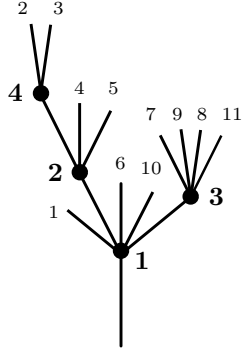
    \caption{Example of a labelled planar tree in $\OO(11;5,3,4,2)$.}\label{fig:element of O}
\end{figure}

We also define a tree substitution that is \emph{compatible with the labellings of the leaves}. Let $(T,\sigma,\tau)$ and $(T',\sigma',\tau')$ be two planar labelled trees with $|V(T)|=k$, $|V(T')|=l$ and $|L(T')|=|\sigma(i)|=m_i$.  
The map $\tau'$ encodes a permutation in the symmetric group with $m_i$ elements. We obtain a new planar tree $(\tau'_{\sigma (i)}) T$ by applying the permutation $\tau'$ on the $m_i$ incoming edges of the vertex $\sigma(i)\in V(T)$. 
We then define 
\begin{equation}\label{eq: substitution compatible with labelling}
    T \bullet_{\sigma(i),\tau'}T' = (\tau'_{\sigma i}) T \bullet_{\sigma (i)} T'.
\end{equation}

In particular, $V(T \bullet_{\sigma(i),\tau'}T') = \{V(T) - \sigma(i)\} \amalg V(T')$.
The labelling on the vertices of $T \bullet_{\sigma(i),\tau'}T'$ is given by the map $\sigma\circ_i\sigma'$, which is the induced bijection $\{1, \dots, k+l - 1\} \to V(T \bullet_{\sigma(i),\tau'}T')$
\[ j \mapsto \left \{ \begin{array}{ll}
  \sigma(j)  & 1 \leq j <i \\
    \sigma' (j- i +1) & i \leq j \leq i+ l\\
    \sigma(j -l + 1) & i+ l < j \leq k+l - 1.
\end{array}\right. \]
An example is shown in Figure~\ref{fig:Ocomposition}. 
\begin{figure}[ht]
    \centering\def\svgwidth{0.7\columnwidth}
    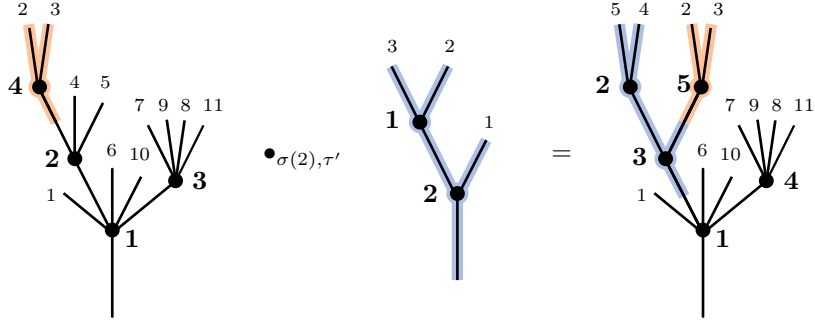
    \caption{Example of composition in $\OO$ where $\tau'$ is different to the planar order of $in(\sigma(2))$.}\label{fig:Ocomposition}
\end{figure}
In the case where the order induced by $\tau'$  on the $m_i$ incoming edges of $\sigma(i)$ is the same as the order induced by the planar structure, then $T \bullet_{\sigma(i),\tau'}T'=T \bullet_{\sigma(i)}T'$.

\begin{definition}\label{def:operad of operads} The \emph{operad of operads $\mathcal{O}$} is the $\mathbb{N}$--colored operad, for which \[\mathcal{O}(n;m_1,\dots,m_k)\] is the discrete space whose elements are isomorphism classes of \emph{labelled} planar rooted trees $(T,\sigma,\tau)$ where $T$ is a planar tree with $k$ vertices and $n$ leaves, with bijections $\sigma\colon |V(T)|\to V(T)$, $\tau\colon |L(T)|\to L(T)$, such that the vertex $\sigma(i)$ has arity $m_i$ for each $1 \leq i \leq k$. The composition operation  
% \[ \mathcal{O}(n;m_1,\dots,m_k)\times \mathcal{O}(m_i;b_1,\dots,b_l) \xrightarrow{\circ_i} \mathcal{O}(n;m_1,\dots,b_1,\dots,b_l,\dots,m_k)  \]
    \[\begin{tikzcd}[row sep=tiny]
        \mathcal{O}(n;m_1,\dots,m_k)\times \mathcal{O}(m_i;b_1,\dots,b_l) \arrow[r, "\circ_i"] & \mathcal{O}(n;m_1,\dots,b_1,\dots,b_l,\dots,m_k) \\
        ((T,\sigma,\tau),(T',\sigma',\tau')) \arrow[r, mapsto] & (T,\sigma,\tau)\circ_i (T',\sigma',\tau')
    \end{tikzcd}\]
is induced by tree substitution that is compatible with the labelling as in \eqref{eq: substitution compatible with labelling}, where 
$$(T,\sigma,\tau)\circ_i (T',\sigma',\tau')=(T\bullet_{\sigma(i),\tau'}T',\sigma\circ _i\sigma',\tau).$$
The unit for this composition, for the color $n$, is the element of $\OO(n;n)$ represented by the corolla $C_n$ equipped with the canonical left-right labelling. The symmetric group $\Sigma_k$ acts on $(T,\sigma,\tau)\in \mathcal{O}(n;m_1,\dots,m_k)$ by precomposition on the labelling $\sigma$ of the vertices $V(T)$.

\end{definition}

We further observe that, for each $m,n \in \N$, 
\[\OO(m;n)\cong \left \{ \begin{array}{ll}
    \Sigma_n & \text{ for } m = n,  \\
    \emptyset & \text{ when } m \neq n.
\end{array} \right . \]

The isomorphism $\OO(n;n) \cong \Sigma_n$ corresponds to labelling the leaves of a corolla $C_n$ in all possible ways. 
The unique arity $0$ operation in $\OO$ is represented by the special tree $\eta\in \OO(1;\emptyset)$.
An $\OO$-algebra, $\pp$, is precisely a one-colored operad. 
That is to say, %if $\pp$ is an $\OO$-algebra, then 
$\pp$ has an underlying $\mathbb{N}$-graded object $\pp=\{\pp(n)\}_{n\in\mathbb{N}}$ in  $\mathcal{S}$. Moreover, $\pp$ admits actions
 $\OO(n;n)\times \pp(n)\rightarrow \pp(n)$ for all $n$ and thus $\pp$ has an underlying symmetric sequence. By definition, we have \[\OO(n;m_1,\ldots,m_k)\times \pp(m_1)\times\ldots\times \pp(m_k) \subset F\pp(n),\] where $F\pp$ is the free operad on the symmetric sequence $\pp$, and 
 \[ F\pp(n) = \coprod_{k \in \mathbb{N}} \left (\coprod_{(m_1, \dots, m_k) \in \mathbb{N}^k} \OO(n;m_1,\ldots,m_k)\times \pp(m_1)\times\ldots\times \pp(m_k)\right )_{\Sigma_k}\] so the action maps  \[\alpha:\OO(n;m_1,\ldots,m_k)\times \pp(m_1)\times\ldots\times \pp(m_k) \rightarrow \pp(n)\]
 induce maps $F\pp(n) \to \pp(n)$ for all $n$, and by the algebra axioms, this is precisely the data of a symmetric operad in $\mathcal{S}$ (See \cite[Example 1.5.6]{bm_resolution}). Note that, in particular, the $\circ_i$-compositions of an operad $\pp$ are governed by the trees with one internal edge in $\OO(n;m_1,m_2)$, where $n=m_1+m_2-1$.

\subsection{The relationship between operads and dendroidal spaces}\label{operads and dendroidal spaces}Reduced dendroidal spaces that satisfy a \emph{strict} Segal condition are closely related to one-colored operads. Explicitly, every operad $\pp$ can be viewed as a dendroidal space via the dendroidal \emph{nerve} construction that defines a functor \[N^d(\pp)(T)=\Hom_{\mathsf{Op}}(\om(T),\pp)\]as $T$ ranges over $\om$. The nerve of the free operad $\om(T)$ is just the representable dendroidal space $\om[T]:= \Hom_{\om}(-,T)$. A dendroidal space $X$ is the \emph{nerve} of an operad if, and only if, the Segal map of Definition~\ref{def: segal condition} is an isomorphism for all $T$ \cite[Lemma 6.4; Proposition 6.5]{cm3}. To put this altogether, there is an isomorphism of categories \[\OO\mathrm{-Alg}_{\mathcal{S}}\cong (\mathcal{S}^{\om^{op}})_{strict}\] where $\OO$ is the colored operad whose algebras are one-colored operads  (Definition~\ref{def:operad of operads}, below) and $(\mathcal{S}^{\om^{op}})_{strict}$ denotes the category of reduced dendroidal spaces satisfying the strict Segal condition. We will prove similar statements for ``thickened'' versions of $\om$ in Theorem~\ref{thm:W_0O-Omega_0} and Theorem~\ref{thm:WO-Omega}.

\section{The operad of brackets $\mathcal{BO}$}\label{sec:BO}
In this section we introduce a new topological operad called the \emph{operad of bracketed trees}. In short, the operad $B\OO$ captures a weak notion of an operad in the sense that a $B\OO$-algebra is a symmetric sequence with $\circ_i$-operations that are only associative up to higher homotopy.  The construction of the operad $B\OO$ allows one to check with relative ease whether a symmetric sequence with compositions assembles into an $\infty$-operad. In Theorem~\ref{thm: cact is BO alg}, we use this to show that normalized cacti admit such a structure. Moreover, we expect that this construction provides a general method that one can use to construct other examples of $\infty$-operads. 

One could instead use the classical Boardman-Vogt $W$-construction on the operad $\OO$ to obtain an operad $W\OO$ whose algebras are homotopy operads ({\em lax operads} in the language of  \cite{BrinkmeierThesis}). It is known to experts that bracketings in trees are related to this operad $W\OO$, but the precise details are difficult to find in the literature. (However, see \cite[Section ~2.3]{Obr19}, in particular Theorem 4, together with Remark~\ref{remark: polytopes} below, for an algebraic version of this in the case of non-symmetric operads.) In Appendix~\ref{sec: BO and WO} we will show that $B\OO$ identifies with a quotient of the operad $W\OO$. Bracketings in trees have also appeared elsewhere, see eg. \cite{devadoss2009realization, DFRS15}, and the parenthesizations of~\cite[2.6]{Sinha04}.

\subsection{Bracketings of trees}\label{sec:bracketings}

We define in this section the poset of bracketings of a tree, starting with the definition of a bracketing: 

\begin{definition}\label{def:bracketing}
A tree is called \emph{large} if it has at least two vertices (or equivalently, at least one internal edge). A set $\{S_{j}\}_{j\in J}$ of subtrees of a tree $T$ is \emph{nested} if, for any $i,j\in J$, the set of common vertices $V(S_{i})\cap V(S_{j})$ is either $V(S_{i})$, $V(S_{j})$ or empty.
A \emph{bracketing} $B$ of a tree $T$ is a (possibly empty) collection $B=\{S_{j}\}_{j\in J}$ of nested large proper subtrees of $T$. 
\end{definition}

Recall from Section~\ref{sec:trees} that a subtree of $T$ is a tree $S$ whose vertices are a subset of the vertices of $T$, and whose half-edges are all the half-edges in $T$ attached to such vertices. Therefore, a subtree is completely determined by its vertices. With this in mind, we will represent bracketings as in Figure~\ref{fig:three_brackets}.

\begin{figure}[ht]
\centering
\includegraphics[width=0.25\textwidth]{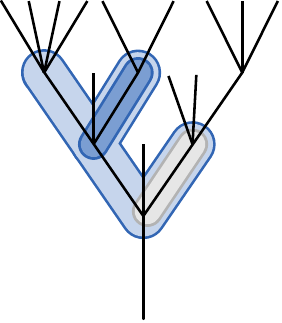}
\caption{Example of a tree bracketing with 3 nested subtrees.}\label{fig:three_brackets}
\end{figure}

\medskip

\begin{definition}\label{def: poset of brackets}
Bracketings of a tree $T$ form a poset of bracketings $\mathcal{B}(T)$ with the relation $B'\leq B$ if $B'\subseteq B$.
\end{definition}
We denote the geometric realisation of the nerve of the poset $\mathcal{B}(T)$ by $|\mathcal{B}(T)|$. A point in 
$$|\mathcal{B}(T)|=\coprod_{r\ge 0}N_r\mathcal{B}(T)\times \Delta^r/\sim$$ is a pair $(B,t)$ with $B=B_0\subset \dots\subset B_r$ a sequence of bracketings and $t\in \Delta^r$. Such a pair $(B,t)$
can be interpreted as a {\em weighted bracketing} with underlying set of brackets $B_r=\cup_{i=0}^r B_i$ and weights given by 
$$t=(1,t_1,\dots,t_r)\in \Delta^r=\{1=t_0\ge t_1\ge \dots\ge t_r\ge 0\}$$
where we assign the weight $t_0=1$ to all brackets in $B_0$, and for each $1\le i\le r$, the weight $t_i$ to all brackets in $B_i\backslash B_{i-1}$. 
In particular, a weighted bracketing with all brackets having weight 1 corresponds to a vertex $B=B_0$ in the nerve of the poset. Also, the equivalence relation on the realization implies that a bracket of weight 0 can be discarded. 
(See also  Appendix~\ref{sec: BO and WO} and in particular the proof of Lemma~\ref{-cubical} where this point of view is used to relate $B\OO$ to the operad $W\OO$.)

\begin{example}\label{example: corolla}
     If $T=C_n$ then $T$ does not admit any large subtree, therefore $\mathcal{B}(T)=\{\emptyset\}$ only has the empty (or trivial) bracketing.
\end{example}

\begin{example}\label{example: pentagon}
Let $T$ be the tree $\pentatree$, then the space $|\mathcal{B}(T)|$ is depicted in Figure \ref{fig:pentagon-hexagon} (left). Note that the initial object in the poset is the empty bracket, in the centre of the pentagon.

     \begin{figure}[h!t]
        \centering
        \begin{minipage}{0.49\textwidth}
            \centering\includegraphics[width=0.9\linewidth]{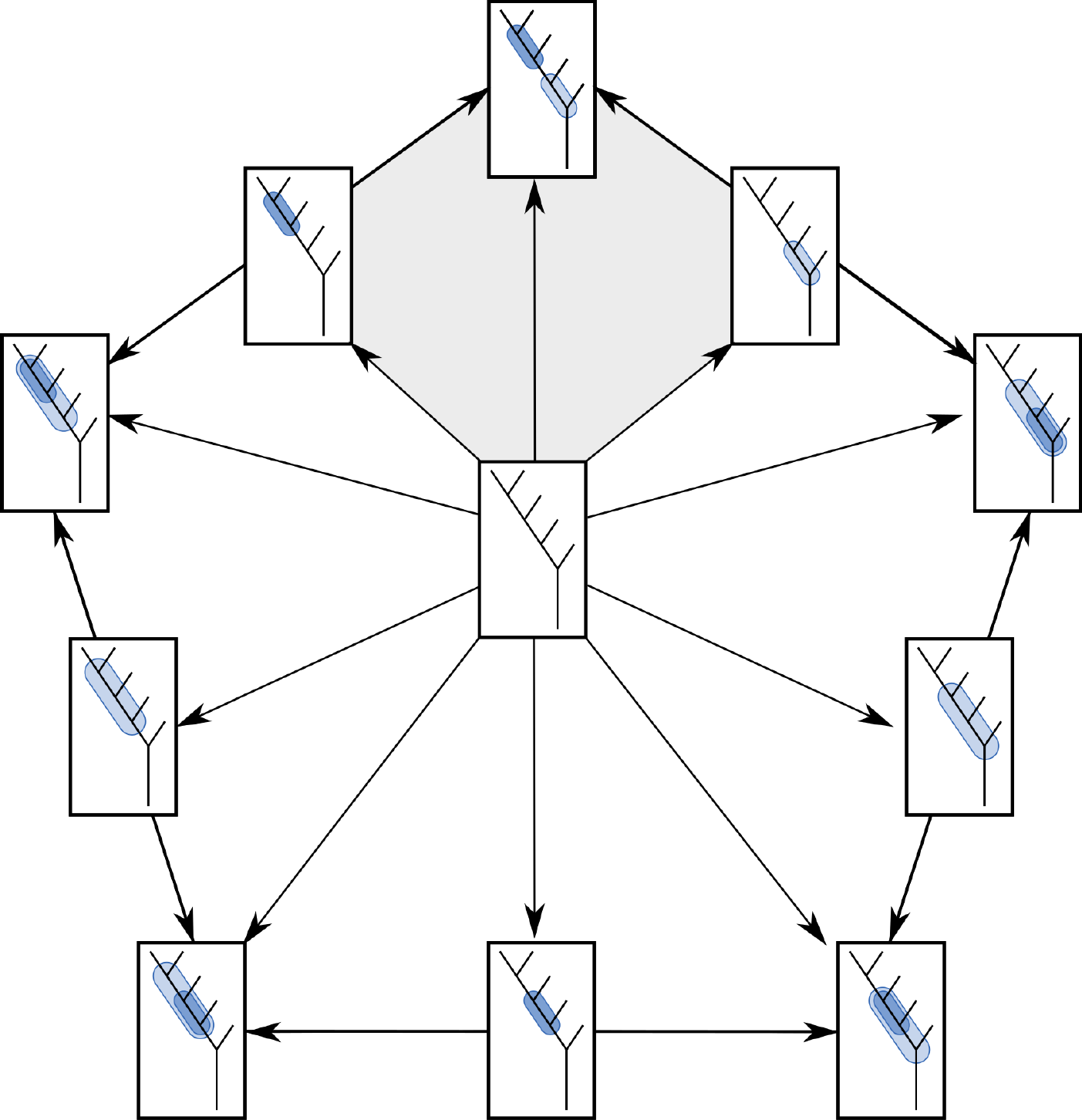}
        \end{minipage}
        \begin{minipage}{0.48\textwidth}
        \centering\includegraphics[width=0.9\linewidth]{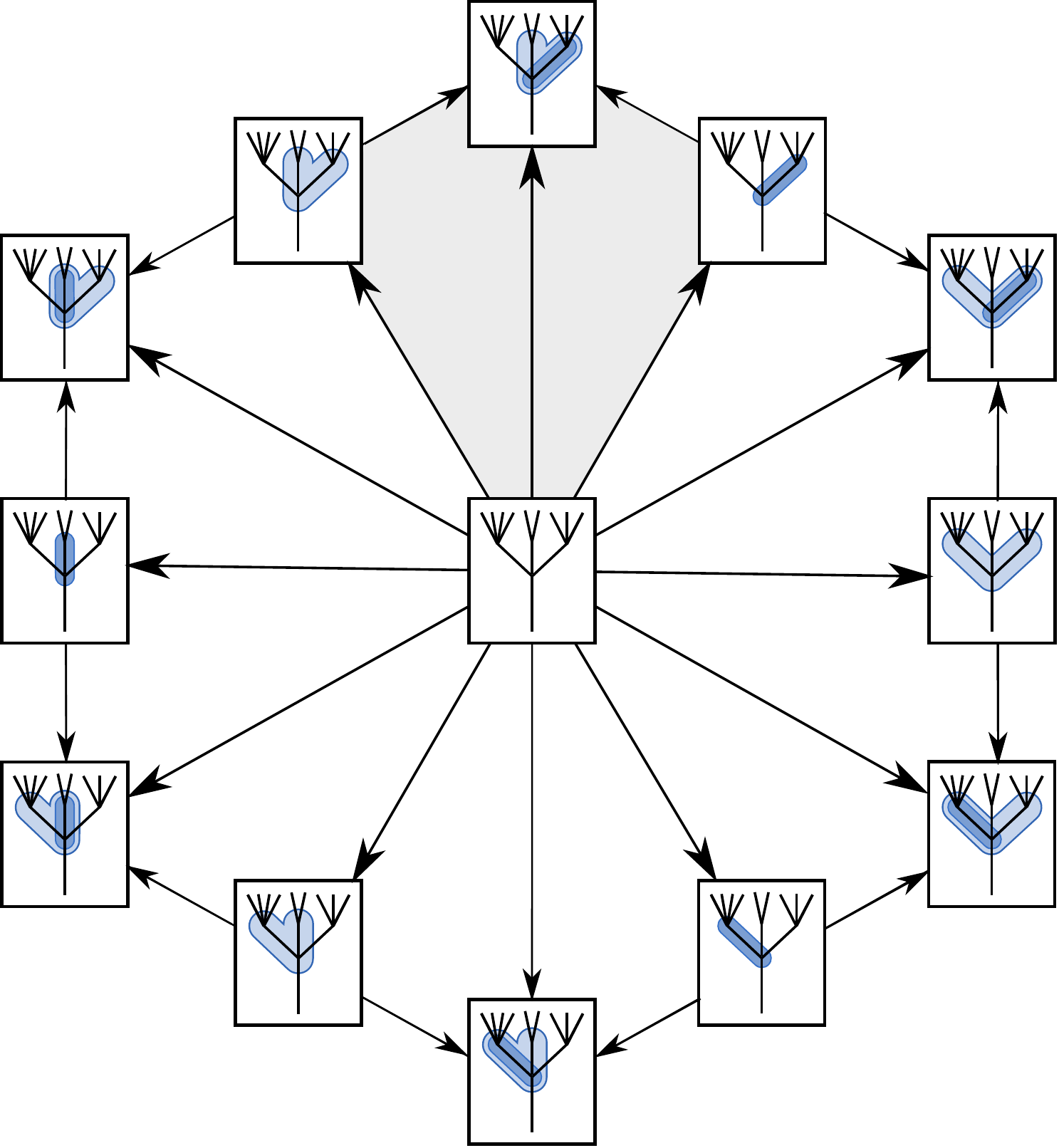}
        \end{minipage}
        \caption{Geometric realization of the poset $\mathcal{B}(T)$ of Examples \ref{example: pentagon} and \ref{example: hexagon-trivalent}.}\label{fig:pentagon-hexagon}
        \end{figure}

    More generally, let $T_n$ be a tree with $n$ vertices such that no vertex is connected to more than two inner edges. For such trees, the set of vertices can always be given a total ordering, for instance by constructing a list starting with a vertex $v$ connected to only one internal edge, and defining the next element of the list to be the vertex sharing an edge with $v$ that has not yet been listed. Then a bracket of $T_n$ can be immediately identified with a meaningful placement of parentheses on a word with $n$ letters where the word is represented by the ordered set of vertices. Therefore, $|\mathcal{B}(T_n)|$ can always be identified with the $n$-th associahedron (see also Remark~\ref{remark: polytopes} for another approach to this statement).
\end{example}

\begin{example}\label{example: hexagon-trivalent}\
     Consider a tree $T$ with three inner edges all meeting at a single vertex. Note that the poset of bracketings depends only on the relative positions of the vertices (or analogously, the inner edges) of the tree $T$, and is independent of the number of leaves at each vertex. Therefore, the realization poset of bracketings of $T$ is the one depicted in the Figure \ref{fig:pentagon-hexagon} (right), using as an example the tree $T=\trivalenttree$. 

\end{example}

\begin{example}\label{example: four-valent}
     Figure \ref{fig: four-valent} depicts the realisation of the poset of bracketings of a tree $T$ with four inner edges meeting at a single vertex. Note that by fixing a large subtree $S$ of $T$, the realisation of the subposet of bracketings of $T$ containing $S$ will correspond to a subspace of the boundary of $|\mathcal{B(T)}|$. Each boundary face of top dimension is then associated to a subtree $S$ of $T$, and two such faces $S_1,S_2$ share a subface if $\{S_1,S_2\}$ is nested.
        \begin{figure}[h!t]
        \centering
        \begin{minipage}{0.45\textwidth}
            \centering\def\svgwidth{0.7\columnwidth}
            %% Creator: Inkscape 1.0beta2 (2b71d25, 2019-12-03), www.inkscape.org
%% Accompanies image file 'fourvalent-tree.pdf' (pdf, eps, ps)
\begingroup%
  \makeatletter%
  \providecommand\color[2][]{%
    \errmessage{(Inkscape) Color is used for the text in Inkscape, but the package 'color.sty' is not loaded}%
    \renewcommand\color[2][]{}%
  }%
  \providecommand\transparent[1]{%
    \errmessage{(Inkscape) Transparency is used (non-zero) for the text in Inkscape, but the package 'transparent.sty' is not loaded}%
    \renewcommand\transparent[1]{}%
  }%
  \providecommand\rotatebox[2]{#2}%
  \newcommand*\fsize{\dimexpr\f@size pt\relax}%
  \newcommand*\lineheight[1]{\fontsize{\fsize}{#1\fsize}\selectfont}%
  \ifx\svgwidth\undefined%
    \setlength{\unitlength}{236.46444966bp}%
    \ifx\svgscale\undefined%
      \relax%
    \else%
      \setlength{\unitlength}{\unitlength * \real{\svgscale}}%
    \fi%
  \else%
    \setlength{\unitlength}{\svgwidth}%
  \fi%
  \global\let\svgwidth\undefined%
  \global\let\svgscale\undefined%
  \makeatother%
  \begin{picture}(1,0.47678853)%
    \lineheight{1}%
    \setlength\tabcolsep{0pt}%
    \put(0,0){\includegraphics[width=\unitlength,page=1]{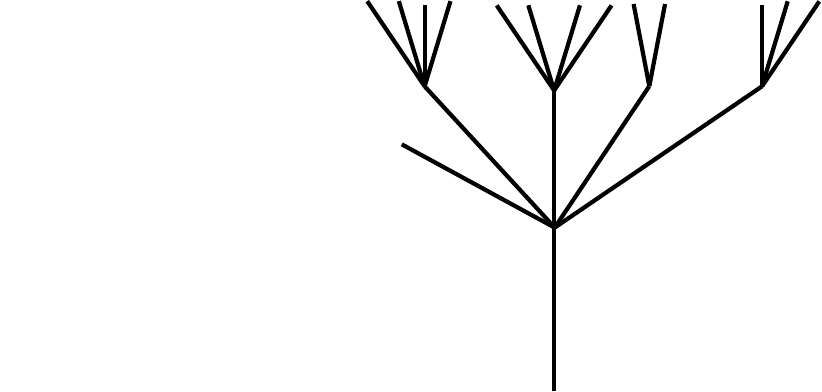}}%
    \put(0.3340831,0.20473176){\makebox(0,0)[rt]{\lineheight{1.25}\smash{\begin{tabular}[t]{r}$T=$\end{tabular}}}}%
  \end{picture}%
\endgroup%

        \end{minipage}
        \begin{minipage}{0.5\textwidth}
        \centering\includegraphics[width=0.9\linewidth]{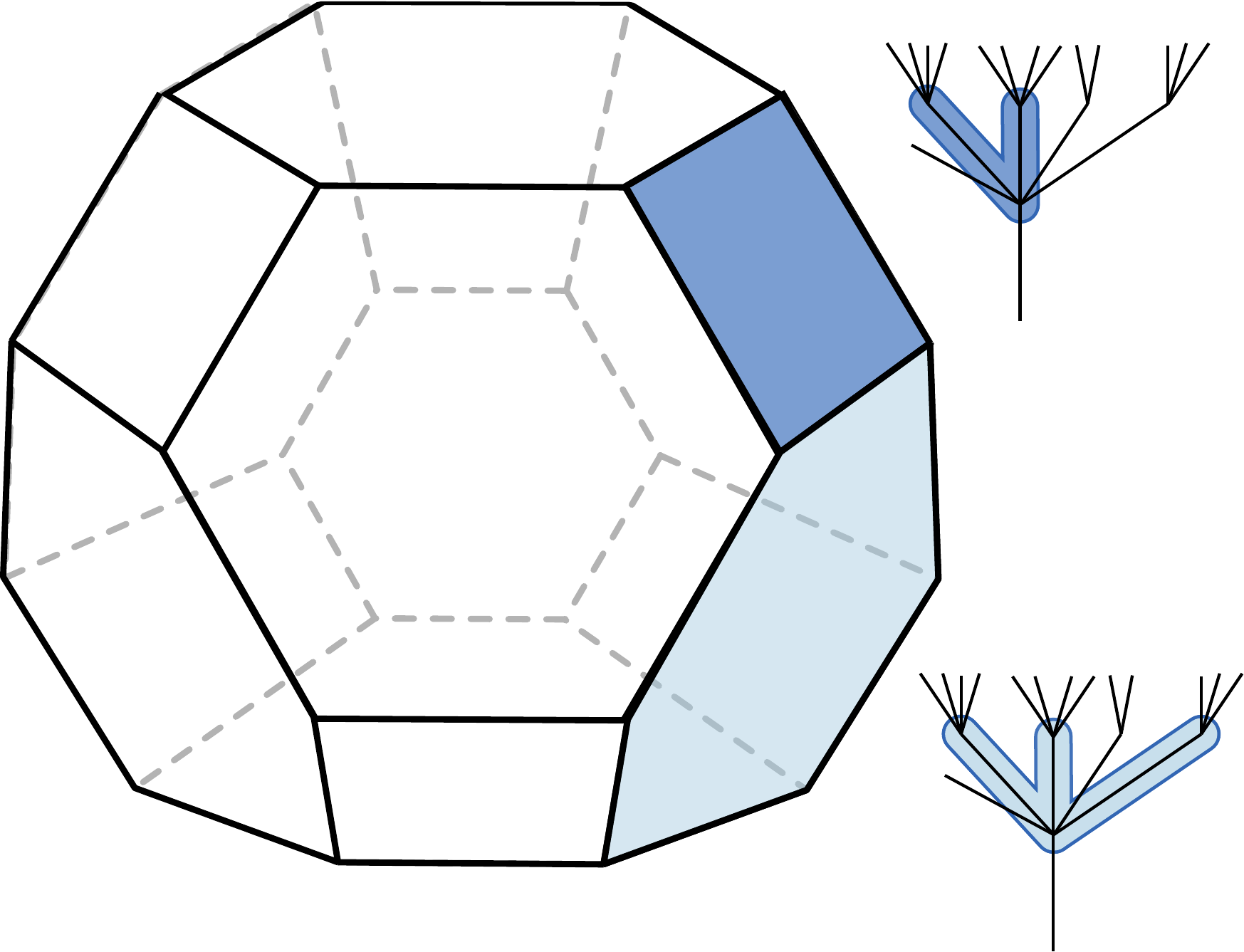}
        \end{minipage}
        \caption{Tree satisfying the conditions of Example \ref{example: four-valent} together with the geometric realisation of its poset of bracketings.}\label{fig: four-valent}
        \end{figure}
\end{example}

\begin{remark}\label{remark: polytopes}
The spaces $|\B(T)|$ are closely related to the abstract polytopes defined in \cite{Obr19}. In fact, we can show that $|\B(T)|$ identifies with the hypergraph polytope of the edge-graph $\mathbf{H}_T$ of $T$, as defined in \cite[Section 2.2.1]{Obr19}. The set of vertices of $\mathbf{H}_T$ is the set of inner edges of $T$, and two such share an edge if they have a common vertex. Then a subset $\mathbf{S}$ of vertices of $\mathbf{H}_T$ uniquely defines a subforest $\langle \mathbf{S}\rangle$ of $T$ whose internal edges are precisely the elements of $\mathbf{S}$, and each tree in this forest is necessarily large because it has an inner edge (see \cite[Section 2.2.1, Lemma 3]{Obr19}). Then we have an order reversing bijection $b$ between the abstract polytope of the edge-graph of $T$ and $\B(T)$, which can be recursively defined as follows: using the notation established in \cite{Obr19}, we take the construct $V(\mathbf{H}_T)$ to the empty bracketing, and if $\mathbf{H}_T\setminus Y\rightsquigarrow \mathbf{H}_{T_1},\dots,\mathbf{H}_{T_n}$, we take the construct $Y\{C_1,\dots,C_n\}$ to the bracketing $\{\langle \mathbf{H}_{T_1}\rangle,\dots,\langle \mathbf{H}_{T_n}\rangle,b(C_1),\dots,b(C_n)\}$. The definition of the constructs guarantees that these sets are nested and therefore define a bracketing, and it is simple to check that this is an order reversing bijection.
\end{remark}

\begin{lemma}\label{bracketing-space-realization}
    For any tree $T$, the space $|\mathcal{B}(T)|$, is contractible. 
\end{lemma}

\begin{proof}
The contractiblity of the space $|\mathcal{B}(T)|$ follows directly from the fact that the poset $\B(T)$ has a minimal element, namely the empty bracketing. 
\end{proof}

\subsection{An  operad of bracketings}\label{sec: operad of bracketings}

We'll use the bracketings $\mathcal{B}(T)$ to construct a topological operad. 
%We define an $\mathbb{N}$-coloured symmetric sequence with 
Let the collection
$$B\OO(n;m_1,\dots,m_k)=\coprod_{(T,\sigma,\tau)\in\mathcal{O}(n;m_1,\dots,m_k)}|\mathcal{B}(T)|$$ define the $\mathbb{N}$-coloured symmetric sequence $B \OO.$
So, elements of $B\OO(n;m_1,\dots,m_k)$ are tuples $(T,\sigma,\tau,B,t)$ where $(T,\sigma,\tau)$ is an element of $ \OO(n;m_1,\dots,m_k)$ (Definition~\ref{def:operad of operads}) and $(B,t)$ is a weighted bracketing of $T$ (ie. a point in $|\mathcal{B}(T)|$).

To define operadic composition in $B\OO$, we use the composition of trees in $\OO$ and induce a bracketing of the resulting tree. Let $(T,\sigma,\tau,B)$ and $(T',\sigma',\tau',B')$ be labeled trees with bracketings. 
The composition in $\OO$ (Definition \ref{def:operad of operads})  is given by the substitution of $T'$ into the vertex $\sigma(i)\in V(T)$,  \[(T,\sigma,\tau)\circ_{i} (T',\sigma',\tau') = (T\bullet_{\sigma(i),\tau'} T', \sigma \circ_i \sigma', \tau).  \] 
Since $T'$ is canonically a subtree of $T\bullet_{\sigma(i),\tau'} T'$, the bracketing $B'$ on $T'$ defines a nested collection of subtrees of $T\bullet_{\sigma(i),\tau'} T'$.
We also construct a nested collection of subtrees $\tilde B  =  \{\tilde S_j\}_{j \in J}$ on $T\bullet_{\sigma(i),\tau'} T'$ that is induced by the bracketing $B = \{S_j\}_{j \in J}$ on $T$. 
If $T' \neq \eta$, then $\tilde B \cong B$ is given by
 \begin{equation} \label{eq: tilde Sj} \tilde S_j = \left \{
 \begin{array}{ll}
    S_j  &  \text{ if } \sigma(i) \notin V(S_j), \\
    S_j \bullet_{\sigma(i),\tau'} T'  & \text{ if } \sigma(i) \in V(S_j).
 \end{array} \right .  \end{equation}  
If $T' = \eta$, then $\tilde B = \{\tilde S_j\}_{j \in J}$ is defined in the same way, unless  $\sigma(i) \in V(S_j)$ and $S_j$ has two vertices, in which case $S_j \bullet_{\sigma(i),\tau'} \eta $ is a corolla and 
is discarded as it is not large. That is, we replace $J$ with another indexing set $J'\subset J$, which is the subset of indices $j$ such that $S_j$ is large. 

We define a bracketing of the tree $T\bullet_{\sigma(i),\tau'} T'$ by
\begin{equation}\label{eq: Bracket comp}B''=\left\{\begin{array}{ll}\tilde B\cup B'\cup \{T'\} & \textrm{if}\  T'\ \textrm{is large}\\
 \tilde B & \textrm{else.}
 \end{array}\right. \end{equation}
See Figure \ref{fig:BOcomposition2}. 
This defines a composition of bracketings of trees. This composition is associative as follows.
Suppose $S_j\subset T$ is a bracket with only two vertices $v$ and $w$, and $T'$ is a tree with at least two vertices. If we first compose $\eta$ in $v$ and then $T'$ in $w$, the bracket $S_j$ is discarded during the first composition, and then replaced by a new bracket $T'$. Reversing the order of these two compositions yields the same result because first composing  $T'$ in $w$ will create a new bracket $T'$, and $S_j$ will not be discarded, but composing further $\eta$ in $v$ will equate $S_j$ and $T'$.
Otherwise, the associativity of the composition follows from the associativity on the composition in $\OO$.

The composition  also respects inclusions and thus is a poset map 
\begin{equation}\label{eq: BO comp}
\begin{tikzcd}  \mathcal{B}(T)\times \mathcal{B}(T')\ \arrow[r] &  \mathcal{B}(T\bullet_{\sigma(i),\tau'} T'). \end{tikzcd}
\end{equation}
The realization of the poset map \eqref{eq: BO comp} induces a map between the geometric realisations of the nerve of the posets.

\begin{figure}[h!t]
    \centering\def\svgwidth{0.8\columnwidth}
    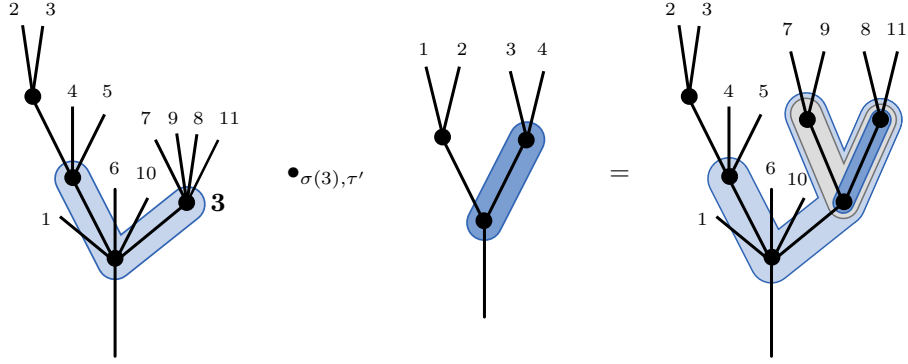
    \caption{Example of composition in $B\OO$ with labelling of the vertices omitted for simplicity.}\label{fig:BOcomposition2}
\end{figure}

Also recall that the unary elements of $\OO$, i.e.~the elements of $\OO(n;n)$ for some $n$, are given by labeled corollas. 
Since there are no non-trivial bracketings of corollas, unary elements of $B\OO$ have the form $(C_n, \sigma, *, \emptyset , \emptyset) \in B\OO(n;n) $ with $\sigma \in \Sigma_n$. In particular, the $n$-coloured identity for the composition $\circ $ in $B \OO$ is given by $(C_n, id_n, *, \emptyset , \emptyset) \in B\OO(n;n) $.
Therefore $B\OO$ is an operad.

\begin{definition}\label{def: BO} 
The \emph{operad of bracketed trees} $B\OO$ is the $\mathbb{N}$-coloured topological operad with underlying symmetric sequence
$$B\OO(n;m_1,\dots,m_k)=\coprod_{(T,\sigma,\tau)\in\mathcal{O}(n;m_1,\dots,m_k)}|\mathcal{B}(T)|$$ and composition given by combining the composition in $\OO$ with the map \eqref{eq: BO comp} described above. 
\end{definition}

\begin{remark}\label{rem:BO as poset operad}
The topological operad $B\OO$ is the realization of an operad in posets. Indeed, the space $B\OO(n;m_1,\dots,m_k)$ is the realization of the poset of elements $(T,\sigma,\tau)$ of $\OO(n;m_1,\dots,m_k)$ together with a bracketing of $T$, where two elements are comparable only if they have the same underlying element of $\OO$. Likewise, the operad structure is defined as the realization of a map on the level of posets. 
\end{remark}

\subsection{$B\OO$-algebras}\label{sec: BO-alg} 
A $B\OO$-\emph{algebra} is an operad whose $\circ_i$-compositions are associative up to all higher homotopies. In particular, a $B\OO$-algebra $\pp=\{\pp(n)\}_{n\in\mathbb{N}}$ has an underlying \emph{symmetric sequence}. To see this, we note that the labelling of the leaves of a corolla $(C_n, \tau, *, \emptyset , \emptyset) \in B\OO(n;n)$ identifies with elements of the symmetric group and we have isomorphisms $$B\OO(n;n)\cong \OO(n;n)\cong \Sigma_n.$$ The action \[\begin{tikzcd} B\OO(n;n)\times \pp(n)\arrow[r]& \pp(n) \end{tikzcd}\] makes $\pp=\{\pp(n)\}_{n\in\mathbb{N}}$ into a symmetric sequence. 

$B\OO$-algebras also have a notion of operadic {\em $\circ_i$ composition}. To see this, recall that such compositions are encoded in the operad $\OO$ by the trees with exactly two vertices, one attached to the $i$th incoming edge of the other.  As such trees admit no large, proper subtrees, they admit no non-trivial bracketing and we have isomorphisms for any $n,m\ge 0$
$$%(T,\sigma,\tau,\empty,0)\in 
B\OO(m+n-1;m,n)|_{V(T)\le 2}\cong \OO(m+n-1;m,n)|_{V(T)\le 2}$$
between the components of the tuples $(T,\sigma,\tau,\emptyset,0)$ (resp. $(T,\sigma,\tau)$) with $T$ having at most two vertices. 
It follows then that $\pp$ is equipped with operadic $\circ_i$-compositions. 

A $B\OO$-algebra is not in general an operad, however. The brackets that arise in trees with more than two vertices capture the different choices one has in iterated compositions of $\circ_i$ operations. More explicitly, if $\{\pp(n)\}_{n\in \mathbb{N}}$ is a $B\OO$--algebra, then for any collection of elements $x_i\in\pp(m_i)$ that decorate the vertices of a tree $(T,\sigma,\tau)\in \OO(n;m_1,\ldots,m_k)$, we have a chosen composition of those elements, namely the one determined by $(T,\sigma,\tau,\emptyset,\emptyset)\in B\OO(n;m_1,\ldots,m_k)$. 
%\abinote{Should it be $\OO$ instead of $\pp$?}
This ``unbracketed'' tree sits in the middle of a polytope of all possible elements $(T,\sigma,\tau,B,s)$ for any bracketing $B$, as in Figure~\ref{fig:pentagon-hexagon}. The corners of this polytope correspond to the possible maximal bracketings of $T$ (the maximal elements of $\B(T)$).
Just like the corners of the Stasheff polytopes give all the possible ways to bracket a $k$--fold multiplication,
these maximal bracketings correspond precisely to the possible ways to bracket the  composition of $\circ_i$ operations, which are those defined using trees with exactly two vertices.
The polytopes arising from the posets of bracketing in trees can be thought of as an operadic analogue of the Stasheff polytopes.

\begin{remark}\label{quasi-operad remark}
In \cite[Definition 1.1.1]{K05}, a \emph{quasi-operad} is a symmetric sequence $\pp=\{\pp(n)\}_{n\in\mathbb{N}}$ together with operadic $\circ_i$-compositions and no further structure. In this way, a $B\OO$-algebra is an extension of a quasi-operad. The operad $B\OO$ is closely related to the $W$-construction of $\OO$, whose algebras go under the name {\em lax operads}, see Appendix~\ref{sec: BO and WO}, where we show that $B\OO$-algebras can be described as strictly symmetric lax operads.
\end{remark}

\section{Thickening the category $\Omega$}\label{sec:thickening-omega}

We have seen that operads are $\OO$-algebras. 
Also recall from Section~\ref{operads and dendroidal spaces} that operads can be described as strict Segal dendroidal spaces.  
The dendroidal category $\om$ is defined as a full subcategory of the coloured operads generated by trees. To obtain a similar description of $B\OO$--algebras as certain ``homotopy dendroidal Segal spaces,'' 
we construct a topological category $\widetilde\om_0$ that is a category with the same objects as $\om$ but its spaces of morphisms are built using posets similar to the posets used to define $B\OO$. 
Theorem~\ref{thm:W_0O-Omega_0} establishes that this category, $\widetilde\om_0$, has the desired property that strict reduced Segal $\widetilde\om_0^{op}$--spaces are precisely $B\mathcal{O}$-algebras. In Section~\ref{WOdendroidal}, we then show how rectification of diagrams can be used to produce an actual Segal dendroidal space from such a homotopy version of a dendroidal space.

Given any category $\mathcal{K}$ with a discrete set of objects, Leitch \cite{Leitch} constructed a new category $\widetilde{\mathcal{K}}$ with the property that $\widetilde{\mathcal{K}}$--diagrams are homotopy coherent $\mathcal{K}$--diagrams. A similar enrichment (the \emph{explosion category}) was also used by Segal \cite[Appendix B]{Segal74} to relate his $\Gamma$--space approach to infinite loop spaces to the operadic approach of Boardman-Vogt and May. Because $\widetilde\om_0$--diagrams are homotopy coherent $\om_0$--diagrams, one can expect that the category $\widetilde\om_0$ is related to this construction of Leitch applied to $\om$. In Appendix~\ref{sec:hatomega}, we construct  an equivalence between these two categories, and show that strict Segal $\widetilde\om$--spaces are closely related to $W\OO$--algebras.

\subsection{Bracketing $\Omega$ and the category $\widetilde\om_0$}\label{sec: Omega W}

Recall from Section~\ref{sec:omega} that the objects of $\om$ are planar isomorphism classes of planar rooted trees. Morphisms in $\om$ are 
compositions of inner and outer face maps, degeneracies and isomorphisms of trees. Inner face maps $\partial_{e}:T/e\rightarrow T$ create inner edges and correspond to operadic composition, while outer face maps  
are subtree inclusions and are associated to projection maps. 
A degeneracy creates a vertex that is adjacent to exactly two edges. 
The category $\widetilde\om_0$ is a version of $\om$ with the same set of objects, but 
with the realization of a poset of bracketings over each composition of inner face maps.

We define the morphism spaces of $\widetilde\om_0$ as follows.
Let $g:S\rightarrow T$ be a morphism in $\om$. 
For each vertex $v\in V(S)$, let $C_v\subset S$ denote the corolla of the vertex $v$ that is, 
$C_v=i_v(C_{|in(v)|})$ where $i_v:C_{|in(v)|}\to S$ is the composition of outer faces in $\om$ sending the vertex of the corolla $C_{|in(v)|}$ to $v$.
Since $g$ is alternatively considered as a map of operads between $\om(S)$ and $\om(T)$,
the image in $S$ of $C_v$ under $g$  is a subtree in $T$, which we denote $$g(C_v)\ \subset \ T.$$ Note that the trees $g(C_v)$ are precisely the subtrees of $T$ that correspond to expansion of vertices into subtrees, going from $S$ to $T$, or collapsed by $g^{op}:T\to S$ in the opposite category $\om^{op}$. These subtrees correspond to the part of $g$ made out of inner face maps.

For a vertex $v\in V(S)$, let $B_v^g$ be a bracketing of $g(C_v)$ as defined in Definition~\ref{def:bracketing}. 
We define a poset $\mathcal{L}_g$ whose objects are tuples $(B_v^g)_{v\in V(S)}$ of bracketings of the trees $g(C_v)$. The poset relation is componentwise inclusion. Taking the realization of these posets, for each morphism $g$ we associate  the space $$L_g:=\prod\limits_{v\in V(S)}|\B(g(C_v))|$$
where $\B(g(C_v))$ is the poset of bracketings of the tree $g(C_v)$ as defined in Definition~\ref{def: poset of brackets}.  Note also that $|\B(g(C_v))|=*$ if $g(C_v)$ admits only the trivial bracketing. 

\begin{example}\label{ex:omega-w-f}
    Consider the morphism $f\in\Hom_\om(R,S)$ of Figure~\ref{fig:K-f}.   Since the image of each corolla under $f$ only admits a trivial bracketing,  \[L_f=\left|\B\left(\bluetree\right)\right|\times \left|\B\left(\pinktree\right)\right|\times \left|\B\left(\orangetree\right)\right|\times \left|\B\left(\yellowtree\right)\right|= *.\]
    \begin{figure}[h!t]
    \centering\def\svgwidth{0.4\columnwidth}
    %% Creator: Inkscape 1.0beta2 (2b71d25, 2019-12-03), www.inkscape.org
%% Accompanies image file 'Normal-weight-omega-example1.pdf' (pdf, eps, ps)
\begingroup%
  \makeatletter%
  \providecommand\color[2][]{%
    \errmessage{(Inkscape) Color is used for the text in Inkscape, but the package 'color.sty' is not loaded}%
    \renewcommand\color[2][]{}%
  }%
  \providecommand\transparent[1]{%
    \errmessage{(Inkscape) Transparency is used (non-zero) for the text in Inkscape, but the package 'transparent.sty' is not loaded}%
    \renewcommand\transparent[1]{}%
  }%
  \providecommand\rotatebox[2]{#2}%
  \newcommand*\fsize{\dimexpr\f@size pt\relax}%
  \newcommand*\lineheight[1]{\fontsize{\fsize}{#1\fsize}\selectfont}%
  \ifx\svgwidth\undefined%
    \setlength{\unitlength}{129.46984863bp}%
    \ifx\svgscale\undefined%
      \relax%
    \else%
      \setlength{\unitlength}{\unitlength * \real{\svgscale}}%
    \fi%
  \else%
    \setlength{\unitlength}{\svgwidth}%
  \fi%
  \global\let\svgwidth\undefined%
  \global\let\svgscale\undefined%
  \makeatother%
  \begin{picture}(1,0.60327419)%
    \lineheight{1}%
    \setlength\tabcolsep{0pt}%
    \put(0,0){\includegraphics[width=\unitlength,page=1]{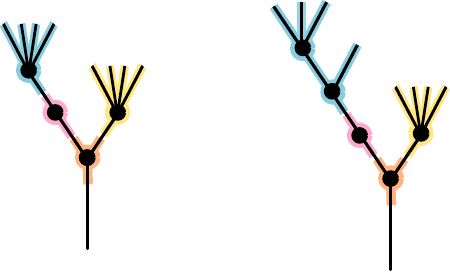}}%
    \put(0.48126138,0.2906449){\color[rgb]{0,0,0}\makebox(0,0)[t]{\lineheight{1.25}\smash{\begin{tabular}[t]{c}$\xrightarrow{\phantom{g}f=\partial_e\phantom{g}}$\end{tabular}}}}%
    \put(0.6704974,0.40782566){\color[rgb]{0,0,0}\makebox(0,0)[t]{\lineheight{1.25}\smash{\begin{tabular}[t]{c}$e$\end{tabular}}}}%
  \end{picture}%
\endgroup%

    \caption{Example of map $f$ in $\om$ and the subtrees $f(C_v)$.}\label{fig:K-f}
    \end{figure}
 \end{example}
 
 \begin{example}\label{example: polytope-in-omega} 
    Let $s$ be the morphism in Figure \ref{fig:map_s}. 
    By Example~\ref{example: pentagon}, if $s(C_v)$ has $3$ vertices such that no vertex is connected to more than two inner edges, then $|\mathcal{B}(s(C_v))|$ is the $3$rd associahedron, which is an interval.
    As in Example \ref{example: hexagon-trivalent}, when $s(C_v)$ is a tree whose three internal edges meet at a single vertex, the realization poset $|\mathcal{B}(s(C_v))|$ corresponds to a hexagon.
    Thus $L_s$ is identified with the hexagonal prism of Figure \ref{fig:L_s}.

\begin{figure}[h!t]
        \hfill
        \subfigure[Map $s$ in $\om$ and the subtrees $s(C_v)$.]{\label{fig:map_s}
        \centering\def\svgwidth{0.38\columnwidth}
        %% Creator: Inkscape 1.0beta2 (2b71d25, 2019-12-03), www.inkscape.org
%% Accompanies image file 'map-s.pdf' (pdf, eps, ps)
\begingroup%
  \makeatletter%
  \providecommand\color[2][]{%
    \errmessage{(Inkscape) Color is used for the text in Inkscape, but the package 'color.sty' is not loaded}%
    \renewcommand\color[2][]{}%
  }%
  \providecommand\transparent[1]{%
    \errmessage{(Inkscape) Transparency is used (non-zero) for the text in Inkscape, but the package 'transparent.sty' is not loaded}%
    \renewcommand\transparent[1]{}%
  }%
  \providecommand\rotatebox[2]{#2}%
  \newcommand*\fsize{\dimexpr\f@size pt\relax}%
  \newcommand*\lineheight[1]{\fontsize{\fsize}{#1\fsize}\selectfont}%
  \ifx\svgwidth\undefined%
    \setlength{\unitlength}{123.50406647bp}%
    \ifx\svgscale\undefined%
      \relax%
    \else%
      \setlength{\unitlength}{\unitlength * \real{\svgscale}}%
    \fi%
  \else%
    \setlength{\unitlength}{\svgwidth}%
  \fi%
  \global\let\svgwidth\undefined%
  \global\let\svgscale\undefined%
  \makeatother%
  \begin{picture}(1,0.54421268)%
    \lineheight{1}%
    \setlength\tabcolsep{0pt}%
    \put(0,0){\includegraphics[width=\unitlength,page=1]{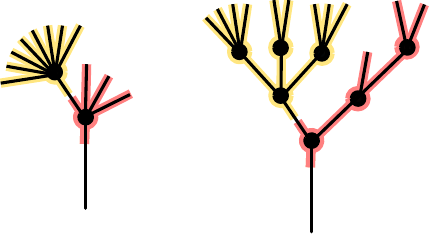}}%
    \put(0.43933196,0.29480681){\makebox(0,0)[t]{\lineheight{1.25}\smash{\begin{tabular}[t]{c}$\xrightarrow{s}$\end{tabular}}}}%
  \end{picture}%
\endgroup%
}
        \hfill
        \subfigure[Space $L_s$ for the map $s$ of Figure \ref{fig:map_s}]{\label{fig:L_s}
        \centering\def\svgwidth{0.49\columnwidth}
        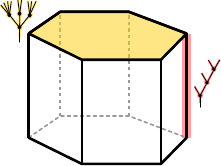}
        \hfill \vspace{-12pt}
        \caption{A map $s$ in $\Omega$ and its corresponding space $L_s$.}
        \end{figure}
\end{example}
The space of morphisms between any two objects in $\widetilde\om_0$ is
$$\Hom_{\widetilde\om_0}(S,T)=\coprod\limits_{g\in\Hom_{\om}(S,T)} L_g.$$ It remains to define composition in $\widetilde\om_0$. To do this, we first define a map of posets 
\begin{equation}\label{poset_comp}
\begin{tikzcd}\mathcal{L}_g\times \mathcal{L}_f\ \arrow[r]& \mathcal{L}_{g\circ f}\end{tikzcd}
\end{equation}
for any two morphisms $f:R\to S$ and $g:S\to T$ in $\om$, then  we take the realization of this composition map to get a composition of spaces $L_g$. 
Let $(B^g_v)_{v\in V(S)}\in \mathcal{L}_g$ and $(B^f_w)_{w\in V(R)}\in \mathcal{L}_f$ be two collections of bracketings. So 
for each $v\in V(S)$, $B^g_v$ is a bracketing of $g(C_v)\subset T$  and for each $w\in V(R)$,
$B^f_w$ is a bracketing of the tree $f(C_w)\subset S$. 
To define the image of \eqref{poset_comp}, we construct a bracketing of the tree $(g\circ f)(C_w)$ from the bracketings of $f$ and $g$.

\medskip

Fix a vertex $w\in V(R)$. For each $v\in f(C_w)\subset S$, there is the subtree $g(C_v)\subset (g\circ f)(C_w)$, as well as a bracketing $B^g_v$ of $g(C_v)$. Also, for each bracket in $S_i\in B^f_w$, the image $g(S_i)$ is a subtree of $(g\circ f)(C_w)$. Therefore we have the following collections of subtrees in $(g\circ f)(C_w)$:
\begin{align*}
&\Tilde{B}^g_{f(w)} = \bigcup_{v\in f(C_w)} B_v^g \ =\  \{S_j:\; S_j\in B^g_v \textrm{ and } v\in f(C_w)\}\\
&\Tilde{B}^{g\circ f}_{w}=\ \{g(C_v): \; v\in f(C_w) \textrm{ and } g(C_v)\subsetneq (g\circ f)(C_w)\textrm{ is large}\}\\
&\Tilde{B}^f_w=\ \{g(S_i): \; S_i\in B^f_w \textrm{ and } g(S_y)\subsetneq (g\circ f)(C_w) \textrm{ is large}\}.
\end{align*}

All of these are collections of proper large subtrees of $(g\circ f)(C_w)$. We set the bracketing $\Tilde{B}_w$ of $(g\circ f)(C_w)$ to be the union
\[\Tilde{B}_w:= \Tilde{B}^g_{f(w)} \cup \Tilde{B}^{g\circ f}_{w} \cup \Tilde{B}^f_w.\]
To see that $\Tilde{B}_w$ is a bracketing of $(g\circ f)(C_w)$, it remains to verify that 
this collection is comprised of nested subtrees. First, each $B^g_v\subset \Tilde{B}^g_{f(w)}$ is a bracketing of $g(C_v)\subset (g\circ f)(C_w)$, so it is nested. 
Moreover, the subtrees $g(C_v)$ are all disjoint and each tree of $\Tilde{B}^g_{f(w)}$ is contained in a tree of $\Tilde{B}^{g\circ f}_w$, so the union $\Tilde{B}^g_{f(w)}\cup \Tilde{B}^{g\circ f}_w$ is nested too. The
%since every $S_y\in B^f_w$ contains at least two vertices of $f(C_w)$, then 
$\Tilde{B}^{g\circ f}_w\cup \Tilde{B}^f_w$ is also nested, since each bracket $g(C_v)$ in the first set is included in each $g(S_y)$ of the second set whenever $v\in S_y$ and otherwise is disjoint from it. Hence $\Tilde{B}^g_{f(w)}\cup \Tilde{B}^f_w$ is also nested, and thus $\Tilde{B}_w$ is nested.

\medskip

Define the composition $(B^f_v)_{v\in V(S)}\circ (B^g_w)_{w\in V(R)}$ to be the collection \begin{equation}\notag
(\tilde B_w)_{w\in V(R)}\in \mathcal{L}_{g\circ f}.
\end{equation} 
Associativity of this composition is analogous to the associativity of the $B\OO$ composition in Section~\ref{sec: operad of bracketings}. In most cases, the composition is associative because vertex substitution is associative. In a composition with a degeneracy, a vertex is removed and so a bracket may be discarded if it is no longer large. Any discarded bracket is recreated in a subsequent composition if it should not have been discarded in the total composition. 

Furthermore, this composition definition respects componentwise inclusion and thus defines the poset map \eqref{poset_comp}. The realization of this poset map induces a map
\begin{equation}\label{mid_omega_comp}
\begin{tikzcd}  L_g\times L_f\ \arrow[r] &  L_{g\circ f}. \end{tikzcd}
\end{equation}
This defines a composition on the morphism spaces of $\tilde \om_0$.

\begin{example}\label{ex:omega-w-gf} 
    Let $f\colon R\to S$ and $g\colon S\to T$ be the morphisms in Figure~\ref{fig: composition of tree morphisms}.
    Then $R$ is a corolla $C_w=C_9$, and $f(C_w)\subset S$ is the proper subtree of $S$ whose vertices are $v_1, v_2, v_3$.
    The images $g(C_{v_1}), g(C_{v_3}), g(C_{v_4}) \subset T$ are the corollas $C_{u_1}, C_{u_5}, C_{u_6}$ respectively and $g(C_{v_2})$ is the subtree with vertices $u_2, u_3, u_4$. 
    The only images of corollas that admit a non-trivial bracketing are $f(C_w)$ and $g(C_{v_2})$. If the bracketing of $f(C_w)$ consists of the bracket $B_1$ in Figure~\ref{fig:B_1} and the bracketing of $g(C_{v_2})$ consists of $B_2$ in Figure~\ref{fig:B_2}, then
    $$\tilde B_{f(w)}^g=\{B_2\}, \quad \tilde B^{g\circ f}_{w}=\{g(C_{v_2})\}, \quad \tilde B^{f}_{w}= \{g(B_1)\}.$$ 
    The bracketing $\tilde B_w\in \B((g\circ f)(C_w))$ is illustrated in Figure~\ref{fig:B_w}. 
    
    \begin{figure}[h!t]
        \centering\def\svgwidth{0.67\columnwidth}
        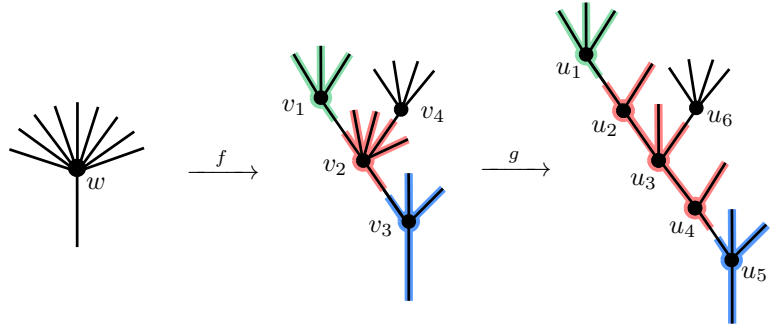
    \caption{Morphisms $f$ and $g$ in $\om$.}\label{fig: composition of tree morphisms}
    \end{figure}

    \begin{figure}[h!t]
    \centering
        %\hfill
        \subfigure[Bracketing $B_1$ of $f(C_w)$.]{\label{fig:B_1}
        \centering\def\svgwidth{0.25\columnwidth}
        %% Creator: Inkscape 1.0beta2 (2b71d25, 2019-12-03), www.inkscape.org
%% Accompanies image file 'Morphism-composition-bracketing1.pdf' (pdf, eps, ps)
\begingroup%
  \makeatletter%
  \providecommand\color[2][]{%
    \errmessage{(Inkscape) Color is used for the text in Inkscape, but the package 'color.sty' is not loaded}%
    \renewcommand\color[2][]{}%
  }%
  \providecommand\transparent[1]{%
    \errmessage{(Inkscape) Transparency is used (non-zero) for the text in Inkscape, but the package 'transparent.sty' is not loaded}%
    \renewcommand\transparent[1]{}%
  }%
  \providecommand\rotatebox[2]{#2}%
  \newcommand*\fsize{\dimexpr\f@size pt\relax}%
  \newcommand*\lineheight[1]{\fontsize{\fsize}{#1\fsize}\selectfont}%
  \ifx\svgwidth\undefined%
    \setlength{\unitlength}{82.1927948bp}%
    \ifx\svgscale\undefined%
      \relax%
    \else%
      \setlength{\unitlength}{\unitlength * \real{\svgscale}}%
    \fi%
  \else%
    \setlength{\unitlength}{\svgwidth}%
  \fi%
  \global\let\svgwidth\undefined%
  \global\let\svgscale\undefined%
  \makeatother%
  \begin{picture}(1,1.11163305)%
    \lineheight{1}%
    \setlength\tabcolsep{0pt}%
    \put(0,0){\includegraphics[width=\unitlength,page=1]{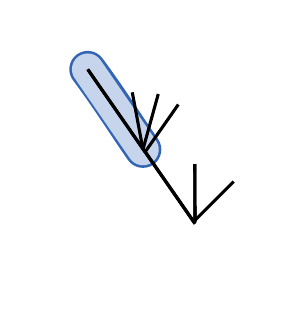}}%
    \put(0.20130793,0.82502257){\color[rgb]{0,0,0}\makebox(0,0)[t]{\lineheight{1.25}\smash{\begin{tabular}[t]{c}$v_1$\end{tabular}}}}%
    \put(0.38997617,0.54549555){\color[rgb]{0,0,0}\makebox(0,0)[t]{\lineheight{1.25}\smash{\begin{tabular}[t]{c}$v_2$\end{tabular}}}}%
    \put(0.59037383,0.30289634){\color[rgb]{0,0,0}\makebox(0,0)[t]{\lineheight{1.25}\smash{\begin{tabular}[t]{c}$v_3$\end{tabular}}}}%
    \put(0,0){\includegraphics[width=\unitlength,page=2]{Morphism-composition-bracketing1.pdf}}%
  \end{picture}%
\endgroup%
}
        %\hfill
        \qquad
        \subfigure[Bracketing $B_2$ of $g(C_{v_2})$.]{\label{fig:B_2}
        \centering\def\svgwidth{0.25\columnwidth}
        %% Creator: Inkscape 1.0beta2 (2b71d25, 2019-12-03), www.inkscape.org
%% Accompanies image file 'Morphism-composition-bracketing2.pdf' (pdf, eps, ps)
\begingroup%
  \makeatletter%
  \providecommand\color[2][]{%
    \errmessage{(Inkscape) Color is used for the text in Inkscape, but the package 'color.sty' is not loaded}%
    \renewcommand\color[2][]{}%
  }%
  \providecommand\transparent[1]{%
    \errmessage{(Inkscape) Transparency is used (non-zero) for the text in Inkscape, but the package 'transparent.sty' is not loaded}%
    \renewcommand\transparent[1]{}%
  }%
  \providecommand\rotatebox[2]{#2}%
  \newcommand*\fsize{\dimexpr\f@size pt\relax}%
  \newcommand*\lineheight[1]{\fontsize{\fsize}{#1\fsize}\selectfont}%
  \ifx\svgwidth\undefined%
    \setlength{\unitlength}{82.1927948bp}%
    \ifx\svgscale\undefined%
      \relax%
    \else%
      \setlength{\unitlength}{\unitlength * \real{\svgscale}}%
    \fi%
  \else%
    \setlength{\unitlength}{\svgwidth}%
  \fi%
  \global\let\svgwidth\undefined%
  \global\let\svgscale\undefined%
  \makeatother%
  \begin{picture}(1,1.11163305)%
    \lineheight{1}%
    \setlength\tabcolsep{0pt}%
    \put(0,0){\includegraphics[width=\unitlength,page=1]{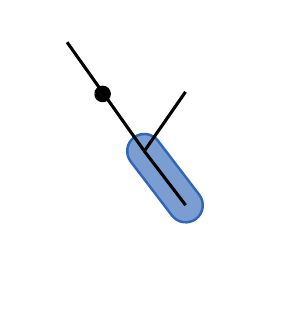}}%
    \put(0.28556684,0.69741285){\color[rgb]{0,0,0}\makebox(0,0)[t]{\lineheight{1.25}\smash{\begin{tabular}[t]{c}$u_2$\end{tabular}}}}%
    \put(0.42129612,0.48747014){\color[rgb]{0,0,0}\makebox(0,0)[t]{\lineheight{1.25}\smash{\begin{tabular}[t]{c}$u_3$\end{tabular}}}}%
    \put(0.57658079,0.29636836){\color[rgb]{0,0,0}\makebox(0,0)[t]{\lineheight{1.25}\smash{\begin{tabular}[t]{c}$u_4$\end{tabular}}}}%
    \put(0,0){\includegraphics[width=\unitlength,page=2]{Morphism-composition-bracketing2.pdf}}%
  \end{picture}%
\endgroup%
}
        \qquad %\hfill
        \subfigure[Bracketing $\Tilde{B}_w$ of $g\circ f(C_w)$.]{\label{fig:B_w}
        \centering\def\svgwidth{0.29\columnwidth}
        %% Creator: Inkscape 1.0beta2 (2b71d25, 2019-12-03), www.inkscape.org
%% Accompanies image file 'Morphism-composition-bracketingw.pdf' (pdf, eps, ps)
\begingroup%
  \makeatletter%
  \providecommand\color[2][]{%
    \errmessage{(Inkscape) Color is used for the text in Inkscape, but the package 'color.sty' is not loaded}%
    \renewcommand\color[2][]{}%
  }%
  \providecommand\transparent[1]{%
    \errmessage{(Inkscape) Transparency is used (non-zero) for the text in Inkscape, but the package 'transparent.sty' is not loaded}%
    \renewcommand\transparent[1]{}%
  }%
  \providecommand\rotatebox[2]{#2}%
  \newcommand*\fsize{\dimexpr\f@size pt\relax}%
  \newcommand*\lineheight[1]{\fontsize{\fsize}{#1\fsize}\selectfont}%
  \ifx\svgwidth\undefined%
    \setlength{\unitlength}{91.13472748bp}%
    \ifx\svgscale\undefined%
      \relax%
    \else%
      \setlength{\unitlength}{\unitlength * \real{\svgscale}}%
    \fi%
  \else%
    \setlength{\unitlength}{\svgwidth}%
  \fi%
  \global\let\svgwidth\undefined%
  \global\let\svgscale\undefined%
  \makeatother%
  \begin{picture}(1,1.00305608)%
    \lineheight{1}%
    \setlength\tabcolsep{0pt}%
    \put(0,0){\includegraphics[width=\unitlength,page=1]{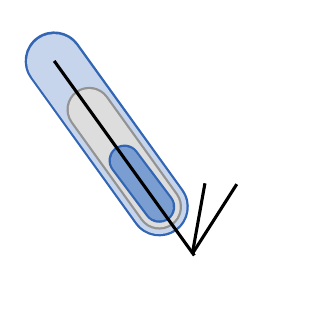}}%
    \put(0.0648355,0.73189184){\color[rgb]{0,0,0}\makebox(0,0)[t]{\lineheight{1.25}\smash{\begin{tabular}[t]{c}$u_1$\end{tabular}}}}%
    \put(0.16794223,0.57548466){\color[rgb]{0,0,0}\makebox(0,0)[t]{\lineheight{1.25}\smash{\begin{tabular}[t]{c}$u_2$\end{tabular}}}}%
    \put(0.28293013,0.41381614){\color[rgb]{0,0,0}\makebox(0,0)[t]{\lineheight{1.25}\smash{\begin{tabular}[t]{c}$u_3$\end{tabular}}}}%
    \put(0.4102646,0.25275653){\color[rgb]{0,0,0}\makebox(0,0)[t]{\lineheight{1.25}\smash{\begin{tabular}[t]{c}$u_4$\end{tabular}}}}%
    \put(0.67891411,0.15199883){\color[rgb]{0,0,0}\makebox(0,0)[t]{\lineheight{1.25}\smash{\begin{tabular}[t]{c}$u_5$\end{tabular}}}}%
    \put(0,0){\includegraphics[width=\unitlength,page=2]{Morphism-composition-bracketingw.pdf}}%
  \end{picture}%
\endgroup%
}
        %\hfill 
        \vspace{-1em}
        \caption{Example of a bracketing induced by $f$ and $g$ in Figure~\ref{fig: composition of tree morphisms}.}\label{fig: bracketing of composition of maps}
        \end{figure}
    
By Example~\ref{example: pentagon},  if $T_n$ is a tree with $n$ vertices such that no vertex is connected to more than two inner edges, then $|\B(T_n)|$ is the $n$th associahedron. The $3$rd associahedron is an interval. Thus,
    \begin{align*}
        &L_f=|\mathcal{B}(f(C_w))|=\left|\B\left(\bluetreetwo\right)\right|\cong [0,1] \\
        &L_g=\left|\B\left(g(C_{v_1})\right)\right|\times \left|\B\left(\bluetreetwo\right)\right|\times \left|\B\left(g(C_{v_3})\right)\right| \cong [0,1].
    \end{align*}
    Again by Example~\ref{example: pentagon} and since $(g\circ f) (C_w)$ is a tree on five vertices, $L_{g\circ f}=|\B( (g\circ f) (C_w))|$ is the $5$th associahedron, which is a three dimensional polytope called an enneahedron.  
\end{example}

\begin{definition}\label{def: Omega W}
The category $\widetilde\Omega_0$ has the same objects as $\Omega$. 
Morphism spaces in $\widetilde\om_0$ are $$\Hom_{\widetilde\om_0}(S,T)=\coprod\limits_{g\in\Hom_{\om}(S,T)} L_g=\coprod\limits_{g\in\Hom_{\om}(S,T)}\prod\limits_{v\in V(S)}|\B(g(C_v))| $$
with composition (\ref{mid_omega_comp}) as described above. 
\end{definition}

\begin{example} %\abinote{Rewrote this example}
    Suppose $T_{n}$ is a planar tree with $(n+1)$ leaves and $n$ vertices, each of which is connected to at most two inner edges. Let the inner edges of $T_n$ be named $e_1, \ldots, e_{n-1}$.
    Morphisms $g\in\Hom_{\Omega}(C_{n+1},T_n)$ are compositions of inner face maps $\partial_{e_1}, \ldots, \partial_{e_{n-1}}$ but since the order of the composition does not affect the total composition, there is only one such morphism $g$. 
    Hence  $\Hom_{\widetilde\om_0}(C_{n+1},T_n)=L_g=|\mathcal{B}(T_n)|$.
    Thus $\Hom_{\widetilde\om_0}(C_{n+1},T_n)$ is the $n$th associahedron by Example~\ref{example: pentagon}; the centre point of the polytope is defined by the empty bracket, which is the initial object in the poset $\mathcal{B}(T_n)$. 
\end{example}

Lemma~\ref{bracketing-space-realization} tells us that each bracketing space $|\B(g(C_v))|$ is contractible, which implies that each $L_g$ is contractible. 
Let $p:\widetilde\Omega_0\rightarrow\Omega$ be the functor that is the identity on objects and projects  each morphism space $L_g$ to $g$. 
By considering $\om$ as a discrete topological category, we  have the following proposition.
\begin{prop}\label{prop: morphisms of hat-Omega_0 are contractible}
    The functor $p:\widetilde\Omega_0\rightarrow\Omega$ induces a homotopy equivalence on morphism spaces. \qed
\end{prop}

This proposition will allow us to associate an actual dendroidal space to any homotopy dendroidal space in Section~\ref{WOdendroidal}.

\subsection{Homotopy dendroidal spaces}\label{sec: homotopy dendroidal spaces}

In Section~\ref{sec:omega}, we defined a Segal condition for dendroidal spaces $X:\om^{op}\to \mathcal{S}$ using the Segal map 
\begin{equation*}%\label{eq: Segal} 
\begin{tikzcd} \chi\colon X(T)\arrow[r] & \lim_{{\rm Sk}_1(T)^{op}}X.\end{tikzcd}%\tag{$\star$}
\end{equation*}
We recall that the category ${\rm Sk}_1(T)$ has the vertices and edges of $T$ as objects, with morphisms given by edge inclusions $\iota_e: \eta\to C_v$ into the corollas of adjacent vertices. The Segal map $\chi$ is the unique map to the limit induced by the edge and corolla inclusions
$$\iota_e: \eta \to T \ \ \ \textrm{and}\ \ \  \iota_v: C_{|v|}\to T.$$ 
Note that the spaces $L_{\iota_e}$ and $L_{\iota_v}$ in $\widetilde\om_0^{op}$ which lie above the morphisms $\iota_e$ and $\iota_v$ are always just a single point, so the Segal map exists unchanged for functors $X:\widetilde\om_0^{op}\to \mathcal{S}$. This allows us to make the following definition:

\begin{definition}\label{def:homotopy dendroidal space} 
A \emph{homotopy dendroidal space} $X$ is a diagram $X:\widetilde\om^{op}_0\rightarrow \mathcal{S}$. A homotopy dendroidal space is \emph{reduced} if $X(\eta)\simeq*$ and \emph{strictly reduced} if $X(\eta)=*$. A homotopy dendroidal space satisfies the \emph{strict} Segal condition if the Segal map is an \emph{isomorphism} for each $\eta\neq T\in\om$ and a homotopy dendroidal space satisfies a \emph{weak} Segal condition if the Segal map is a homotopy equivalence for each $\eta\neq T\in\om$. 
\end{definition}

Recall from Section~\ref{operads and dendroidal spaces} that one-colored operads are identified with strictly reduced dendroidal Segal spaces via the dendroidal nerve \[N^d:\OO\textrm{-Alg}\rightarrow\mathcal{S}^{\om^{op}}.\] The following theorem is a version of this nerve theorem for homotopy dendroidal spaces.  We construct a functor $$\Phi: B\mathcal{O}\mathrm{-Alg}_{\mathcal{S}} \longrightarrow \mathcal{S}^{\widetilde\Omega_0^{op}}$$ and show that a homotopy dendroidal space $X\in\mathcal{S}^{\widetilde\om_0^{op}}$ with $X=*$ is strictly Segal if, and only if, $X\cong \Phi(\mathcal{P})$ for some $B\OO$-algebra $\mathcal{P}$.  

\medskip

Write $(\mathcal{S}^{\widetilde\om_0^{op}})_{strict}$ for the full subcategory of $\widetilde\om_0$-diagrams whose objects are strictly reduced homotopy dendroidal spaces satisfying the strict Segal condition. Then we have the following result:

\begin{theorem}\label{thm:W_0O-Omega_0}
There exists an isomorphism of categories \[\begin{tikzcd}\Phi: B\mathcal{O}\mathrm{-Alg}_{\mathcal{S}} \arrow[r, "\cong"]& (\mathcal{S}^{\widetilde\Omega_0^{op}})_{strict}. \end{tikzcd}\] 
\end{theorem}

\begin{proof}
Given a $B\OO$--algebra $\pp=\{\pp(n)\}_{n\ge 0}$ with structure maps
$$\alpha_\pp:B\OO(n;m_1,\dots,m_k)\times \pp(m_1)\times\dots\times \pp(m_k)\longrightarrow \pp(n)$$ we will define
$$\Phi(\pp)=\Phi(\pp,\alpha_\pp)\colon \widetilde\om_0^{op} \longrightarrow \mathcal{S}$$ as follows. We set $\Phi(\pp)(\eta)=*$. 
On objects $T\not=\eta$ of $\widetilde\om_0$, we set
$$\Phi(\pp)(T)=\prod_{w\in V(T)}\pp(|w|).$$
Given a morphism $g:S\to T$ in $\om$, we need to define maps
$$\Phi(\pp)(g): L_g\times \prod_{w\in V(T)}\pp(|w|) \longrightarrow \prod_{v\in V(S)}\pp(|v|).$$
We proceed one vertex $v$ at a time. As $L_g=\prod_{v\in V(S)}|\B(g(C_v))|$, at each $v\in V(S)$ we have projection maps
$$\pi_v: L_g\times \prod_{w\in V(T)}\pp(|w|)\longrightarrow |\B(g(C_v))|\times \prod_{w\in V(g(C_v))}\pp(|w|).  $$
An application of the structure map $\alpha_\pp$ defines a map
\begin{equation}\label{equ:la}
   \alpha_v: |\B(g(C_v))|\times \prod_{w\in V(g(C_v))}\pp(|w|) \longrightarrow \pp(|v|).\tag{$*$}
    \end{equation}
Indeed, an element of $|\B(g(C_v))|$ is a weighted bracketing $(B,t)$ of the subtree $g(C_v)\subset T$. Because $T$ is a planar tree, $g(C_v)$ inherits a planar structure. We consider $g(C_v)$ as an element of $\OO$ by picking an ordering $\sigma$ of its vertices $\{w_1,\dots,w_k\}$, and labeling its leaves via the map $\tau$ ordering them according to its planar structure. This way $((g(C_v),\sigma,\tau),B,t)$ is an element of $B\OO(|v|;|w_1|,\dots,|w_k|)$.  
To define the map \eqref{equ:la}, we first order the factors $\pp(|w|)$ for $w\in V(g(C_v))$, in accordance with our chosen $\sigma$, and then apply $\alpha_\pp$ noting that our choice of ordering does not affect the result by the equivariance of $\alpha_\pp$. Finally we act on the resulting element of $\pp(|v|)$ by the permutation induced by $g$ that identifies the inputs of $v$ with the leaves of $g(C_v)$, comparing the labeling $\tau$ from the planar structure of $T$ to the planar ordering of $in(v)$ (which comes from the planar structure of $S$). 
We now set $$\Phi(\pp)(g):=(\alpha_v\circ \pi_v)_{v\in V(S)}.$$

The fact that $\Phi(\pp)$ commutes with composition follows from the fact that composition in $\widetilde\om_0$ is defined exactly as the operadic composition of $B\OO$ by taking the union of the brackets from the first morphism which remain large after applying the second morphisms, the brackets from the second morphism, and  new ``middle brackets'', the images of the middle corollas, if they are large. It follows then that  $\Phi(\pp):\widetilde\om_0\to \mathcal{S}$ is a functor.
Since $\Phi(\pp)(\eta)=*$, the Segal map is the map 
$$\Phi(\pp)(T)\longrightarrow \prod\limits_{v\in V(T)}\Phi(\pp)(C_v)$$
induced by the inclusions of the corollas. It is an isomorphism by definition of $\Phi(\pp)$. 

\medskip

The data required in the definition of the homotopy dendroidal space $\Phi(\pp)$ is the underlying symmetric sequence $\pp=\{\pp(m)\}$, the $B\OO$-algebra structure maps $\alpha_{\pp}$ and the projection maps $\pi_v$, all of which are natural under maps of $B\OO$-algebras. Thus, the assignment $\pp\mapsto \Phi(\pp)$ defines a functor \[\Phi:B\mathcal{O}\mathrm{-Alg}_{\mathcal{S}} \longrightarrow (\mathcal{S}^{\widetilde\Omega_0^{op}})_{strict}.\]

\medskip

It remains to show that the functor $\Phi$ is an isomorphism of categories. Given two $B\OO$-algebras $\pp$ and $\mathcal{Q}$  with $\Phi(\pp)=\Phi(\mathcal{Q})$, the underlying symmetric sequences $\{\pp(n)\}_{n\ge 0}$ and $\{\mathcal{Q}(n)\}_{n\ge 0}$ are necessarily equal, being the value at the corollas $C_n$ and the corolla isomorphisms $\Hom_{\widetilde\om_0}(C_n,C_n)\cong\Hom_{\om}(C_n,C_n)\cong \Sigma_n$.  
%and $\pp(n)=\mathcal{Q}(n)$ for each $n$. % as $\Sigma_n$--modules, and 
Moreover, the structure maps $\alpha_\pp$ and $\alpha_{\mathcal{Q}}$ likewise must agree as they agree with the evaluation of $\Phi(\pp)=\Phi(\mathcal{Q})$ at corresponding morphisms in $\widetilde\Omega_0^{op}$.  It follows that $\Phi$ is injective. 

\medskip

On the other hand, given any $X\in (\mathcal{S}^{\widetilde\Omega_0^{op}})_{strict}$, we can construct a $B\OO$-algebra $\pp_X$ by setting $\pp_X(n)=X(C_n)$ with a symmetric group action induced by the image under $X$ of the isomorphisms of $C_n$ in $\widetilde\om_0$.  %$\widetilde\om_0(C_n,C_n)\cong\om(C_n,C_n)\cong \Sigma_n$. 
The $B\OO$-algebra structure maps of $\pp_X$ are defined using the above identification of the spaces $B\OO(n;m_1,\dots,m_k)$ with morphism spaces in $\widetilde\om_0$. The fact that $X$ is a functor will then give that $\pp_X$ is a $B\OO$--algebra. Thus the functor $\Phi$ is surjective. 
\end{proof}

\subsection{Rectifying homotopy dendroidal spaces}\label{WOdendroidal}

We have just seen that $B\OO$-algebras correspond to homotopy dendroidal spaces satisfying the strict Segal condition. In this section we will show how to produce, from a $B\OO$-algebra, an actual dendroidal space satisfying the weak Segal condition. To do this we will use some elementary facts about the homotopy theory of diagram categories in the form of Quillen model categories. 

\medskip 

A commutative diagram in a topological category $\mathcal{S}$ is a functor from a discrete category $\mathcal{K}$ to $\mathcal{S}$. A homotopy commutative diagram can be similarly described as a functor from a topological category $\widetilde{\mathcal{K}}$ to $\mathcal{S}$, with the homotopies encoded as paths in the spaces of morphisms. In this language, a homotopy commutative diagram $X:\widetilde{\mathcal{K}}\to \mathcal{S}$ can be {\em rectified}, or {\em strictified}, to a functor $X':\mathcal{K}\to \mathcal{S}$ precisely when there is an equivalence $p:\widetilde{\mathcal{K}}\to \mathcal{K}$. We briefly recall this rectification of diagrams, which was used by Segal in \cite{Segal74}, and treated in great generality by Dwyer and Kan \cite{DwyerKan}; see also \cite[Sec 2]{Wahl04} for a detailed account of what we will use here. Our examples will be $\mathcal{K}=\om$ with $\widetilde{\mathcal{K}}=\widetilde\om_0$.

Let $p:\widetilde{\mathcal{K}}\to \mathcal{K}$ be a functor between categories enriched over topological spaces. There is an induced functor
$$p^*:\mathcal{S}^{\mathcal{K}} \longrightarrow \mathcal{S}^{\widetilde{\mathcal{K}}}$$
defined by precomposition with $p$. The homotopy left Kan extension defines also a functor 
$$p_*:\mathcal{S}^{\widetilde{\mathcal{K}}} \longrightarrow \mathcal{S}^{\mathcal{K}}$$
that can be explicitly given as follows: 
given a diagram $Y\in\mathcal{S}^{\widetilde{\mathcal{K}}}$, its evaluation at an object $d$ of $\mathcal{K}$ is the realization of a simplicial space with space of $k$--simplices
 $$(p_*Y(d))_k=\coprod\limits_{c_0,\dots,c_k\in ob(\widetilde{\mathcal{K}})} Y(c_0)\times  \Hom_{\widetilde{\mathcal{K}}}(c_0,c_1)\times\dots \times \Hom_{\widetilde{\mathcal{K}}}(c_{k-1},c_k)\times \Hom_{{\mathcal{K}}}(c_k,d). $$

\begin{lemma}\cite[Proposition 2.1]{Wahl04}\label{the-natural-transformations} Let $p\colon\widetilde{\mathcal{K}}\to \mathcal{K}$ be a functor inducing a homotopy equivalence of morphism spaces, and let $Y\colon\widetilde{\mathcal{K}}\to \mathcal{S}$ be a diagram, with $p_*Y\colon \mathcal{K}\to \mathcal{S}$ its rectification as defined above.  
Then there exists a zig-zag of natural transformations $p^*p_*Y\leftarrow\ \overline{p^*p_*Y}  \rightarrow Y$, which induces a homotopy equivalence on objects: $p^*p_*Y(d)\simeq \overline{p^*p_*Y}(d) \simeq Y(d)$. 
\end{lemma}
In the statement, $\overline{p^*p_*Y}$ is an explicit functor from $\widetilde{\mathcal{K}}$ to $\mathcal{S}$ associated to $Y$ given by a two-sided bar construction (details in the proof of \cite[Proposition 2.1]{Wahl04}). 

\medskip

Proposition~\ref{prop: morphisms of hat-Omega_0 are contractible} states that the functor $p:\widetilde\om_0\to \om$ induces a homotopy equivalence of morphism spaces. % is an equivalence of topological categories
Below, we apply Lemma~\ref{the-natural-transformations} to describe Segal dendroidal spaces %$p_*Y\in\mathcal{S}^{\om^{op}}$ 
which arise as the rectification of a homotopy dendroidal Segal space $Y\in \mathcal{S}^{\widetilde\om_0^{op}}$.

\medskip

For any small category $\mathcal{K}$, the category of diagrams $\mathcal{K}^{op}\rightarrow\mathcal{S}$ admits a \emph{projective model structure} in which weak equivalences and fibrations are defined entrywise \cite[11.6.1]{MR1944041}. In particular, there is a projective model category structure on the category of reduced dendroidal spaces \cite[Proposition 3.10]{Bergner_Hackney_14}. Similarly, there is a projective model category structure on the category of reduced homotopy dendroidal spaces. Moreover, an application of \cite[Proposition A.3.3.7]{LurieHTT} implies that the homotopy left Kan extension $p_*$ is the left Quillen functor in a Quillen equivalence 
\[
\begin{tikzcd}
    \mathcal{S}^{\widetilde\om_0^{op}} \arrow[r, bend left=15, "p_*"] & \mathcal{S}^{\om^{op}}.\arrow[l, bend left=15, "p^*"]
\end{tikzcd}
\]

We note that a fibrant diagram in either $\mathcal{S}^{\widetilde\om_0^{op}}$ or $\mathcal{S}^{\om^{op}}$ is, in particular, entrywise fibrant. When $Y$ is a reduced homotopy dendroidal space and fibrant, then one could identify the limit defining the Segal map with the homotopy limit. In this case, the Segal map $\widetilde\chi$ can be written 
\[\begin{tikzcd}
Y(T)\arrow[r,"\widetilde\chi"] & \prod\limits_{v\in V(T)}Y(C_v).
\end{tikzcd}\]  

\begin{prop}\label{prop:homotopy Segal implies Segal}
Let $Y\in\mathcal{S}^{\widetilde\om_0^{op}}$ be a fibrant reduced homotopy dendroidal space. Then
the rectification $p_*Y\in\mathcal{S}^{\om^{op}}$ is a reduced dendroidal space, and the fibrant replacement $(p_*Y)_f$ satisfies the weak Segal condition if and only if $Y$ does.  
\end{prop}

\begin{proof}
Note first that, because $p$ is the identity on objects, we have $p^*X(T)=X(T)$ for any $X:\om^{op}\to \mathcal{S}$ and any $T\in Obj(\om)\equiv Obj(\widetilde\om_0)$.  It follows  that $p_*Y(\eta)\simeq *$ if $Y(\eta)\simeq *$ as $p_*Y(\eta)=p^*p_* Y(\eta)\simeq Y(\eta)\simeq *$. 
\medskip

We are left to show that the Segal map a is weak equivalence for every $T\neq \eta$ for $p_*Y:\om^{op}\to \mathcal{S}$ if and only if it is the case for the original functor $Y:\widetilde\om^{op}_0\to \mathcal{S}$. Recall that the Segal maps for $Y$ and $p_*Y$ are the maps   
 \begin{align*}
 Y(T) & \xrightarrow{\widetilde\chi} \lim_{{\rm Sk}_1(T)^{op}}Y(C_v) \\
 p_*Y(T)  & \xrightarrow{\chi} \lim_{{\rm Sk}_1(T)^{op}}p_*Y(C_v) 
 \end{align*} 
 both of which are induced by corolla and edge inclusions in $\widetilde\om_0$ and $\om$, respectively. 
 Now, $p$ takes the map $\widetilde{\chi}$, and each map $\iota_e$ and $\iota_v$ in $\widetilde\om_0$ used to define the limit, to the corresponding map in $\om$ used to define $\chi$. 
 Using that $p$ is the identity on objects, 
 for any $X:\om^{op}\to \mathcal{S}$, we have 
 $$\xymatrix{ X(T)\ar@{=}[d]  \ar[r]^-{\chi} & \lim_{{\rm Sk}_1(T)^{op}}X(C_v) \ar@{=}[d] \\
  p^*X(T)  \ar[r]^-{\widetilde\chi} & \lim_{{\rm Sk}_1(T)^{op}}p^*X(C_v)
 }$$
 in which the two horizontal maps describe the exact same map in $\mathcal{S}$.

Since $p_*$ is the left adjoint in our Quillen pair, it may not the be the case that $p_*Y$ is fibrant.  However, we do know that $p^*(p_*Y)_f$ is fibrant, where $(p_*Y)_f$ denotes the fibrant replacement of $p_*Y$ in $\mathcal{S}^{\om^{op}}$. Similarly, we let $(\overline{p^*p_*Y})_f$ denote the fibrant replacement of the diagram $\overline{p^*p_*Y}$.  
Since limits commute with homotopy equivalences whenever our diagram is fibrant,
the natural equivalences of functors $p^*p_*Y \leftarrow \overline{p^*p_*Y} \rightarrow Y$ of Lemma \ref{the-natural-transformations} %$\varepsilon$ 
give us the vertical homotopy equivalences in the following commuting diagram in $\mathcal{S}$

\[\begin{tikzcd}[column sep=huge]{p^*(p_*Y)_f(T)} \arrow[r, "\widetilde\chi"]  & {\prod\limits_{v\in V(T)}p^*(p_*Y)_f(C_v)} \\{(\overline{p^*p_*Y})_f(T)} \arrow[r, "\widetilde\chi"] \arrow[u, "\simeq"]\arrow[d, "\simeq"'] & {\prod\limits_{v\in V(T)}(\overline{p^*p_*Y})_f(C_v)} \arrow[d, "\simeq"]\arrow[u, "\simeq"'] \\{Y(T)} \arrow[r, "\widetilde\chi"]  & {\prod\limits_{v\in V(T)}Y(C_v)}.  \end{tikzcd}\] 

Using the previous remark in the case $X=p_*Y$ identifies the top line of the diagram with the Segal map for $p_*Y$. It thus follows that $p_*Y$ satisfies the weak Segal condition (i.e. the top map is a weak equivalence) if, and only if, $Y$ satisfies the weak Segal condition (i.e. the bottom map is a weak equivalence).
\end{proof}

We are now ready to prove Theorem~\ref{ThBO} from the introduction.

\begin{proof}[Proof of Theorem~\ref{ThBO}]
Let $\pp$ be a $B\OO$-algebra. Applying the functor $\Phi$ of Theorem~\ref{thm:W_0O-Omega_0}, we obtain a homotopy dendroidal space $X:=\Phi(\pp)\in\mathcal{S}^{\widetilde\om_0^{op}}$, which we know, by the theorem, is a strictly reduced homotopy dendroidal space satisfying the strict Segal condition. If $X$ is not fibrant, we take a fibrant replacement $(X)_f$.

Set $Y:=(p_*X)_f= (p_*\Phi(\pp))_f\in \mathcal{S}^{\om^{op}}$ which, by Lemma~\ref{the-natural-transformations}, has the property that $Y(C_w)=(p_*\Phi(\pp))_f(C_w)\simeq \Phi(\pp)(C_w)=\pp(|w|)$ and the value of $Y$ on inner face maps identifies under these homotopy equivalences with the value of $X$ on inner face maps, and hence identifies with the $B\OO$--algebra composition. We can now apply Proposition~\ref{prop:homotopy Segal implies Segal} to $X$ to conclude that $Y$ is a reduced dendroidal space that satisfies the weak Segal condition. 
    
\end{proof}

\begin{remark}
Since the dendroidal category is a generalized Reedy category \cite[Example 1.6]{bmreedy}, there is also a Reedy model structure on the category of reduced dendroidal spaces. Proposition 3.3 of \cite{Bergner_Hackney_14} says that the identity functor induces a Quillen equivalence between the Reedy model structure and the projective model structures on reduced dendroidal spaces. 

We use the projective model structure here because, for our purposes, it is not necessary to show that the category $\widetilde\om_0$ is an enriched generalized Reedy category.  If one wished to do so, one would need to put an enriched generalized Reedy model structure on $\widetilde\om_0$ and then repeat the argument in Proposition~\ref{prop:homotopy Segal implies Segal} with the appropriate Reedy fibrant objects. 
\end{remark}
\section{Normalized Cacti as an infinity operad}\label{sec:cacti}

The first goal of this section is to define an operad $MS^+$ and show that, despite not being an operad itself,
normalized cacti and their composition can be described as elements and compositions inside $MS^+$.
In  Section~\ref{sec: Cact1 is BO}, we will use $MS^+$ to show that normalized cacti and the normalized composition extends to define
a $B\mathcal{O}$-algebra structure. 
Using the results of Sections~\ref{sec:BO} and \ref{sec:thickening-omega}, this implies that we have an explicit construction of an $\infty$-operad with underlying sequence the spaces $\cact^1(n)$.

\medskip

A cactus is a configuration of circles of various lengths attached to each other in a treelike fashion.
In the original definition by Voronov \cite[Section 2.7]{V05},
there is a global basepoint  associated to the ``outside circle'' of the cactus, as well as a basepoint for each circle (or \textit{lobe}).
A \emph{spineless cactus} is a variant introduced by Kaufmann \cite[Section 2.3]{K05},  where the basepoint of each lobe is its closest point to the global basepoint along the outside circle. See Figure~\ref{fig:Cact-example} for an example. 
\begin{figure}[ht]
    \centering\def\svgwidth{0.5\columnwidth}
    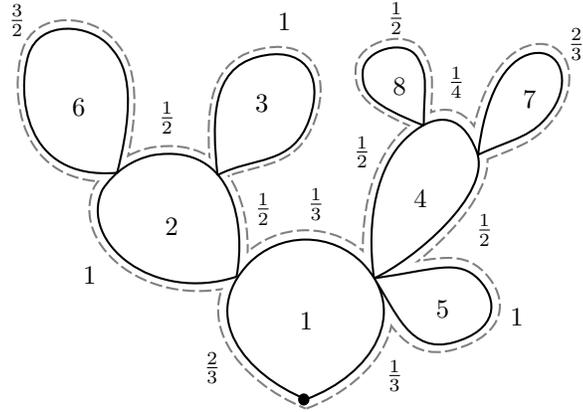
\caption{Cactus with $8$ lobes, its outside circle indicated by the dotted line. }\label{fig:Cact-example}
\end{figure}
The space of all spineless cacti with $k$ lobes is denoted $\cact(k)$. The symmetric group acts on this space by permuting the labels of the lobes.  The  symmetric sequence $\cact=\{\cact(k)\}_{k\geq0}$ is given a composition 
\begin{equation*}
			\circ_i \colon \cact(k) \times \cact(j) \to \cact(k+j-1)
		\end{equation*}
that is defined by inserting the second cactus into the $i$th lobe of the first cactus and aligning its global basepoint with the basepoint of the $i$th lobe. 
The insertion is done by rescaling the second cactus so that its total length is equal to the length of the $i$th lobe of the first cactus, then identifying the outside circle of the second cactus with the $i$th lobe of the first cactus. 
This composition makes $\cact$ into an operad, which is equivalent to the little $2$-discs operad \cite[Section 3.2.1]{K05}.
A rigorous definition of this composition requires close attention to subtleties and we refer to \cite[Section 2]{K05} for precise definitions. 

\begin{figure}[ht]
    \centering\def\svgwidth{0.6\columnwidth}
    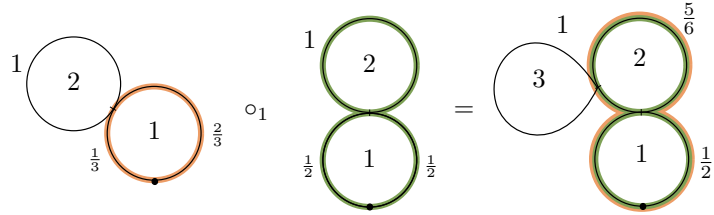
\caption{A composition of normalized cacti. }\label{fig:Cact-composition}
\end{figure}

The space of {\em normalized cacti}  $\cact^1(k)\subset \cact(k)$ is the subspace of spineless cacti whose lobes all have length equal to 1 (\cite[Definition 2.3.1]{K05}). 
They form a symmetric sequence $\cact^1=\{\cact^1(k)\}_{k\ge 0}$.
Composition of normalized cacti 
\begin{equation}\label{eq: cact^1 composition}
			\circ_i \colon \cact^1(k) \times \cact^1(j) \to \cact^1(k+j-1),
		\end{equation}
is defined by 
reparameterizing the $i$th lobe of a 
cactus $x\in\cact^1(k)$ to have length $j$, then identifying this lobe with the 
outer circle of the second cactus $y\in\cact^1(j)$ and aligning their basepoints.
In contrast to $\cact$, 
the $i$th lobe of the first cactus is scaled 
%before insertion the $i$th lobe of the first cactus is scaled to have length $j$,
instead of scaling the second cactus to the length of the $i$th lobe. % it is inserted in, one scales the lobe to the size of the inserted cactus (Figure~\ref{fig:Cact-composition}). 
See Figure~\ref{fig:Cact-composition} for an example. 
This composition is not associative \cite[Remark 2.3.19]{K05}, as illustrated in Figure~\ref{fig:CactCompUnassoc}.
Thus $\cact^1$ is \emph{not} an operad.

\begin{figure}[ht]
    \centering\def\svgwidth{0.7\columnwidth}
    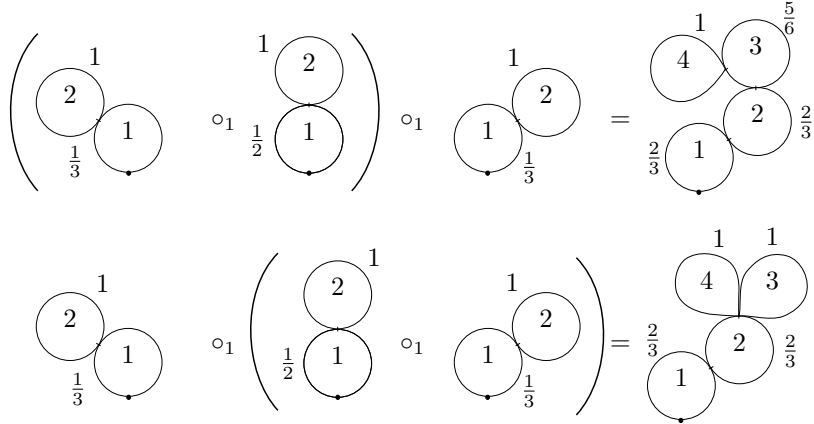
    \caption{Non-associativity in $\cact^1$}\label{fig:CactCompUnassoc}
    \end{figure}

\begin{remark}[Composition in the graph cobordism category]\label{rem:cobcomposition}
This composition of normalized cacti is highly relevant to the graph model of the cobordism category of Riemann surfaces mentioned in the introduction of the paper. 
To model the gluing of cobordisms, we use graphs to represent surfaces with potentially many incoming and outgoing boundary components. 
Normalized cacti are a simple case of this model, representing surfaces of genus zero with potentially many inputs but always just one output.
Two surfaces are glued by attaching the incoming boundaries of the first surface to the outgoing boundaries of the second.
According to \cite{godin07} (see also \cite[Theorem A]{egas_comparing}), we may assume that all incoming boundaries of a surface are disjoint embedded circles in the corresponding graph (like the lobes of the cactus, if they where pulled apart a little bit). Since these boundary circles are disjoint in the graph, they can be scaled independently to each match the length of an outgoing boundary in the graph of the second surface, just like scaling the $i$th lobe of the first cactus in $\cact^1$ composition. There is no obvious way to define a ``$\cact$-like" composition for such more general graphs, because the outgoing circles of the second surface cannot be assumed to be disjoint, and hence cannot be scaled independently to the appropriate length.  (See \cite[Section 3.3]{egas_comparing} for more details about this gluing of fat graphs.) 
\end{remark}

\subsection{An operad $MS^+$ that contains $\cact^1$}\label{subsection: MS operad}
In their proof of the Deligne conjecture, McClure and Smith \cite{MS02,MS04} introduced an operad $MS$ equivalent the little $2$-discs operad.\footnote{The operad $MS$ is denoted $\mathcal{C}'$ in \cite[Section 5]{MS02}.} Later, Salvatore \cite[Section 4]{Sal09} used similar methods to show directly that the operad $MS$ is equivalent to the non-normalized cactus operad $\cact$.
Here we will define a variant of $MS$ called $MS^+$, and, following \cite{Sal09}, start by showing that it is an operad by proving that it embeds in $\CoEnd(S^1)$.
We then show that normalized cacti are a subspace of the underlying symmetric sequence of $MS^+$ and that their composition can be written in terms of compositions in $MS^+$.

\medskip

The space of operations $MS^+(k)$ is built from a space $\mathcal{F}(k)$, which we will show is homeomorphic to $\cact^1(k)$.
In fact, we can think of an element of $\F(k)$ as the outer circle of a cactus.
 
\begin{definition}\cite[Definition 4.1]{Sal09}\label{def:Fn}
Let $S^1=[0,1]/0\!\sim\!1$ be the circle of circumference equal to 1.  
    Define $\mathcal{F}(k)$ as the space of partitions $x=(I_1(x),\dots,I_k(x))$ of $S^1$ into closed $1$-manifolds $I_j(x)\subset S^1$, each of which have total length $\frac{1}{k}$, with pairwise disjoint interiors, and such that 
    \begin{itemize}
    \item[$(*)$] there does not exist a cyclically ordered $4$-tuple $(z_1;z_2;z_3;z_4)\in S^1$ with $z_1,z_3\in \mathring{I}_j(x)$ and $z_2,z_4\in \mathring{I}_i(x)$, for $j\neq i$.  
    \end{itemize}
    For an example, see Figure~\ref{fig:x in F(3)}. The topology of $\mathcal{F}(k)$ is induced by the metric measuring the size of the overlap between partitions: 
    for $x,y\in \mathcal{F}(k)$, $d(x,y)=1-\sum_{j=1}^k\ell(I_j(x)\cap I_j(y))$ for $\ell$ the length function on submanifolds of $S^1$. 

    The symmetric group $\Sigma_k$ acts on $\mathcal{F}(k)$ by reindexing the labels of the $1$-manifolds. 
\end{definition}

\begin{figure}[h!t]
        \hfill
        \subfigure[$x\in\mathcal{F}(3)$.]{\label{fig:x in F(3)}
        \centering\def\svgwidth{0.2\columnwidth}
        %% Creator: Inkscape 1.0beta2 (2b71d25, 2019-12-03), www.inkscape.org
%% Accompanies image file 'Fn.pdf' (pdf, eps, ps)
\begingroup%
  \makeatletter%
  \providecommand\color[2][]{%
    \errmessage{(Inkscape) Color is used for the text in Inkscape, but the package 'color.sty' is not loaded}%
    \renewcommand\color[2][]{}%
  }%
  \providecommand\transparent[1]{%
    \errmessage{(Inkscape) Transparency is used (non-zero) for the text in Inkscape, but the package 'transparent.sty' is not loaded}%
    \renewcommand\transparent[1]{}%
  }%
  \providecommand\rotatebox[2]{#2}%
  \newcommand*\fsize{\dimexpr\f@size pt\relax}%
  \newcommand*\lineheight[1]{\fontsize{\fsize}{#1\fsize}\selectfont}%
  \ifx\svgwidth\undefined%
    \setlength{\unitlength}{223.25640106bp}%
    \ifx\svgscale\undefined%
      \relax%
    \else%
      \setlength{\unitlength}{\unitlength * \real{\svgscale}}%
    \fi%
  \else%
    \setlength{\unitlength}{\svgwidth}%
  \fi%
  \global\let\svgwidth\undefined%
  \global\let\svgscale\undefined%
  \makeatother%
  \begin{picture}(1,1.0210514)%
    \lineheight{1}%
    \setlength\tabcolsep{0pt}%
    \put(0,0){\includegraphics[width=\unitlength,page=1]{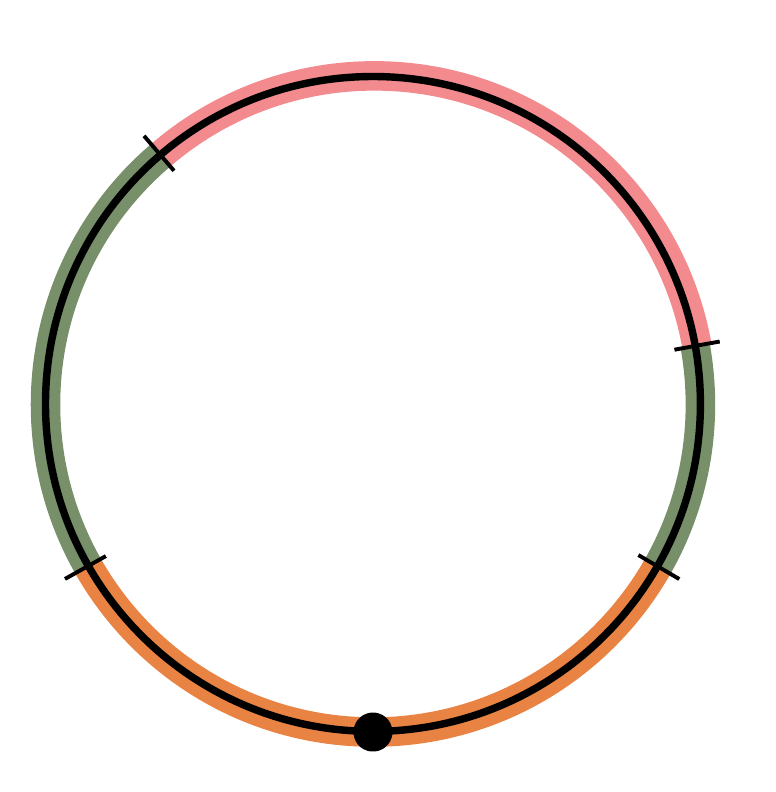}}%
    \put(0.32362355,0.01549087){\color[rgb]{0,0,0}\makebox(0,0)[rt]{\lineheight{1.25}\smash{\begin{tabular}[t]{r}$I_1(x)$\end{tabular}}}}%
    \put(0.05788137,0.22920179){\color[rgb]{0,0,0}\makebox(0,0)[rt]{\lineheight{1.25}\smash{\begin{tabular}[t]{r}$\frac{1}{6}$\end{tabular}}}}%
    \put(0.16338278,0.89351173){\color[rgb]{0,0,0}\makebox(0,0)[rt]{\lineheight{1.25}\smash{\begin{tabular}[t]{r}$\frac{7}{18}$\end{tabular}}}}%
    \put(0.94171511,0.55538988){\color[rgb]{0,0,0}\makebox(0,0)[lt]{\lineheight{1.25}\smash{\begin{tabular}[t]{l}$\frac{13}{18}$\end{tabular}}}}%
    \put(0.90591395,0.21329016){\color[rgb]{0,0,0}\makebox(0,0)[lt]{\lineheight{1.25}\smash{\begin{tabular}[t]{l}$\frac{15}{18}$\end{tabular}}}}%
    \put(0.02866103,0.56771577){\color[rgb]{0,0,0}\makebox(0,0)[rt]{\lineheight{1.25}\smash{\begin{tabular}[t]{r}$I_2(x)$\end{tabular}}}}%
    \put(0.70597914,0.96179504){\color[rgb]{0,0,0}\makebox(0,0)[t]{\lineheight{1.25}\smash{\begin{tabular}[t]{c}$I_3(x)$\end{tabular}}}}%
  \end{picture}%
\endgroup%
}\hfill
        \subfigure[Maps $c_x^i$.]{\label{fig:maps c-x}
        \includegraphics[width=0.6\linewidth]{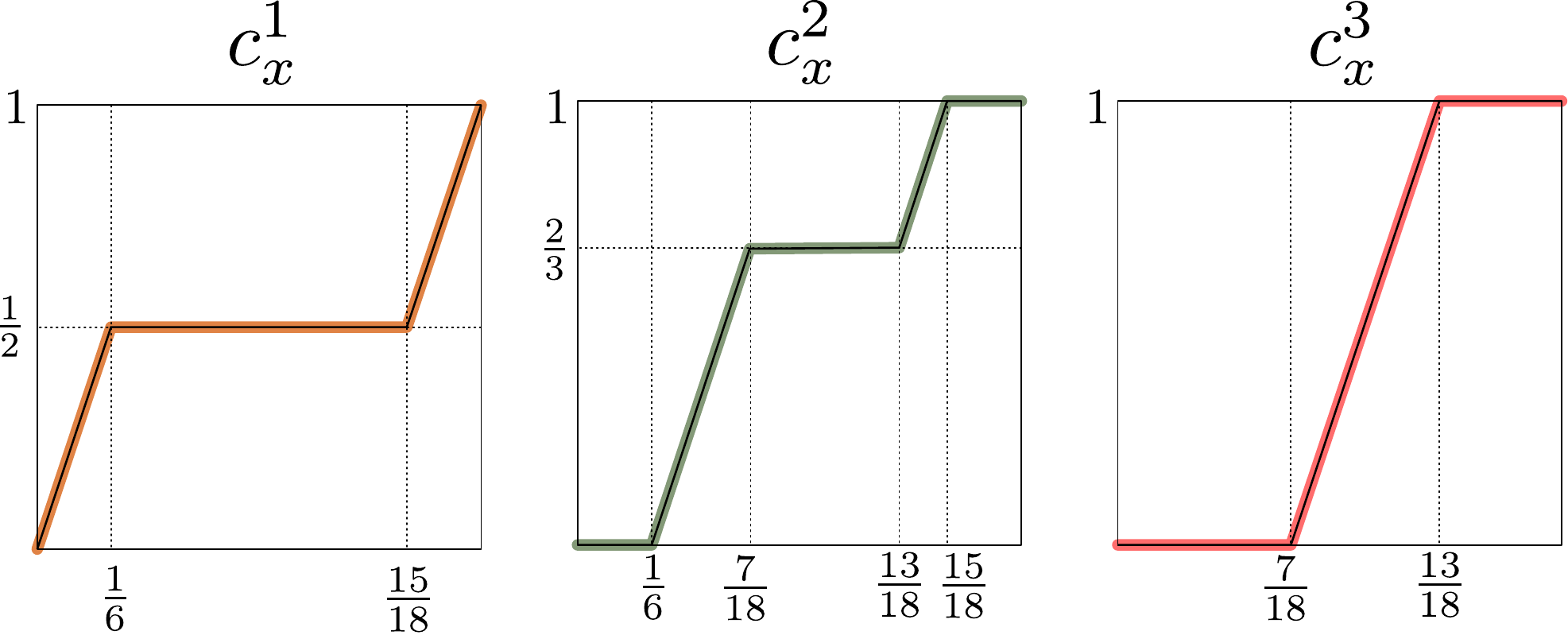}}
        \vspace{-12pt}
        \caption{Element of $x\in\mathcal{F}(3)$ and associated projections}\label{fig:F}
\end{figure}
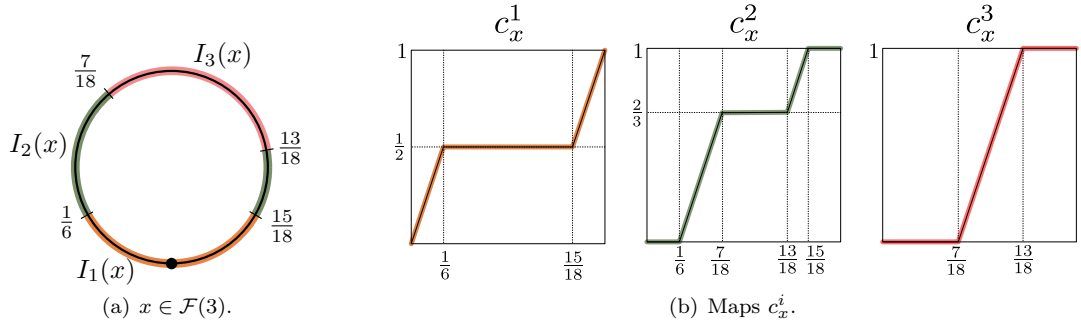

\begin{definition}\label{def: cactus map}
Given an element $x\in\mathcal{F}(k)$, we associate to each $I_j(x)$ a projection map $c_x^j:S^1\to S^1$ that takes the quotient of $S^1$ under the identification of all the points in the same path component of $\overline{S^1\setminus\mathring{I_j}}$ and then scales 
this circle by a factor of $k$. See Figure~\ref{fig:maps c-x} for an example. The \emph{cactus map} $c_x\colon S^1 \to (S^1)^k$ is the collection of maps $c_x:=(c_x^1,\dots,c_x^k)$. 
Then there is a map
\begin{eqnarray*}
\begin{tikzcd}[row sep=tiny]c:\mathcal{F}(k)\arrow[r]&
Map(S^1,(S^1)^k) \\
x \ar[r, mapsto] & c_x=(c_x^1,\dots,c_x^k)\colon S^1\to (S^1)^k.\end{tikzcd}
\end{eqnarray*}
\end{definition}

For any $x\in \mathcal{F}(k)$, we also use $x$ to denote the configuration of circles in the image of the cactus map $c_x\colon S^1\to (S^1)^k$.
Condition $(*)$ in Definition~\ref{def:Fn} guarantees that this configuration is treelike, as it forces the submanifolds $I_j(x)$ to be nested. The global basepoint of $x$ is the image of the basepoint of $S^1$ and a planar structure is induced by the orientation of the source $S^1$ (see \cite[Definition 4.2]{Sal09}). Since each part of a partition $x\in \mathcal{F}(k)$ has equal length, $x$ is a normalized cactus as shown in Figure~\ref{fig:FnCact}. This is the sketch of the proof for the next lemma.

\begin{lemma}\cite[Section 4]{Sal09} \label{lem: cact1 homeo to F}
    For each $k\ge 1$, the space $\mathcal{F}(k)$ is homeomorphic to $\cact^1(k)$. %\qed
\end{lemma}

Recall the coendomorphism operad $\CoEnd(S^1)$ from Example~\ref{Example: CoEnd}, whose underlying symmetric sequence is a collection of $\CoEnd(k)(S^1):=\Map(S^1,(S^1)^{k})$. 
% and composition induced by the composition of maps.  
We use the map $$c:\cact^1(k)\cong\mathcal{F}(k)\hookrightarrow \Map(S^1,(S^1)^{k})=\CoEnd(S^1)(k)$$ to define an embedding of symmetric sequences. 

\begin{lemma}\label{lem:cx}
    The map $c:\mathcal{F}(k)\to Map(S^1,(S^1)^k)$ is a topological embedding.
\end{lemma}

\begin{proof}
We  first check injectivity.  Given a map $c_x=(c_x^1,\dots,c_x^k)$ in the image of $c$, we can completely determine $x\in\mathcal{F}(k)$. We know that each $c_x^j$  is a ``step-map'' with linear of slope $k$ over its non-constant parts, by the definition of $c$. (See Figure~\ref{fig:maps c-x}.) Then $I_j(x)$ is precisely the subset of points of $S^1$ where the derivative $(c_x^j)'$ equals $k$.
Continuity of $c$ follows from the fact that the topology in the mapping space can be defined using the convergence metric, using likewise the metric on $S^1$. 

\end{proof}

This embedding of symmetric sequences does not extend to an embedding of operads.  As already mentioned, $\cact^1$ is not an operad and one can check that the image of $c$ is not a suboperad of $\CoEnd(S^1)$. Indeed, if we compose two elements in $\CoEnd(S^1)$ that came from elements of $\mathcal{F}$, their composition will not be in the image of any $\mathcal{F}(k)$ because all elements in the image of $\mathcal{F}(k)$ are piecewise linear graphs of slope $0$ or $k$, and this property is not preserved by the composition in $\CoEnd(S^1)$.

\begin{figure}[ht]
\centering
\begin{minipage}{0.4\textwidth}
    \centering\def\svgwidth{0.5\columnwidth}

\end{minipage}
\begin{minipage}{0.4\textwidth}
    \centering\def\svgwidth{0.5\columnwidth}
    %% Creator: Inkscape 1.0beta2 (2b71d25, 2019-12-03), www.inkscape.org
%% Accompanies image file 'cactx.pdf' (pdf, eps, ps)
\begingroup%
  \makeatletter%
  \providecommand\color[2][]{%
    \errmessage{(Inkscape) Color is used for the text in Inkscape, but the package 'color.sty' is not loaded}%
    \renewcommand\color[2][]{}%
  }%
  \providecommand\transparent[1]{%
    \errmessage{(Inkscape) Transparency is used (non-zero) for the text in Inkscape, but the package 'transparent.sty' is not loaded}%
    \renewcommand\transparent[1]{}%
  }%
  \providecommand\rotatebox[2]{#2}%
  \newcommand*\fsize{\dimexpr\f@size pt\relax}%
  \newcommand*\lineheight[1]{\fontsize{\fsize}{#1\fsize}\selectfont}%
  \ifx\svgwidth\undefined%
    \setlength{\unitlength}{177.99999619bp}%
    \ifx\svgscale\undefined%
      \relax%
    \else%
      \setlength{\unitlength}{\unitlength * \real{\svgscale}}%
    \fi%
  \else%
    \setlength{\unitlength}{\svgwidth}%
  \fi%
  \global\let\svgwidth\undefined%
  \global\let\svgscale\undefined%
  \makeatother%
  \begin{picture}(1,1.18539333)%
    \lineheight{1}%
    \setlength\tabcolsep{0pt}%
    \put(0,0){\includegraphics[width=\unitlength,page=1]{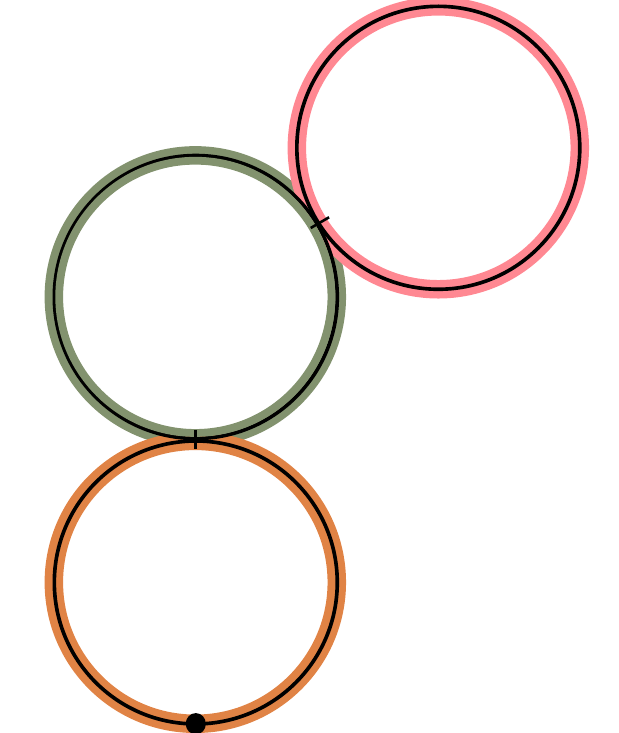}}%
    \put(0.0537004,0.18938584){\color[rgb]{0,0,0}\makebox(0,0)[rt]{\lineheight{1.25}\smash{\begin{tabular}[t]{r}$\frac{1}{2}$\end{tabular}}}}%
    \put(0.57824632,0.18938584){\color[rgb]{0,0,0}\makebox(0,0)[lt]{\lineheight{1.25}\smash{\begin{tabular}[t]{l}$\frac{1}{2}$\end{tabular}}}}%
    \put(0.04503538,0.67496348){\color[rgb]{0,0,0}\makebox(0,0)[rt]{\lineheight{1.25}\smash{\begin{tabular}[t]{r}$\frac{2}{3}$\end{tabular}}}}%
    \put(0.56671112,0.55356915){\color[rgb]{0,0,0}\makebox(0,0)[lt]{\lineheight{1.25}\smash{\begin{tabular}[t]{l}$\frac{1}{3}$\end{tabular}}}}%
    \put(0.95339134,1.00446274){\color[rgb]{0,0,0}\makebox(0,0)[lt]{\lineheight{1.25}\smash{\begin{tabular}[t]{l}$1$\end{tabular}}}}%
    \put(0.31566938,0.21155817){\color[rgb]{0,0,0}\makebox(0,0)[t]{\lineheight{1.25}\smash{\begin{tabular}[t]{c}$1$\end{tabular}}}}%
    \put(0.31556232,0.67558796){\color[rgb]{0,0,0}\makebox(0,0)[t]{\lineheight{0.5}\smash{\begin{tabular}[t]{c}$2$\end{tabular}}}}%
    \put(0.70864312,0.916027){\color[rgb]{0,0,0}\makebox(0,0)[t]{\lineheight{1.25}\smash{\begin{tabular}[t]{c}$3$\end{tabular}}}}%
  \end{picture}%
\endgroup%

\end{minipage}
\caption{An element $x$ of $\mathcal{F}(3)$ and the corresponding normalized cactus $c_x$. }\label{fig:FnCact}

\end{figure}
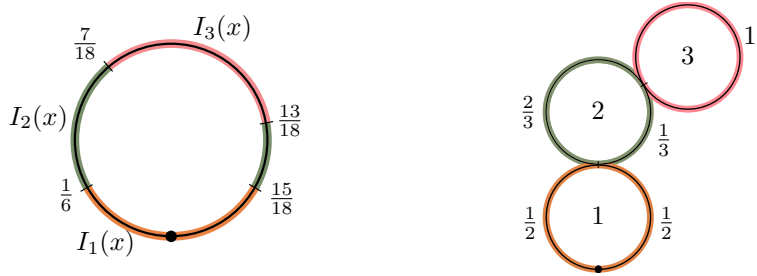

Here we define the symmetric sequence $MS^+=\{ MS^+(k)\}_{k \geq 0}$, which is built from $\mathcal{F}(k)$ and a collection $Mon^+(I,\partial I)$ of scaling maps on the interval $I$. 
It has the important property that $\cact^{1}(k)\subset MS^+(k)$ for each $k\geq 0$. 

\begin{definition}[$MS^+$ as a symmetric sequence]
    For each $k\ge 0$, we define the space $MS^+(k)$ as 
        \begin{align*}
        MS^+(0) &= *\\
        MS^+(k) &= \mathcal{F}(k) \times Mon^+(I,\partial I)
        \end{align*}
    where $Mon^+(I,\partial I)$ is the space of strictly monotone self-maps of $I$ that restrict to the identity on $\partial I$. We consider $Mon^+(I,\partial I)$ as a subspace of the space of self-maps of $S^1=I/\partial I$. For each $k$, there is an action of the symmetric group $\Sigma_k$ on $MS^+(k)$ by the reindexing of the labels of the $1$-manifolds in $\mathcal{F}(k)$.
\end{definition}

\begin{remark}The operad $MS$ that appears in \cite{MS02,MS04,Sal09} has an underlying symmetric sequence obtained by replacing $Mon^+(I,\partial I)$ by the larger space $Mon(I,\partial I)$ of weakly monotone maps. The inclusion $MS^+\hookrightarrow MS$ is a homotopy equivalence as both $Mon(I,\partial I)$ and $Mon^+(I,\partial I)$ are contractible (in fact, they are both convex). 
\end{remark}

In order to show that $MS^+$ is an operad, we start by showing that each space of operations $MS^+(k)$ embeds in $\CoEnd(S^1)(k)$. We also check that the operad composition of $\CoEnd(S^1)$ preserves the image of $MS^+$, and hence is a suitable composition for $MS^+$, thus making $MS^+$ a suboperad of $\CoEnd(S^1)$.

\begin{prop}\label{prop:MSCoend}
There is a topological embedding $\phi:MS^+(k)\to \CoEnd(S^1)(k)$ that sends $(x,f)\in MS^+(k)$ to the composite
$$S^1 \xrightarrow{f} S^1 \xrightarrow{c_x} (S^1)^k$$ where $c_x$ is the cactus map as in Definition~\ref{def: cactus map}. 
\end{prop}

A version of Proposition~\ref{prop:MSCoend} is stated for the operad 
$MS$ in \cite[Section 4]{Sal09}. As we rely heavily on this result we 
give more complete details here. 

\begin{proof}
The fact that $\phi$ is continuous follows from Lemma~\ref{lem:cx}, so we are left to check that $\phi$ is injective. Let $x\in \mathcal{F}(k)$. 
Recall that the map $c_x=(c_x^1,\dots,c_x^k)
\colon S^1\to (S^1)^k$ is a collection of ``step-maps'' of  linear of slope $k$ over its non-constant parts.
Each map $c_x^j\colon S^1\to S^1$ identifies points in the same path component of $\overline{S^1\setminus \mathring{I}_j(x)}$ and linearly takes $I_j(x)$ (of length $1/k$) to a circle of circumference $1$.
So, these maps satisfy that
$$\frac{1}{k}\sum\limits_{j=1}^k c_x^j=\mbox{Id}_{S^1}.$$
In particular, this means that if $c_x\circ f=c_y\circ g$, then 
$$f=(\frac{1}{k}\sum\limits_{j=1}^k c_x^j)\circ f=\frac{1}{k}\sum\limits_{j=1}^k (c_x^j\circ f)= \frac{1}{k}\sum\limits_{j=1}^k (c_x^j\circ g)=(\frac{1}{k}\sum\limits_{j=1}^k  c_y^j)\circ g=g.$$
Moreover, as $f,g$ are strictly monotone and hence invertible, for each $j=1,\dots,k$, 
$$c^j_x=(c^j_x\circ f) \circ f^{-1}=(c^j_y\circ g) \circ f^{-1}=(c^j_y\circ g) \circ g^{-1}=c^j_y.$$
This shows that $c_x=c_y$ and therefore the map is injective.
\end{proof}

Proposition~\ref{prop:MSCoend} shows that $MS^+$ is a symmetric subsequence of $\CoEnd$ and this next lemma shows that the operad structure maps of $\CoEnd$ preserve this structure. 
\begin{lemma}\label{lemma:composition preserves MS}
The operad structure maps of $\CoEnd$ preserve the symmetric subsequence $MS^+$. 
\end{lemma}

\begin{proof}
It suffices to consider the composition operations $\circ_i$ in $\CoEnd$ as defined in Example~\ref{Example: CoEnd}. Given $(x,f)$ and $(y,g)$ in $MS^+$, we need to check that the composition 
\begin{equation}\label{eq: CoEnd composition}
    S^1 \xrightarrow{f} S^1 \xrightarrow{c_x} (S^1)^k \xrightarrow{1\times g\times 1} (S^1)^k \xrightarrow{1\times c_y\times 1} (S^1)^{j+k-1}
\end{equation} 
is in the image of $MS^+$, where $1\times g\times 1$ denotes the map where $g$ acts only on the $i$th circle. For this, we will show two things: 
\begin{enumerate}
\item[(i)] $(1\times g\times 1)\circ c_x=c_{\tilde x}\circ \tilde g$, for some $\tilde g\in Mon^+(I,\partial I)$ and $\tilde x\in \mathcal{F}(k)$,  
\item[(ii)] $(1\times c_y\times 1)\circ c_x=c_z \circ h_{x,y}$ for some $h_{x,y}\in Mon^+(I,\partial I)$ and $z\in \mathcal{F}(j+k-1)$.  
\end{enumerate}
For statement (i), the map $(1\times g\times 1)$ acts only on the $i$th circle, so in the composition with $c_x$ it only affects points in $I_i(x)$. 
Recall that we identify $S^1$ with $I/\partial I$. If $I_i(x)=J_1\sqcup \dots \sqcup J_r$ with each $J_s$ a subinterval of $[0,1]$ and $I_i(x)$ of total length $\frac{1}{k}$, then define $I_i(\tilde x)=\tilde J_1\sqcup \dots \sqcup \tilde J_r$ for $\tilde J_s$ an interval of length  $\frac{1}{k}\ell(g(J_s))$. 
We obtain $\tilde x\in \mathcal{F}(k)$ from $x$ by replacing each subinterval $J_s$ by the interval $\tilde J_s$ and shifting each path component of $[0,1]\setminus \mathring{I}_r(x)$ accordingly. This makes sense as, by construction, the total length of $I_i(\tilde x)$ is again $\frac{1}{k}$.  
Also $\tilde g$ is defined as the
canonical identification of $I_r(x)$ with $I_r(\tilde{x})$ for all $r\in \{1, \ldots, k\}$. 
See Figure~\ref{fig:tildegx} for an example. 

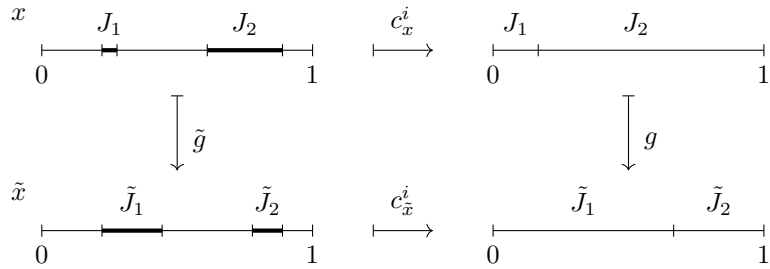
\begin{figure}[h]
    \centering
    \begin{tikzpicture} [scale=1, inner sep=2mm]
			\begin{scope}[scale=0.4]
			    \coordinate (0) at (0,1);
			    \coordinate (1) at (9, 1);
			    \draw (0)--(1);
			    
			    \foreach \x in {0, 2, 2.5, 5.5,8, 9} {
                  \draw (\x,0.8) -- (\x,1.2);
                }
                \node[below] at (0,1) {$0$};
                \node[below] at (9,1) {$1$};
			    \node[above] at (2.25, 1) {$J_1$};
			    \node[above] at (6.75, 1) {$J_2$};
			    \draw[ultra thick] (2, 1)--(2.5, 1);
			    \draw[ultra thick] (5.5, 1)--(8, 1);
			    
			    \draw[|->] (11, 1)--(13,1);
			    \node[above] at (12, 1) {$c_x^i$};
			    
			    \draw[|->] (4.5, -0.5)--(4.5,-3);
			    \node[right] at (4.5, -2) {$\tilde{g}$};
			    \node[above left] at (-0,1.5) {$x$};
			\end{scope}
			\begin{scope}[scale=0.4, xshift=0cm, yshift=-6cm]
			    \coordinate (0) at (0,1);
			    \coordinate (1) at (9, 1);
			    \draw (0)--(1);
			    
			    \foreach \x in {0, 2, 4, 7,8, 9} {
                  \draw (\x,0.8) -- (\x,1.2);
                }
                \node[below] at (0,1) {$0$};
                \node[below] at (9,1) {$1$};
			    \node[above] at (3, 1) {$\tilde{J}_1$};
			    \node[above] at (7.5, 1) {$\tilde{J}_2$};
			    \draw[ultra thick] (2, 1)--(4, 1);
			    \draw[ultra thick] (7, 1)--(8, 1);
			    
			    \draw[|->] (11, 1)--(13,1);
			    \node[above] at (12, 1) {$c_{\tilde{x}}^i$};
			    \node[above left] at (-0,1.5) {$\tilde{x}$};
			\end{scope}
			\begin{scope}[scale=0.4, xshift=15cm, yshift=0cm]
			    \coordinate (0) at (0,1);
			    \coordinate (1) at (9, 1);
			    \draw (0)--(1);
			    
			    \node[below] at (0,1) {$0$};
                \node[below] at (9,1) {$1$};
			    \foreach \x in {0, 1.5, 9} {
                  \draw (\x,0.8) -- (\x,1.2);
                }
			    \node[above] at (0.75, 1) {$J_1$};
			    \node[above] at (4.75, 1) {$J_2$};
			    
			    \draw[|->] (4.5, -0.5)--(4.5,-3);
			    \node[right] at (4.5, -2) {$g$};
			\end{scope}
			\begin{scope}[scale=0.4, xshift=15cm, yshift=-6cm]
			    \coordinate (0) at (0,1);
			    \coordinate (1) at (9, 1);
			    \draw (0)--(1);
			    
			    \node[below] at (0,1) {$0$};
                \node[below] at (9,1) {$1$};
			    \foreach \x in {0, 6, 9} {
                  \draw (\x,0.8) -- (\x,1.2);
                }
			    \node[above] at (3, 1) {$\tilde{J}_1$};
			    \node[above] at (7.5, 1) {$\tilde{J}_2$};
			\end{scope}
			\end{tikzpicture} 
    %\caption{Temporary illustration of the pair $(\tilde g,\tilde x)$}
    \caption{An example of the commutative diagram $g\circ c_x^i=c_{\tilde{x}}^i\circ \tilde{g}$}
    \label{fig:tildegx}
\end{figure}

For statement (ii), we consider a composition $(1\times c_y\times 1)\circ c_x:S^1\to (S^1)^{j+k-1}$ with $c_y$ on the $i$th position. Such a composition maps the $r$th partition $I_r(x)$, for $r\neq i$, to the $r$th (if $r<i$) or $(r+k-1)$st (if $r>i$) component in the target by a slope $k$ map, while $I_i(x)$ is mapped by  slope $jk$ maps to the remaining components. Let $h_{x,y}:S^1\to S^1$ be the rescaling map that scales each $I_r(x)$ by a factor $\frac{k}{j+k-1}$ for $r\neq i$, and $I_i(x)$ by a factor $\frac{jk}{j+k-1}$. Then the image under $h_{x,y}$ of each $I_r(x)$ will be of size $\frac{1}{j+k-1}$ for $r\neq i$, while $I_i(x)$ will have image of total size $\frac{j}{j+k-1}$. Note that this gives a well-defined map in $Mon^+(I,\partial I)$ as the sum of the length of the images $h_{x,y}(I_r(x))$ is $(k-1)\frac{1}{j+k-1}+\frac{j}{j+k-1}=1$. Subdividing the image under $h_{x,y}$ of $I_i(x)$ into $j$ parts as prescribed by $y$, together with the images of the other $I_r(x)$'s, then defines $z\in\mathcal{F}(j+k-1)$. The relation  $(1\times c_y\times 1)\circ c_x=c_z \circ h_{x,y}$ holds by construction. 
\end{proof}

Therefore we have shown that $MS^+$ is a suboperad of $\CoEnd(S^1)$ via the embedding $\phi$ in Proposition~\ref{prop:MSCoend}. 
\begin{definition}[$MS^+$ as an operad]
The symmetric sequence $MS^+=\{MS^+(k)\}_{k\in\mathbb{N}}$ becomes an operad with composition %$(x,f)\bullet_i (y,g)$ induced by \eqref{eq: MS composition}
    \begin{equation}\label{eq: MS composition}
        (x,f)\bullet_i (y,g):=\phi^{-1}(\phi(x,f)\circ_i \phi(y,g))
    \end{equation}
where $\circ_i$ is the composition in $\CoEnd(S^1)$ defined in \eqref{eq: CoEnd composition}, and the pre-image exists as a consequence of Lemma \ref{lemma:composition preserves MS}. 
\end{definition}

We will often use scaling maps  in $Mon^+(I,\partial I)$ to encode the scaling of lobes in the composition of normalized cacti.
Given a partition $x=(I_1(x),\dots,I_k(x))\in\mathcal{F}(k)\cong  \cact^1(k)$, and natural numbers $m_1,\dots,m_k\ge 0$, we let
\begin{equation}\label{def: repramaterization maps g}
g=g(x;m_1,\dots,m_k): S^1\longrightarrow S^1
\end{equation}
be the element of $Mon^+(I,\partial I)$ that scales $I_j(x)$ by the factor 
$\frac{km_j}{m_1+\dots +m_k}$, $1\leq j\leq k$. 
Each $I_j(x)$ has total length $\frac{1}{k}$, so the image of $I_j(x)$ will have length $\frac{m_j}{m_1+\dots+ m_k}$ for each $1\leq j\leq k$. Note that $g(x;1,\dots,1)=\id$ is just the identity map on $S^1$. 

\begin{figure}[ht]
    \centering\def\svgwidth{0.8\columnwidth}
    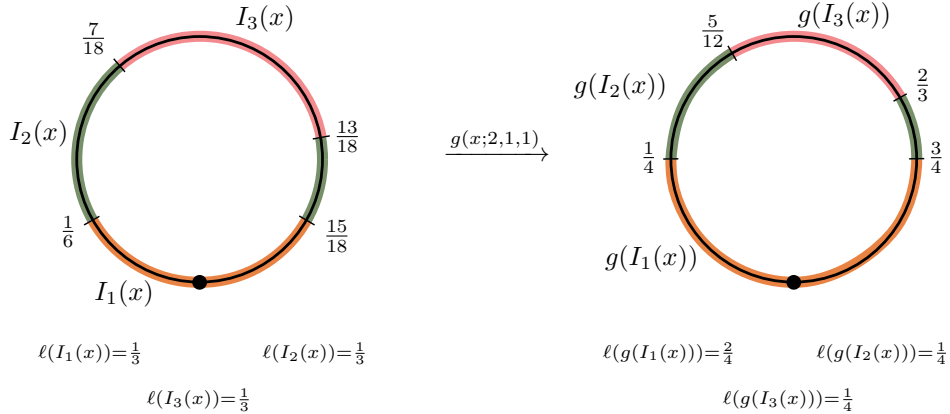
\caption{Map $g(x;2,1,1)$, for $x$ from Figure~\ref{fig:FnCact}.}\label{fig:example-g}

\end{figure}

We will now show that the $\circ_i$-compositions of normalized cacti from \eqref{eq: cact^1 composition}, \begin{equation*}
\begin{tikzcd} \circ_i:\cact^1(k)\times \cact^1(j)\arrow[r]& \cact^1(k+j-1), \end{tikzcd}
\end{equation*} 
and, more generally, the $\cact^1$--composition maps 
\[
\gamma_{\cact^1} \colon \cact^1(k)\times \cact^1(m_1)\times \cdots \times \cact^1(m_k) \longrightarrow
\cact^1(\Sigma_{i=1}^{k}m_i),
\] 
are restrictions of the corresponding compositions of appropriately chosen elements of $MS^+$.

For a collection of cacti $x\in \cact^1(k)$ and ${y_j}\in \cact^1(m_j)$, $1\leq j\leq k$, the quasi-operad composition $\gamma_{\cact^1}(x;y_1,\dots,y_k)$ scales each lobe of $x$ so that the $i$th lobe now has length $m_i$, and then inserts (without any further scaling) each ${y_i}$ in place of the $i$th scaled lobe. 
%We also use $\gamma_{MS^+}$ to denote the composition in $MS^+$.

Under the homeomorphism $\cact^1(k)\cong \F(k)$ in Lemma~\ref{lem: cact1 homeo to F}, a normalized cactus $x\in \cact^1(k)$ precisely corresponds to a partition $x\in \mathcal{F}(k)$ of $[0,1]$ into $k$ submanifolds $I_j(x)$ of equal lengths $\frac{1}{k}$, satisfying the conditions of Definition~\ref{def:Fn}. 
Since the $i$th lobe of $x$ corresponds to the submanifold $I_i(x)$, there is a scaling by $\frac{1}{k}$ in the identification $\cact^1(k)$ to take a lobe of length $1$ to $I_i(x)$.
In the next lemma we will use $x\in \F(k)$ and $y_j\in \F(m_j)$ for $1\leq j\leq k$ to represent a sequence of cacti in $\cact^1(k)$ and $\cact^1(m_j)$ respectively.
We will still denote the composition by $\circ_i$ or $\gamma_{\cact^1}$.

\begin{lemma}\label{lem:rescaling}
Let $\gamma_{MS^+}$ and $\gamma_{\cact^1}$ denote the (quasi-)operad compositions in  $MS^{+}$ and $\cact^1$, respectively. Then for 
%$x\in \cact^1(k)$ and $y_j\in \cact^1(m_j)$,
$x\in \F(k)$ and $y_j\in \F(m_j)$,
with $1\le j\le k$, we have 
$$\gamma_{MS^+}((x,g^{-1}(x;m_1,\dots,m_k));(y_1,id),\dots,(y_k,id))=(\gamma_{\cact^1}(x;y_1,\dots,y_k),id)$$
in $MS^+(\sum m_j)$. 
In  particular, 
$$(x,g^{-1}(x;1,\dots,m_i,\dots,1))\bullet_i (y_i,\id)=(x \circ_i y_i,\id)$$ 
where $\bullet_i$ denotes the composition \eqref{eq: MS composition} of $MS^+$ and $\circ_i$ represents the composition \eqref{eq: cact^1 composition} of $\cact^1$. 
\end{lemma}

\begin{proof}

Let $\F_{(m_1,\dots,m_k)}(k)$ denote the scaled version of $\F(k)$ where the $i$th partition, $I_i$, now has length $m_i$ instead of $\frac{1}{k}$. In particular, we have that $\F(k)=\F_{\left(\frac{1}{k},\dots,\frac{1}{k}\right)}(k)$ and $\cact^1(k)=\F_{(1,\dots,1)}(k)$. Thus the homeomorphism $\cact^1(k)\cong \F(k)$ implies that the composition $\gamma_{\cact^1}$ on $\cact^1$ can be interpreted as a map in $\F$, written as

\begin{align*}%\label{equ:cact1comp}
    \F(k)\times 
    &
    \left(
    \F(m_1)\times \dots \times \F(m_k)
    \right) \\
    &\xrightarrow{\ S\ } \ \F_{(m_1,\dots,m_k)}(k)\times \left(\F_{(1,\dots,1)}(m_1)\times \dots\times \F_{(1,\dots,1)}(m_k)\right) \\
    &\xrightarrow{\ \gamma_{}\ } \ \F_{(1,\dots,1)}(\sum_i m_i) \\
    &\xrightarrow{\ N\ }\  \F_{\frac{1}{\sum m_i},\dots,\frac{1}{\sum m_i}}(\sum_i m_i)=\F(\sum_i m_i) 
\end{align*}
where $S$ and $N$ are scaling and normalising maps and  the map labelled $\gamma_{}$ is the insertion map.

Our task is to write the composition $\gamma_{\cact^1}$ in terms of the operad $MS^+$. To do this, we will use %replace  these
scaling maps %by maps 
inside $Mon^+(I,\del I)$. 
More precisely, define a map
$$\F(k)\times   \left(
    \F(m_1)\times \dots \times \F(m_k)
    \right) %\cact^1(k)\times (\cact^1(m_1)\times \dots\times \cact^1(m_k))
     \xrightarrow{\hspace{7pt} G\hspace{8pt} } MS^+(k)\times (MS^+(m_1)\times \dots\times  MS^+(m_k))$$ that takes $(x;y_1,\dots,y_k)$ to $((x,g^{-1});(y_1,\id),\dots,(y_k,\id))$ where $g=g(x;m_1,\dots,m_k)\in Mon^+(I,\del I)$ is the map in equation~\eqref{def: repramaterization maps g}.
The statement we want to prove is that $(\gamma_{\cact^1}, \id)$ can be written as the composition 
\begin{align*}
    \F(k)\times %& 
    \left(\F(m_1)\times \dots \times \F(m_k)\right)
    %\\&\longrightarrow\F(k)\times (\F(m_1)\times \dots\times  \F(m_k)) 
%    \\[0.5em]
    & \xrightarrow{\hspace{7pt}G\hspace{8pt}} MS^+(k)\times (MS^+(m_1)\times \dots\times  MS^+(m_k)) \\[0.5em]
    & \xrightarrow{\gamma_{MS^+}} MS^+(\sum_i m_i).
\end{align*}
In particular,
%and where part of the 
we claim that the resulting element of $MS^+(\sum_i m_i) $ is in the image of $\cact^1$, that is, of the form $(z,\id)$. 

To prove this, we start by expressing $G$ as a composition $G'\circ S$, where $S$ is the scaling map in the description %(\ref{equ:cact1comp}) 
$\gamma_{\cact^1}=N\circ \gamma_{}\circ S$ given above and 
\begin{multline*}
    G'\colon \F_{(m_1,\dots,m_k)}(k)\times (\F_{(1,\dots,1)}(m_1)\times \dots\times \F_{(1,\dots,1)}(m_k))  \\ \longrightarrow MS^+(k)\times (MS^+(m_1)\times \dots\times MS^+(m_k))
\end{multline*}
is the map that takes a tuple $(x;y_1,\dots, y_k)$ to the tuple  $((N(x),g^{-1});(N(y_1),\id),\dots, (N(y_k),\id))$, with $N$ the normalization also as above. 
In order to compare $\gamma_{\cact^1}=N\circ \gamma_{}\circ S$ with $\gamma_{MS^+}\circ G' \circ S$, we have to show that %$\gamma_{MS^+}\circ G'=N\circ \gamma_{\F}$, i.e.~that 
the diagram 
$$\xymatrix{\F_{(m_1,\dots,m_k)}(k)\times (\F_{(1,\dots,1)}(m_1)\times \dots\times \F_{(1,\dots,1)}(m_k)) \ar[d]_{G'}
     \ar[r]^-{\gamma_{}} &\F_{(1,\dots,1)}(\sum_i m_i) \ar[d]^{(N, id)} \\
    MS^+(k)\times (MS^+(m_1)\times \dots\times MS^+(m_k)) \ar[r]^-{\gamma_{MS^+}}& MS^+(\sum_i m_i)   }$$
commutes, where the right vertical map takes $z$ to $(N(z),\id)$. 
To see this, let $(x;y_1,\dots,y_k)$ be an element in the top left corner of the square. Its image 
$\gamma_{MS^+}\circ G'(x;y_1,\dots,y_k)$ along the bottom composition is the element of 
$MS^+(m_1+\dots+m_k)$
%, considered as a subspace of $\CoEnd(m_1+\dots+m_k)$, 
given by the following composition: 
$$S^1\xrightarrow{g^{-1}}S^1\xrightarrow{c_x} (S^1)^k \xrightarrow{c_{y_1}\times \dots\times  c_{y_k}} (S^1)^{m_1+\dots+m_k}$$
since we consider $MS^+(m_1+\dots+m_k)$ as a subspace of $\CoEnd(m_1+\dots+m_k)$ and use the composition in 
 \eqref{eq: MS composition}.
The $j$th factor $S^1$ in the above $(S^1)^k$ is subdivided into submanifolds $I_s(y_j)$ according to $c_{y_j}$.  

Their inverse image 
$g\circ (c_x)_j^{-1}(I_s(y_j))$ in the source $S^1$ of the composition is thus taken to the $(m_1+\dots+m_{j-1}+s)$th factor $S^1$ in $(S^1)^{m_1+\dots+m_k}$, being first scaled by a factor $\frac{\sum m_i}{km_j}$ (using $g^{-1}$), then by a factor $k$ (via the $j$th component of $c_x$) and finally by a factor $m_j$ (via the $s$th component of $c_{y_j}$). So in total the composition takes $g\circ (c_x)_j^{-1}(I_s(y_j))$ to $S^1=I/\del I$ linearly by a factor $\sum m_i$, and is constant on the connected components of the complement of $g\circ (c_x)_j^{-1}(I_s(y_j))$. In particular, $g\circ (c_x)_j^{-1}(I_s(y_j))$ has length $\frac{1}{\sum m_i}$, which is independent of  $j$ and $s$. Thus we see that the resulting element does indeed live in the image of $\cact^1$.
As the scaling is always independent of $s$ and $j$, the proportion of each $g\circ (c_x)_j^{-1}(I_s(y_j))$ inside the source $S^1$
%$I_s(y_j)$ inserted this way inside its corresponding $I_j(x)$ 
is always as dictated by $c_{y_j}$, with each $g^{-1}(I_j(x))$ having total length $\frac{m_j}{\sum_i m_i}$. Hence the composition is the same as following the other side of the square, which inserts $I_s(y_j)$ inside $I_j(x)$, scaling each $I_j(x)$ to length $m_j$, then scales it by $\frac{1}{\sum m_i}$ to be inside $MS^+$. 

\end{proof}

Therefore up to scaling in accordance with the homeomorphism $\cact^1\cong \F$ in Lemma~\ref{lem: cact1 homeo to F}, we have shown that both $\cact^1$ and its composition are contained within the operad $MS^+$, but not as a suboperad.

\subsection{\texorpdfstring{$\cact^1$}{Lg} is a \texorpdfstring{$B\OO$}{Lg}--algebra} \label{sec: Cact1 is BO}
Here we will construct an action of $B\OO$ on normalized cacti using the fact that $\cact^1\subset MS^+$, and that its composition can also be described in terms of the composition in $MS^+$. 
%, to define an action of $B\OO$ on normalized cacti.
In Theorem~\ref{thm: cact is BO alg}, we show that the quasi-operad structure on normalized cacti $\cact^1=\{\cact^1(k)\}_{k\geq 0}$ is part of a $B\mathcal{O}$-algebra structure. 
In Corollary~\ref{cor: cact is dendroidal segal space}, we conclude that $\cact^1$ determines a dendroidal Segal space $X\in\mathcal{S}^{\om^{op}}$ with $X(C_v)=\cact^1(|v|)$.

\medskip

Recall from Section~\ref{sec: BO-alg} that a $B\OO$--algebra is a symmetric sequence with $\circ_i$--operations that are homotopy associative up to all higher homotopies. Elements from $B\OO$ are $(T, \sigma, \tau, B, t)$ where $T$ is a planar tree equipped with bijections $\sigma:|V(T)|\rightarrow V(T)$ and $\tau:|L(T)|\rightarrow L(T)$, and $(B, t)$ is a weighted bracketing of $T$.

Let 
\begin{equation}\label{eq: reparameterization map R}
    \mathcal{R}\colon MS^{+}\longrightarrow \F
\end{equation}
denote the projection map 
%\abinote{This should really be $\F$, not $\cact^1$, right?}
that forgets the $Mon^+(I,\del I)$ component, $\mathcal{R}(x,f)=x$. 
This is a map of symmetric sequences. 
%An application of this is a {\em renormalization}, in the sense that, 
If we think of elements of $MS^+$ as cacti, the map $\mathcal{R}$ has the effect of {\em renormalizing}, that is, rescaling the lobes so that they all have the same length.
Since $MS^+$ is an operad, it is an $\OO$-algebra. The $\OO$-action 
$$\lambda_{MS^+}: \OO(k;m_1,\dots,m_k)\times MS^+(m_1)\times\ldots\times MS^+(m_k)\longrightarrow MS^+(\sum_i m_i)$$
takes a sequence of elements 
\[((T,\sigma,\tau),(x_1,f_1),\ldots,(x_k,f_k))\in \OO(k;m_1,\dots,m_k)\times MS^+(m_1)\times\ldots\times MS^+(m_k)\] 
to %an element 
%$$\lambda_{MS^+}\big((T,\sigma,\tau);(x_1,f_1),\ldots,(x_k,f_k))\in MS^+(\sum_i m_i)\big)$$ which is 
the composition of the elements $(x_1,f_1),\ldots,(x_k,f_k)$ according to $\gamma_{MS^+}$, in the order prescribed by the labeled tree $(T,\sigma)$, acting by the permutation $\tau$ on the resulting element of $MS^+(\sum_i m_i)$. This composition can be depicted by labeling the $i$th vertex of $(T,\sigma,\tau)$ by $(x_i,f_i) \in MS^+(m_i)$. 
This action is compatible with the composition in $\OO$ because $MS^+$ is an operad. 
We will use this existing $\OO$-algebra structure to define the $B\OO$-algebra structure of $\cact^1$ by representing the $\cact^1$-composition by  $\mathcal{R}\circ \lambda_{MS^+}$.

\begin{theorem}\label{thm: cact is BO alg}
The $\cact^1$-composition \eqref{eq: cact^1 composition} 
is part of a $B\OO$--algebra structure. 
\end{theorem}

\begin{proof}
In order to construct a $B\OO$--algebra structure on the sequence $\{\cact^1(n)\}_{n\ge 0}$, we want to define a map
$$B\OO(k;m_1,\dots,m_k)\times \cact^1(m_1)\times\ldots\times \cact^1(m_k)\longrightarrow \cact^1(\sum_i m_i)$$
that restricts to the $\Sigma_n$--action on $\cact^1(n)$, which permutes the labels on the lobes,  and its already defined $\circ_i$--compositions. Using the homeomorphism $\cact^1\cong \F$ from Lemma~\ref{lem: cact1 homeo to F}, we will equivalently construct a map
$$\lambda: B\OO(k;m_1,\dots,m_k)\times \F(m_1)\times\ldots\times \F(m_k)\longrightarrow \F(\sum_i m_i).$$

Firstly, the $\Sigma_n$-action on the $n$--space of a $B\OO$--algebra is encoded  by the labeled corollas 
$$(C_n,1,\tau,\emptyset,\emptyset)\in B\OO(n;n)\cong \OO(n;n)\cong \Sigma_n,$$ 
where $\tau$ labels the leaves of the corollas $C_n$, which are thought of as elements of the symmetric group $\Sigma_n$, and the identity corresponds to the planar ordering. This fixes the action of such elements of $B\OO$ as we have already fixed the $\Sigma_n$--action on $\cact^1(n)$.

The $\cact^1$ $\circ_i$-composition is encoded in $B\OO$ by the trees with exactly two vertices, one attached to the $i$th incoming edge of the other. These trees admit no non-trivial bracketings so such elements of $B\OO$ have the form 
$$(T,\sigma,\tau,\emptyset,\emptyset)\in B\OO(m+n-1;m,n)$$ 
where $\sigma$ labels the two vertices of $T$ and $\tau$ labels its $n+m-1$ leaves. 
The compatibility with the pre-chosen operadic composition of $\cact^1$ dictates the action of such elements of $B\OO$: $(T,\sigma,\tau,\emptyset,\emptyset)$ acts on $x_1\in \cact^1(m)$ and $x_2\in \cact^1(n)$ by taking their $\circ_i$-composition, as dictated by the tree, and then acting by $\tau$ on the lobes of the resulting element of $\cact^1(m+n-1)$. 
Figure~\ref{fig: BO action on cact1 composition} illustrates an example of this action.
\begin{figure}[ht]
    \centering\def\svgwidth{0.85\columnwidth}
    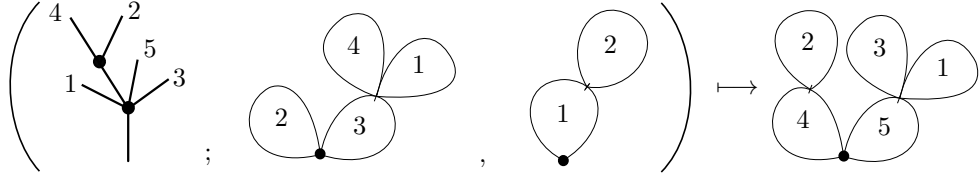
    %\centering
    %\includegraphics[width=0.8\textwidth]{Images/ExampleBOaction.jpeg}
    \caption{Example of the $B\OO$-action on $\cact^1$.}
    \label{fig: BO action on cact1 composition}
\end{figure}

By Lemma~\ref{lem:rescaling}, this $\cact^1$-composition can be defined in terms of the $MS^+$ composition:
$$\mathcal{R}\circ \lambda_{MS^+}\big((T,\sigma,\tau);(x_1,g_1),(x_2,g_2)\big)$$
where $\mathcal{R}$ is the projection map \eqref{eq: reparameterization map R}, and $g_1=g(x_1,1,\dots,k_i, \dots, 1)$ and $g_2=g(x_2,1,\dots,l_j, \dots, 1)$ are the rescaling maps of \eqref{def: repramaterization maps g}, with $k_i=n$ and $l_j=1$ if  first vertex is the bottom vertex and the second is attached to its $i$th input, or $k_i=1$ and $l_j=m$ if the second vertex is the bottom vertex with the first attached to its $j$th input. 
Let $(y_T, f_T)\in MS^+$ denote the element $\lambda_{MS^+}\big((T,\sigma,\tau);(x_1,g_1),(x_2,g_2)\big)$.
\medskip

We will now extend this definition of the $B\OO$--action of trees with at most two vertices to an action of the whole operad. We start by defining an explicit expression for the action of bracketings of trees $(T,\sigma,\tau,B,\underline 1)$ with brackets of weight $1$, and afterwards extend this definition to the remaining elements of $B\OO$, whose brackets have weight strictly between 0 and 1.  

Let ${\bf T}=(T,\sigma,\tau, B,\underline 1)$ be an element of $B\OO(n; m_1,\ldots,m_k)$ with all brackets of weight 1, and 
let $x_i\in \F(m_i)\cong \cact^1(m_i)$ for each $1\leq i \leq k$. 
We first construct scaling maps $g\in Mon^+(I,\partial I)$ as in \eqref{def: repramaterization maps g}.
Recall from Definition~\ref{def:bracketing} that a bracketing $B=\{S_j\}_{j\in J}$ consists of large, nested proper subtrees of $T$. Here we allow $B$ to be empty. %, and we can assume that $B$ is non-empty.
Recall that $\sigma$ orders the vertices of $T$. %Let $v_i\in V(T)$ be the $i$th vertex according to this ordering and 
For a fixed $i\in \{1, \ldots, k\}$, let $S\in B$ be the smallest bracket that contains the vertex $\sigma(i)$, allowing $S=T$ if there are no such bracket. 
Recall that $in(\sigma(i))$ is the set of incoming edges of $\sigma(i)$, and $L(S)$ is the set of leaves of the bracket $S$. 
We define a map 
\begin{equation}\label{eq: xi}\xi:in(\sigma(i))\longrightarrow \mathbb{N}\end{equation}
by setting 
\begin{enumerate}
\item[(i)]  $\xi(e)=1$ if $e\in L(S)$;
\item[(ii)]  $\xi(e)=|L(S')|$ if $e$ is the root of a bracket $S'\subset S$ in $B$, with $S'\subset S$ the largest such bracket; 
\item[(iii)] $\xi(e)=|w|$ if $e\in \ie(S)$ is not the root of any $S'\in B$, where $|w|$ denotes the arity of the vertex $w\in V(S)$ for which $e$ is the outgoing edge. 
\end{enumerate}
\begin{figure}[ht]
\centering
\begin{minipage}{0.6\textwidth}
    \centering\def\svgwidth{0.7\columnwidth}
    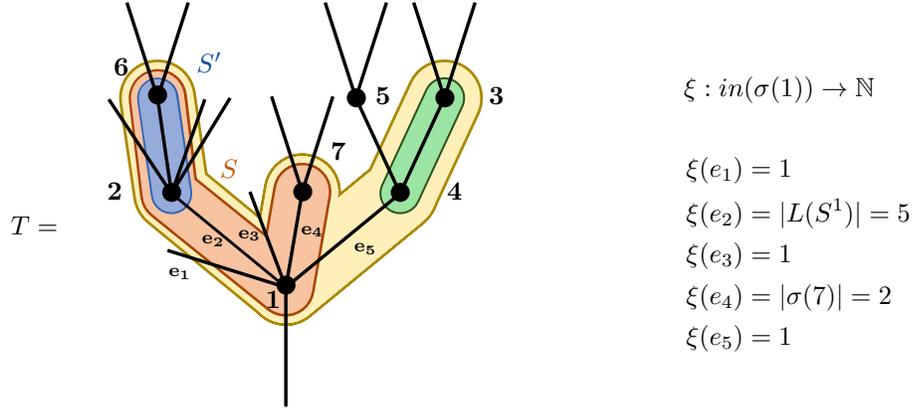
\end{minipage}
\begin{minipage}{0.3\textwidth}
%     \begin{align*}
%       \xi(e_1) &= 4     &    \xi(e_2) &=2     \\
%       \xi(e_3) &=4      &    \xi(e_4) &=3     \\
%       \xi(e_5) &=4      &    \xi(e_6) &=3     \\
%       \xi(e_7) &=14      &             &
%   \end{align*}
    \begin{align*}
      \xi:in&(\sigma(1)) \to\mathbb{N}\\
      &     \\
      \xi(e_1) &= 1     \\
      \xi(e_2) &= |L(S^1)|=5      \\
      \xi(e_3) &=1      \\    
      \xi(e_4) &= |\sigma(7)|=2      \\
      \xi(e_5) &=1      
  \end{align*}
\end{minipage}
\caption{An example of $\xi$.}\label{fig:example of xi}
\end{figure}

Figure~\ref{fig:example of xi} shows an example of the map $\xi$.
We then set 
\begin{equation}\label{eq: g_i definition in BO action}
    g_i:=g(x_i;\xi(e_1),\dots,\xi(e_{m_i}))
\end{equation} 
for $e_1,\dots, e_m$ the incoming edges of $\sigma(i)$ ordered by the planar ordering of $T$. 

\medskip

We define the action of $B\OO$ inductively on the size of the bracketing $B$. 

If $B$ is empty, then we define 
$$\lambda({\bf T};x_1,\dots,x_k):=\mathcal{R}\circ \lambda_{MS^+}((T,\sigma,\tau);(x_1,g_1),\dots,(x_k,g_k))$$
and use $(y_T, f_T)\in MS^+$ to denote the image of $\lambda_{MS^+}$.
Note that when $k=1$ or $2$, this is the same as the $B\OO$-structure already defined above. %$\cact^1$ composition above. 

If $B$ is not empty, then we define additional scaling maps for each bracket, using the inductive hypothesis that the action has already been defined action on subtrees with fewer brackets. 

Let $T'$ be the tree obtained from $T$ by adding a binary vertex at the root of each bracket $S_j\in B$. Extend the order $\sigma$ of the vertices of $T$ to an order $\sigma'$ of vertices of $T'$ by setting the $|J|=|B|$ new vertices last.  
An example of $T'$ is shown in Figure~\ref{fig:example of T'}.
We will use each additional vertex of $T'$ to assign a scaling map to the associated bracket.

\begin{figure}[ht]
    \centering\def\svgwidth{0.4\columnwidth}
    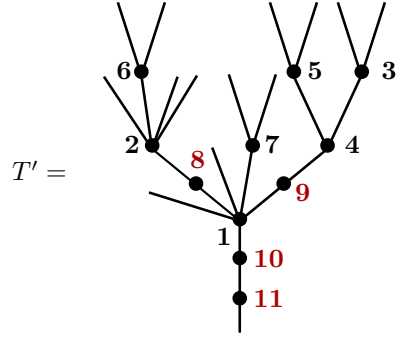
    \caption{An example of $T'$ for the bracketing of $T$ in Figure~\ref{fig:example of xi}}
    \label{fig:example of T'}
\end{figure}

Let~$w_j\in V(T')\setminus V(T)$ be the $j$th vertex of $T'$ not in $T$, according to the chosen order $\sigma'$.
%Suppose that the bracket $S_j\in B$ associated to $w_j$ contains $k_j$ vertices, and let $B_j\subset B$ be the subset of brackets that are contained within $S_j$. Then let $(y_j, f_j)\in MS^+$ be \[\lambda_{MS^+}\big((S_j,\sigma|_{S_j},\tau|_{S_j});(x_1,g_1),\ldots,(x_{k_j},g_{k_j}),(1,f_1),\dots,(1,f_{|B_j|})\big)\]
Let $S_j\in B$ be the bracket associated to $w_j$. Since the number of brackets of $B$ that lie inside $S_j$ is less than $|B|$, we have an element $$(y_{S_j}, f_{S_j})\in MS^+$$ defined by the inductive assumption by restricting ${\bf T}=(T,\sigma,\tau, B,\underline 1)$ to the subtree $S_j$. 
Consider the tree $T/S_j$ in which  all vertices in $S_j$ are identified and internal edges between them are collapsed. The tree $T/S_j$ has a vertex $[S_j]$ associated to the collapsed tree $S_j$. 
We have an induced bracketing $\tilde B$ of $T/S_j$ from the bracketing $B$ of $T$, and thus can define a map $\xi_j:in([S_j])\to \mathbb{N}$ as in $\eqref{eq: xi}$ by replacing $(T,B)$ by $(T/S_j,\tilde B)$.
Then define 
\begin{equation}\label{eq: f_i definition in BO action}
    h_j:=g(y_{S_j};\xi_j(e_1),\dots,\xi_j(e_l))\circ f_{S_j}^{-1}
\end{equation}
for $e_1,\dots,e_l$ the incoming edges of $[S_j]$ in $T/S_j$.
We define the action of $B\OO$ by setting
\begin{equation}\label{equ:BOaction}
\lambda({\bf T};x_1,\dots,x_k):= 
\mathcal{R}\circ \lambda_{MS^+}\big((T',\sigma',\tau);(x_1,g_1),\ldots,(x_k,g_k),(1,h_1),\dots,(1,h_{|B|})\big)
\end{equation}
for the rescaling maps $g_i$ and $h_j$ defined above.

We claim that the formula for the action (\ref{equ:BOaction}) is indeed compatible with composition of bracketings of trees of weight 1. %this type of element. 
It is enough to check this for a $\circ_i$-composition in $B\OO$, so consider ${\bf T_1}=(T_1,\sigma_1,\tau_1, B_1,\underline 1)$ and ${\bf T_2}=(T_2,\sigma_2,\tau_2, B_2,\underline 1)$ in $B\OO$. % where $|B_1|=|J_1|$ and $|B_2|=|J_2|$. 
We need to check that 
\begin{equation}\label{equ:comp}
\lambda({\bf T_1};x_1,\dots,x_{i-1},\lambda({\bf T_2};x_i,\dots,x_{i+l-1}), x_{i+1}, \dots,x_{k+l-1})
= \lambda({\bf T_1}\circ_i {\bf T_2};x_1,\dots,x_{k+l-1}).
\end{equation}
From the above definition, we have 
%\begin{align*}
    $\lambda({\bf T_1}; x_1,\dots,x_{i-1},\lambda({\bf T_2};x_i,\dots,x_{i+l-1}), x_{i+1}, \dots,x_{k+l-1})=$
\begin{multline*} \mathcal{R}\circ \lambda_{MS^+}\big((T_1',\sigma_1',\tau_1);(x_1,g_1),\ldots,(x_{i-1},g_{i-1}), \\
 (y_{T_2},g_i),(x_{i+l},g_{i+l}), \dots,(x_k,g_k),(1,h_1),\dots,(1,h_{|B_1|})\big)\end{multline*}
%\end{align*}
for 
$$y_{T_2}=\mathcal{R}\circ \lambda_{MS^+}\big((T'_2,\sigma'_2,\tau_2);(x_i,g'_1),\ldots,(x_{i+l-1},g'_{l}),(1,h'_1),\dots,(1,h'_{|B_2|})\big)$$
where the maps $g_i$ and $h_i$ are those associated to $(T_1,B_1)$ and the maps $g'_i$ and $h'_i$ associated to $(T_2,B_2)$. %\nwnote{added some details}
In the above notation, we also have 
$$(y_{T_2},f_{T_2})=\lambda_{MS^+}\big((T'_2,\sigma'_2,\tau_2);(x_i,g'_i),\ldots,(x_{i+l-1},g'_{i+l-1}),(1,h'_1),\dots,(1,h'_{|B_2|})\big).$$

Note that one can change the $Mon^+(I,\del I)$ component of an element of $MS^+$ by doing a $\circ_1$--composition in the operad. In particular, 
$$(y_{T_2},g_i)=(1,g_i\circ f_{T_2}^{-1})\circ_1 (y_{T_2},f_{T_2})$$ in $MS^+$ ane we can rewrite the left hand side of (\ref{equ:comp}) as the first component of the $MS^+$--composition
\begin{align*}
 \lambda_{MS^+}&\big((T_1',\sigma_1',\tau_1)\circ_i (\bar T_2',\bar \sigma_2',\tau_2);(x_1,g_1),\ldots,(x_{i-1},g_{i-1}), \\
& (x_i,g'_i),\ldots,(x_{i+l-1},g'_{i+l-1}),(1,h'_1),\dots,(1,h'_{|B_2|}),(1,g_i\circ f_{T_2}^{-1}) ,\\
& (x_{i+l},g_{i+l}), \dots,(x_k,g_k),(1,h_1),\dots,(1,h_{|B_1|})\big)
\end{align*}
where $\bar T_2'$ has an extra vertex at the bottom of the tree to encode the change of $Mon^+(I,\del I)$--component for the $T_2$ composition. 
%instead of $(y_{T_2},\id)$ , %the scaled version of $(y_{T_2},f_{T_2})$ that appears in the computation of the left hand side of (\ref{equ:comp}), 
%we need to further compose with $(1,g_i)$ in $MS^+$. 
If $T_2$ is large, this extra vertex corresponds exactly to the extra bracket  $T_2$ arising in the $B\OO$--composition, and one checks that the corresponding scaling map $h$ defined by the formula (\ref{eq: f_i definition in BO action}) is precisely the map $g_i\circ f_{T_2}^{-1}$. As the other labels of the vertices of the composed tree agree with those of the right hand side, we see that we recover the right hand side of \eqref{equ:comp}. %by the additional vertex in $(T_1\circ_i T_2)'$ associated to the bracket $T_2$ if $T_2$ is large. 
If $T_2$ is not large, then there is no such additional bracket in the $B\OO$--composition, but in this case $f_{T_2}=\id$ and the left and right hand side agree directly.

\smallskip

Recall from Remark~\ref{rem:BO as poset operad} 
that we may consider $B\OO$ as the geometric realization of the simplicial operad of bracket trees. Then the above definition of $\lambda$ on bracketings of weight $1$ defines the action of the vertices of $B\OO$. 
We finally extend this action to all bracketings of a tree $T$
%to the simplicial subdivision of the cube determined by $(B,t)$ of $B\OO$ 
by linear interpolation on the rescaling maps $g_i$.
For a fixed tree $T$ and a point $((B_0\subset\dots\subset B_r),\underline t)$  in the realization of the poset $\mathcal{B}(T)$, 
let $g_i(T,B_j)$ denote the definition of the rescaling map $g_i$ with respect to the bracketing $B_j$ on $T$ in \eqref{eq: g_i definition in BO action}, and likewise for the maps $f_j$ in \eqref{eq: f_i definition in BO action}.
We set
\[g_i=t_0g_i(T,B_0)+...+t_rg_i(T,B_r).\] This is well-defined as $Mon^+(I,\del I)$ is convex. 
Also note that this is continuous in $B\OO$ as going to the $l$th face of the simplex $(B_0\subset\dots\subset B_r)$ corresponds to $t_l$ going to 0, that is, dropping the bracket $B_{l}$.%, with the corresponding $t_l$ going to 0.
%\abinote{Added this last line}
Then we define $(y_T, f_T)$ and $\lambda(T, \sigma, \tau, B, t):=\mathcal{R}(y_T, f_T)$ as in \eqref{equ:BOaction} but with this definition of $g_i$ instead.

\smallskip

This defines the action of $B\OO$ on $\cact^1$. It is compatible under composition because the composition in $B\OO$ is the realization of the composition in the poset operad, and we have already checked the compatibility under composition there.  

\end{proof}

\medskip

Given that normalized cacti, together with the cactus composition \eqref{eq: cact^1 composition}, forms a $B\OO$-algebra we can now use the rectification results from Proposition~\ref{prop:homotopy Segal implies Segal} to define an $\infty$-operad. 

\begin{cor}\label{cor: cact is dendroidal segal space}
%\nwnote{fixed the statement}
Normalized cacti define dendroidal spaces of the following two flavors: 
\begin{enumerate} 
\item[(i)] There exists a reduced homotopy dendroidal space $X\in\mathcal{S}^{\widetilde\om_0^{op}}$, satisfying the strict Segal condition, such that $X(C_n)=\cact^1(n)$ and with value on the inner face maps $\del_e$ given by the $\cact^1$--composition. 

\item[(ii)] There exists reduced dendroidal space $Y\in\mathcal{S}^{\om^{op}}$, satisfying the weak Segal condition,  such that $Y(C_n)\simeq\cact^1(n)$ and with value on the inner face maps $\del_e$ homotopic to the $\cact^1$--composition.
\end{enumerate}
\end{cor}

\begin{proof}
Theorem~\ref{thm: cact is BO alg} shows that $\cact^1$ is a $B\OO$-algebra. Applying the construction from Theorem~\ref{thm:W_0O-Omega_0}, we define a homotopy dendroidal space $X:=\Phi(\cact^1)\in\mathcal{S}^{\widetilde\om_0^{op}}$. By construction, 
$\Phi(\cact^1)(C_n)=\cact^1(n)$, and by the theorem it is a reduced homotopy dendroidal space satisfying the strict Segal condition. The evaluation of $\Phi(\cact^1)$ on an inner edge is the $\circ_i$ composition, as encoded by the $B\OO$-structure, which in the present case is the $\cact^1$--composition by Theorem~\ref{thm: cact is BO alg}. This proves (i) in the statement.

For (ii), we set $Y:=(p_*X)_f= (p_*\Phi(\cact^1))_f\in \mathcal{S}^{\om^{op}}$ to be the rectification of $X$, as constructed in Proposition~\ref{prop:homotopy Segal implies Segal}.  By Lemma~\ref{the-natural-transformations}, $Y(C_w)=(p_*\Phi(\cact^1))_f(C_w)\simeq \Phi(\cact^1)(C_w)=\cact^1(|w|)$ and the value of $Y$ on inner face maps is identifies under these homotopy equivalences with the value of $X$ on inner face maps, and hence identifies with the $\cact^1$--composition. The $\widetilde\om_0$-diagram $X$ takes values in the category of topological spaces and is therefore fibrant as it is entrywise fibrant. We can now apply Proposition~\ref{prop:homotopy Segal implies Segal},
to get that $Y$ is reduced and satisfies the weak Segal condition.  
\end{proof}

\appendix
\section{Relation between the operads $B\OO$ and $W\OO$}\label{sec: BO and WO}

The Boardman-Vogt $W$-construction is a construction on operads with the property that, for any topological operad $\pp$, algebras over $W\pp$ are ``up-to-homotopy'' or ``weak'' $\pp$-algebras. A lax operad \cite{BrinkmeierThesis} is an algebra over the operad $W\mathcal{O}$, the Boardman-Vogt $W$--construction applied to the operad of operads $\mathcal{O}$ (Definition~\ref{def:operad of operads}), and is a notion of a ``weak'' or ``infinity'' operad. 
It is known that there exists a zig-zag of Quillen equivalences between the category of $W\mathcal{O}$-algebras and the category of reduced dendroidal spaces by, for example, combining Theorem 4.1 of \cite{bm_resolution} with either Theorem 1.1 of \cite{Bergner_Hackney_14} or a restriction of Theorem 8.15 of \cite{cm3}. 

Here we show how the operad $B\OO$ can be identified with a variant $W_0$ of the $W$-construction of the operad $\OO$ of operads (see Theorem~\ref{thm:BOWO}).  From this, it will follow that $B\OO$--algebras are lax operads that are strictly symmetric and with a strict identity (see Example~\ref{ex:W0O}). We start by recalling the $W$--construction.

\subsection{The $W$-Construction}\label{sec: W-construction}
The Boardman-Vogt $W$-construction is an enlargement of the free operad construction.  Given an operad $\pp$, there are canonical morphisms of topological operads
\[ F\pp \hookrightarrow W\pp \xrightarrow{\sim} \pp, \] where the map $p: W\pp \to \pp$ is a surjective homotopy equivalence. 
Algebras for $W\mathcal{P}$ are up-to-homotopy $\mathcal{P}$-algebras. We briefly recall the construction here and refer the reader to \cite[Section 17]{BrinkmeierThesis} or \cite[Section 3]{bm_resolution} for full details. %for the case of colored operads.   

\begin{definition}\label{def: W construction}

Let $\pp$ be a  $\mathfrak{C}$--colored (discrete or topological) %\nwnote{added} 
operad. The operad $W\pp$ is a topological operad with the same set of colors $\mathfrak{C}$, built from the free operad $F(\pp)$ (Definition~\ref{def: free operad}) by adding length in $[0,1]$ to the internal edges of the trees that define the elements of $F(\pp)$. More precisely, for each list of colors $c;c_1,\dots,c_k$ in $\mathfrak{C}$, we have  
$$W\mathcal{P}(c;c_1,\dots,c_k)= \Big(\coprod_{(\mathbb{T},f,\lambda)} \big( [0,1]^{|\ie(\mathbb{T})|} \times \prod_{v\in V(\mathbb{T})} \pp(out(v);in(v))\big)\Big)/\sim$$ 
%is the quotient space of the space $A(c_1,\dots,c_n;c)$ 
where the disjoint union, as for the free operad, runs over the isomorphim classes of leaf-labeled $\mathfrak{C}$--colored planar trees 
$$(\mathbb{T},f\colon E(\mathbb{T})\to \mathfrak{C},\lambda\colon \{1,\dots,k\}\to L(\mathbb{T}))$$ 
with $k$ leaves such that $f(\lambda(i))=c_i$, $f(R(\mathbb{T}))=c$. The equivalence relation is generated by the relation ($\ast$) in Definition~\ref{free_operad_relation} in addition to the following additional relations that capture ``weak'' operadic composition and units:

\begin{enumerate}
    \item any tree with an internal edge of length of zero is identified with the tree where that edge has been collapsed and the operations labelling its end vertices composed; 
    \item any tree that has a vertex with only one input and one output, both colored by $c\in\mathfrak{C}$, labeled by the identity in $\iota_c \in \pp(c;c)$, is identified with the tree where that vertex is deleted. The resulting new edge, if internal, has length the maximum length of the two original internal edges connected to the deleted vertex.
\end{enumerate}
See \cite[p 75]{BrinkmeierThesis} for a pictorial version of these relations.
The symmetric group acts on $W\mathcal{P}$ by relabeling the leaves, as for the free operad.
Composition is by grafting, giving length 1 to the newly created internal edge. 

We will denote elements of $W\mathcal{P}$ by $(\mathbb{T},f,\lambda,s,p)$, where $\mathbb{T}$ is a planar tree, $f:E(\mathbb{T})\to \mathfrak{C}$ is the map coloring its edges, $\lambda:\{1,\dots,k\}\to L(\mathbb{T})$ is the bijection labeling its leaves, $s\in [0,1]^{|\ie(\mathbb{T})|}$ is a collection of weights, and $p=(p_{v})_{v\in V(\mathbb{T})}$ is a labeling of the vertices by operations in $\pp$.
An example is shown in Figure~\ref{fig:coloured-w-construction}. There is a canonical projection map $\pi:W\pp\rightarrow\pp$ defined by sending all the edge lengths to $0$ and composing the operations of $\pp$ as dictated by the trees. 

\begin{figure}[h!t]
    \centering\def\svgwidth{0.45\columnwidth}
    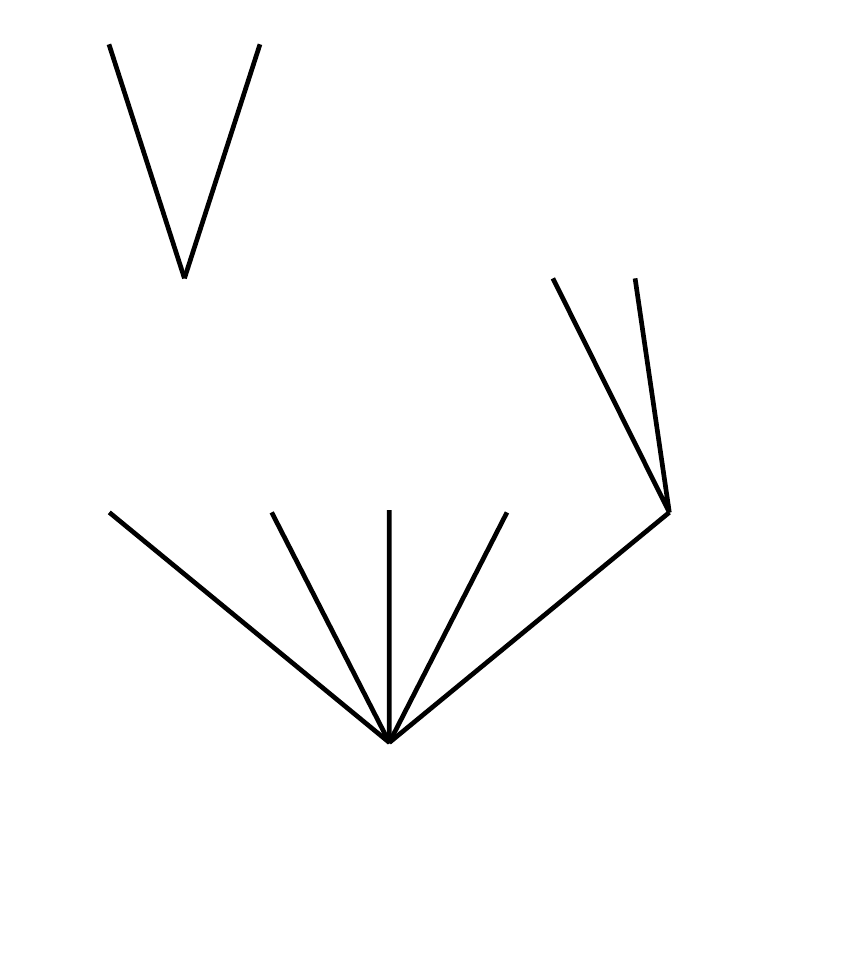
    \caption{Example of an element of ${W\protect\mathcal{P}(p;a,b,c,d,g,i,k,\ell,m,o)}$.}\label{fig:coloured-w-construction}
\end{figure}
\end{definition}

\subsection{A variant on the $W$-construction}\label{sec: W construction}

Given a (discrete or topological) $\mathfrak{C}$-coloured operad $\pp$, the topological operad $W_0\pp$ is defined as the quotient of $W\pp$ by replacing relation (2) in Definition~\ref{def: W construction} by the following stronger relations for arity one vertices, as well as a version for arity zero vertices: 
\begin{enumerate}
    \item[($2'$)] any tree that has a vertex $v$ with only one input and one output both colored by $c$, adjacent to at least one other vertex $w$, with $v$ labeled by {\em any} element $\pp(c;c)$, is identified with the tree where the vertex $v$ is deleted, and the label of $v$ and $w$ are composed in $\pp$ (Figure~\ref{fig:w_0-2'}).

    If the resulting new edge is internal, then its length is the maximum length of the two original (then necessarily internal) edges adjacent to $v$.

    \begin{figure}[ht]
    \centering\def\svgwidth{0.6\columnwidth}
    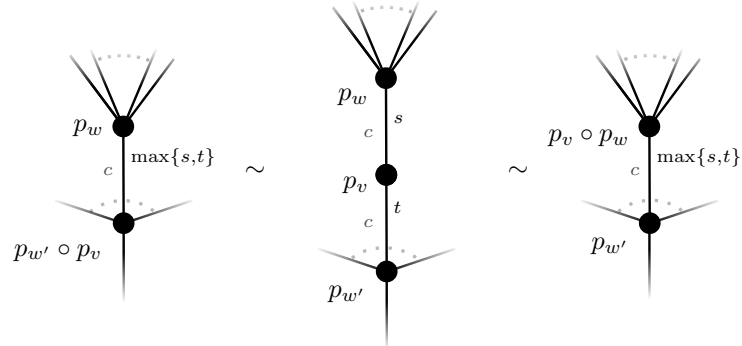
    \caption{Local representation of the relation ($2'$) on a tree.}\label{fig:w_0-2'}
    \end{figure}

    \item[(3')] any tree that has a vertex $v$ with no input, adjacent to another vertex $w$, with $v$ labeled by {\em any} element of $\pp(c;\emptyset)$, is identified with the tree where the vertex $v$ and the edge  between $v$ and $w$ are deleted, and the labels of $v$ and $w$ are composed in $\pp$ (Figure~\ref{fig:w_0-3'}). 
        \begin{figure}[ht]
        \centering\def\svgwidth{0.6\columnwidth}
        %% Creator: Inkscape 1.0beta2 (2b71d25, 2019-12-03), www.inkscape.org
%% Accompanies image file 'W_0-3.pdf' (pdf, eps, ps)
\begingroup%
  \makeatletter%
  \providecommand\color[2][]{%
    \errmessage{(Inkscape) Color is used for the text in Inkscape, but the package 'color.sty' is not loaded}%
    \renewcommand\color[2][]{}%
  }%
  \providecommand\transparent[1]{%
    \errmessage{(Inkscape) Transparency is used (non-zero) for the text in Inkscape, but the package 'transparent.sty' is not loaded}%
    \renewcommand\transparent[1]{}%
  }%
  \providecommand\rotatebox[2]{#2}%
  \newcommand*\fsize{\dimexpr\f@size pt\relax}%
  \newcommand*\lineheight[1]{\fontsize{\fsize}{#1\fsize}\selectfont}%
  \ifx\svgwidth\undefined%
    \setlength{\unitlength}{131.73172958bp}%
    \ifx\svgscale\undefined%
      \relax%
    \else%
      \setlength{\unitlength}{\unitlength * \real{\svgscale}}%
    \fi%
  \else%
    \setlength{\unitlength}{\svgwidth}%
  \fi%
  \global\let\svgwidth\undefined%
  \global\let\svgscale\undefined%
  \makeatother%
  \begin{picture}(1,0.3367609)%
    \lineheight{1}%
    \setlength\tabcolsep{0pt}%
    \put(0.50092523,0.14905712){\color[rgb]{0,0,0}\makebox(0,0)[t]{\lineheight{1.25}\smash{\begin{tabular}[t]{c}$\sim$\end{tabular}}}}%
    \put(0,0){\includegraphics[width=\unitlength,page=1]{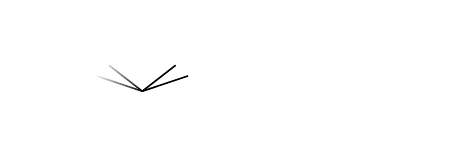}}%
    \put(0.28403401,0.27215057){\color[rgb]{0,0,0}\makebox(0,0)[rt]{\lineheight{1.25}\smash{\begin{tabular}[t]{r}$p_v$\end{tabular}}}}%
    \put(0,0){\includegraphics[width=\unitlength,page=2]{W_0-3.pdf}}%
    \put(0.28497649,0.10215477){\color[rgb]{0,0,0}\makebox(0,0)[rt]{\lineheight{1.25}\smash{\begin{tabular}[t]{r}$p_w$\end{tabular}}}}%
    \put(0,0){\includegraphics[width=\unitlength,page=3]{W_0-3.pdf}}%
    \put(0.66418439,0.15198091){\color[rgb]{0,0,0}\makebox(0,0)[rt]{\lineheight{1.25}\smash{\begin{tabular}[t]{r}$p_w\circ p_v$\end{tabular}}}}%
    \put(0.29118625,0.20742621){\color[rgb]{0.30196078,0.30196078,0.30196078}\makebox(0,0)[rt]{\lineheight{1.25}\smash{\begin{tabular}[t]{r}$\scriptstyle{c}$\end{tabular}}}}%
    \put(0,0){\includegraphics[width=\unitlength,page=4]{W_0-3.pdf}}%
    \put(0.32652751,0.23109389){\color[rgb]{0,0,0}\makebox(0,0)[lt]{\lineheight{1.25}\smash{\begin{tabular}[t]{l}$\scriptstyle{t}$\end{tabular}}}}%
  \end{picture}%
\endgroup%

        \caption{Local representation of the relation (3') on a tree.}\label{fig:w_0-3'}
        \end{figure}
\end{enumerate}
So a $W_0\pp$--algebra is a weak $\pp$--algebra ($W\pp$--algebra) for which the nullary and uniary operation are strict. And in particular, one has that $W_0\pp(c;c)=\pp(c;c)$ and $W_0\pp(c;\emptyset)=\pp(c;\emptyset)$ for any color $c$. Also, one can always choose representatives of elements of $W_0\pp$ using trees with no valence $0$ or $1$ vertices (unless it only has $0$ or $1$ vertex). In a tree that defines an element of $W_0\pp$, an arity one vertex lying in between two other vertices can be slid up or down to either of its neighboring vertices, composing its label with that of the chosen vertex, while an arity zero vertex can be ``pushed down'' to the vertex it is attached to. 

\begin{example}\label{ex:W0O} %\nwnote{new example}
The example relevant to us here is when we set $\pp=\OO$ is the operad of operads. In this case, $\mathfrak{C}=\mathbb{N}$ is the natural numbers and an $\OO$-algebra is a (monochrome) operad.  The nullary operations in $\OO(1;\emptyset)$ encode the identity operation in the $\OO$--algebra, while the unary operations in $\OO(n;n)$ encode the action of the symmetric groups. It follows that a $W_0\OO$--algebra is a strictly symmetric weak operad with a strict identity. 
\end{example}

%\medskip 

By construction, the canonical projection $p: W\pp\to \pp$
factors through the quotient map $q:W\pp \rightarrow W_0\pp$. Moreover, both $W\pp$ and $W_0\pp$ are homotopy equivalent to $\pp$:

\begin{prop}\label{prop:WW_0}
There are operad maps $W\pp\to W_0\pp\to \pp$, inducing homotopy equivalences 
\[\begin{tikzcd}W\pp(c;c_1,\ldots, c_n)\arrow[r, "q" above, "\sim" below] & W_0\pp(c;c_1,\ldots, c_n) \arrow[r, "p_0 " above, "\sim" below]& \pp(c;c_1,\ldots, c_n)\end{tikzcd}\]  
for each $n\geq 0$ and each $c;c_1,\ldots, c_n$ in $\mathfrak{C}$. 
 \end{prop}   
  
  \begin{proof}
For each $n\geq 0$ and $c; c_1,\ldots,c_n$ the map $$q:W\pp(c;c_1,\ldots,c_n)\to W_0\pp(c; c_1,\ldots, c_n)$$ is the projection on to the quotient. It is an operad map because if elements of $W\pp$ are 
equivalent in $W_0\pp$ before being composed, they are necessarily also equivalent in $W_0\pp$ after 
composition. The map $p_0:W_0\pp\to \pp$ contracts the remaining edges in the trees of $W_0\pp$ by 
sending the lengths to 0 (and composing the operations in $\pp$). This map is well-defined, as it is 
compatible with the  relations (2') and (3'), and respects the operad structure. These maps induce 
homotopy equivalences, with homotopy inverses given by including $\pp(c_1,\ldots, c_n;c)$ as labelled 
corollas in $W_0\pp(c;c_1,\ldots, c_n)$ or $W\pp(c;c_1,\ldots, c_n)$. 
\end{proof}

\subsection{$B\OO$-algebras are strictly symmetric lax operads}
In this section we show that there is an isomorphism of topological operads $B\OO\cong W_0\OO$. 
In particular, any $B\OO$--algebra will receive a canonical $W\OO$--structure via the map $W\OO\to W_0\OO$. Combining this with Example~\ref{ex:W0O}, gives a description of $B\OO$--algebras as strictly symmetric lax operads with strict identity.

Our main theorem in this appendix is:

\begin{theorem}\label{thm:BOWO}
    The operads $W_0\OO$ and $B\OO$ are isomorphic. 
\end{theorem}

Combining Theorem~\ref{thm:W_0O-Omega_0} and Theorem~\ref{thm:BOWO}, we immediately get 

\begin{cor}
There exist isomorphisms of categories
        $$ W_0{O}\mathrm{-Alg}_{\mathcal{S}} \cong (\mathcal{S}^{\widetilde\Omega_0^{op}})_{strict}.$$
\end{cor}

The proof of the theorem will be given in Section~\ref{sec:proof of BOWO}.  
Though not saying this explicitly, the proof uses the natural association of a bracketing to a \emph{clustering tree}, which is described for instance in \cite[Definition 2.7]{Sinha04}.

\medskip

Since the $W$-construction is built out of cubes, to prepare for the proof, we start by giving  an alternative description of $B\OO$ in terms of cubes as well. 

\begin{definition}
    We can define a \emph{weighted bracketing} 
    %\nwnote{changed length to weights. I think it's confusing with lengths} 
    of a tree $T$ to be a pair $(B,t)$ with bracketing $B=\{S_{j}\}_{j\in J}$ of $T$ and $t\in[0,1]^J$. The $j$th coordinate $t_j\in t$ is the \emph{weight} of $S_j$. The addition of weights associates to each bracketing a cube $[0,1]^{|B|}$. %to each bracketing $B$ of a tree $T$. 
    These cubes fit together to form a space: 
    $$\mathbb{B}(T)= \coprod\limits_{{B\in \mathcal{B}(T)}} [0,1]^{|B|}/_\sim $$ where the equivalence relation is by 
    identifying any bracketings with weights that only differ by a bracket of weight $0$ (see Figure \ref{fig:subfig_hexagon}). 
\end{definition}

    \begin{figure}[ht]
    \centering
    \subfigure[A bracketing $B$ of $T$ with $2$ brackets and its associated cube ${[0,1]}^{|B|}$ in $\mathbb{B}(T)$.]{ \label{fig:subfig_bracketonly}
    \includegraphics[width=0.4\textwidth]{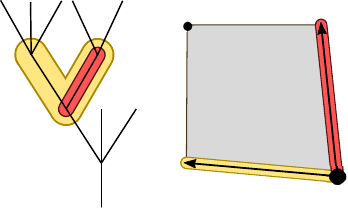}}
    \qquad \qquad
    \subfigure[All bracketings of $T$ and their associated cubes assembled to make $\mathbb{B}(T)$ a hexagon.]{
    \label{fig:subfig_hexagon}
    \includegraphics[width=0.3\textwidth]{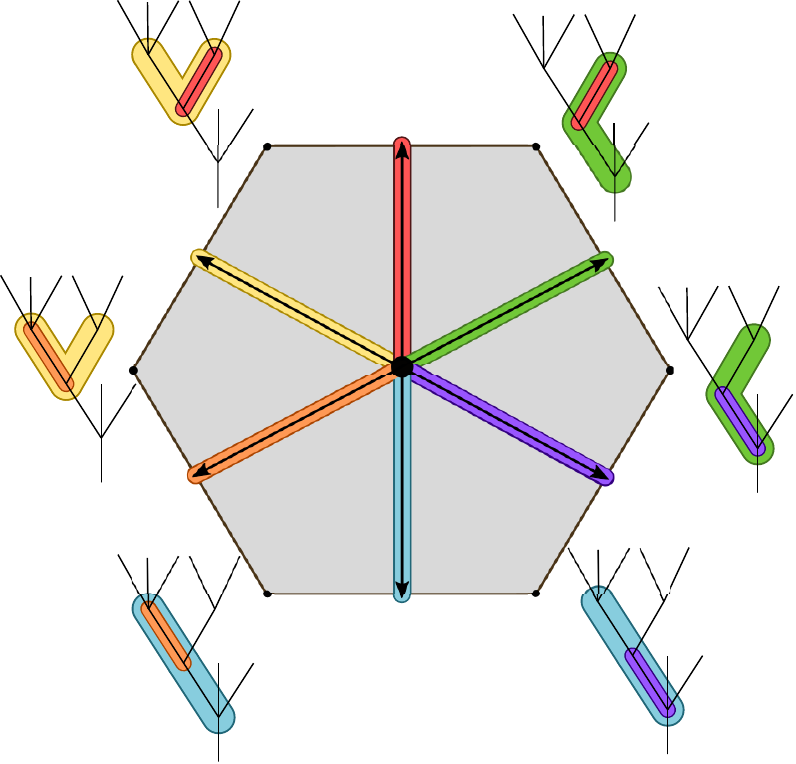}}
    \vspace{-1em}
    \caption{A tree $T$ and its corresponding space of bracketings $\mathbb{B}(T)$.}
    \end{figure}

Recall from Definition~\ref{def: poset of brackets} the poset $\B(T)$ of bracketings of a tree $T$ under the inclusion relation. 

\begin{lemma}\label{-cubical}
    Let $T$ be a tree. There is a homeomorphism $|\mathcal{B}(T)|\cong \mathbb{B}(T)$, between the realization of the nerve of the poset $\mathcal{B}(T)$ and the cubical space $\mathbb{B}(T)$.
\end{lemma}

\begin{proof}
We consider the topological $k$--simplex as the space 
$$\Delta^k=\{(s_1, \ldots, s_k) \in \mathbb{R}^k \colon 1=s_0\ge s_1 \ge \dots\ge s_k\ge 0\}.$$ Fix a tree $T$ and let $\sigma$ denote a $k$--simplex $B_0\subseteq\dots\subseteq B_k$ of the nerve of the poset $\mathcal{B}(T)$. To each $\sigma$ we associate a map 
$$\chi_\sigma:\Delta^k\longrightarrow \mathbb{B}(T)$$
where $\chi_\sigma(s_1,\dots,s_k)$ is the weighted bracketing of $T$ in which all trees of $B_0$ have weight $1=s_0$ and all trees of $B_i\backslash B_{i-1}$ have weight $s_i$ for $i\ge 1$.

The maps $\chi_{\sigma}$ assemble into a continuous map 
\[
\chi:|\mathcal{B}(T)|=\Big(\coprod_{k\ge 0} \mathcal{B}(T)_k\times 
\Delta^k\Big)/_\sim \ \longrightarrow \ \mathbb{B}(T)=\coprod\limits_{{B\in \mathcal{B}(T)}} [0,1]^{|B|}/_\sim .
\]

This map is a homeomorphism with inverse defined by mapping a cube $[0,1]^{|B|}$ in $\mathbb{B}(T)$ to the realization of the sub-poset $\mathcal{B}_{\le B}$, which is a cube whose dimension is the cardinality $|B|$ of the bracketing. Explicitly, given an element $(B,t)\in \mathbb{B}(T)$ with $B=(B_0,\dots,B_k)$, we order the coordinates of $t=(t_1,\dots,t_k)$ so that they are in decreasing order 
\begin{multline*}
    1=t_{\sigma(1)}=\dots= t_{\sigma(r_1)}> t_{\sigma(r_1+1)}=\dots=t_{\sigma(r_1+r_2)}> \\
    \dots> t_{\sigma(r_1+\dots+r_{l+1}+1)}=\dots=t_{\sigma(r_1+\dots+r_{l+2})}=0.
    \end{multline*}
%We then 
This defines an $l$--simplex $\bar B_0\subset \bar B_1\subset\dots\subset \bar B_l$ by setting \[\bar B_i=B_{\sigma(1)}\cup\dots\cup B_{\sigma(r_1+\dots+r_{i+1})}.\qedhere\]\end{proof}

\subsection{Proof of Theorem~\ref{thm:BOWO}}\label{sec:proof of BOWO}
In order to prove $B\OO\cong W_0\OO$,  we first recall some definitions. 
Recall that elements $(\mathbb{T},f,\lambda,s,p)\in  W_0\mathcal{O}(n;m_1,\dots,m_k)$ are represented by a planar tree $\mathbb{T}$ with $k$ leaves ordered by the bijection $\lambda:\{1,\dots,k\}\to L(\mathbb{T})$ and with an edge colouring $f:E(\mathbb{T})\to \mathbb{N}$ that, in particular, colours the leaves by $m_1,\dots, m_k$. In addition, $\mathbb{T}$ is equipped with a collection of lengths $s\in [0,1]^{|\ie(\mathbb{T})|}$, and a decoration of the vertices $p=(p_{v})_{v\in V(\mathbb{T})}$ by operations by $p_v$ in $\OO(out(v);in(v))$.

We call a representative $(\mathbb{T},f,\lambda,s,p)$ {\em reduced} if the tree $\mathbb{T}$ has no vertices of arity zero or one, unless such a vertex cannot be removed using the equivalence relation in $W_0\OO$, i.e.~if $\mathbb{T}$ is the corolla $C_0$ or $C_1$. In particular, every element of $ W_0\OO$ has a reduced representative, which in general is not unique. It greatly simplifies the proof of Lemma~\ref{lemma:Constructing-the-map-WO-BO} to work with reduced representative. 

For a given tree $\mathbb{T}$, and vertices $v,w \in V(\mathbb{T})$, we say that $w$ is \textit{above} $v$ if the unique shortest path between $w$ and the root of the tree goes through $v$. In this case $v$ is \textit{below} $w$. Every other vertex of $\mathbb{T}$ is above the \textit{root vertex} $v_0$ whose outgoing edge is the root of $\mathbb{T}$.

\begin{lemma}\label{lemma:Constructing-the-map-WO-BO}
    There is a map of topological operads $\Psi:W_0\mathcal{O}\rightarrow B\mathcal{O}$.
\end{lemma}

The map $\Psi$ is illustrated in Figure~\ref{fig:bracketing}. 

\begin{figure}[h!t]
\centering
\begin{minipage}{0.45\textwidth}
    \centering\def\svgwidth{\columnwidth}
    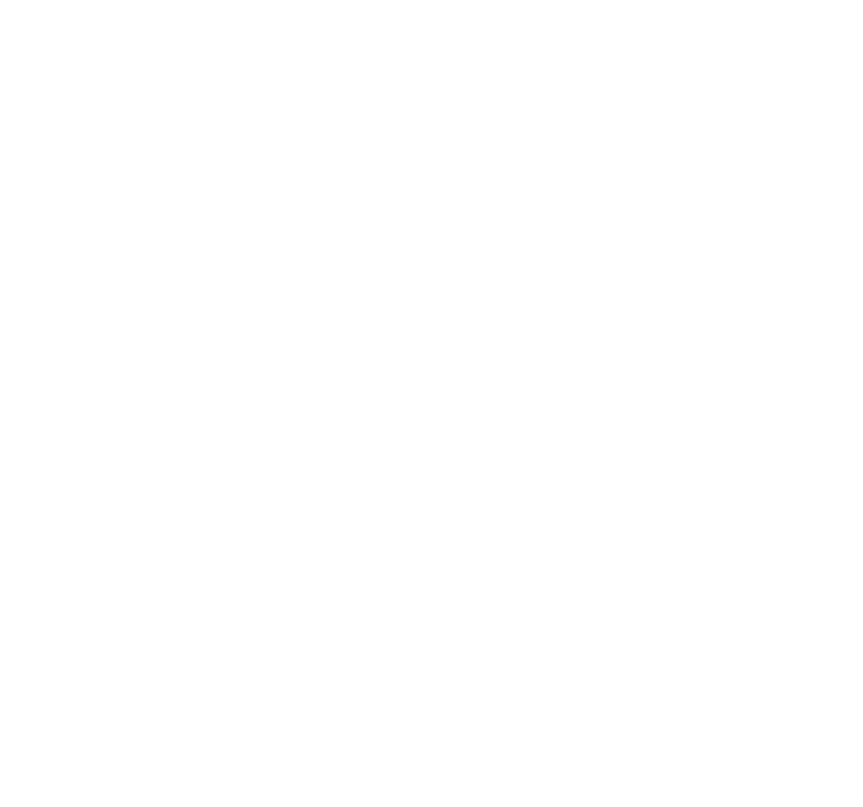
\end{minipage}
\begin{minipage}{0.5\textwidth}
    \centering\def\svgwidth{\columnwidth}
    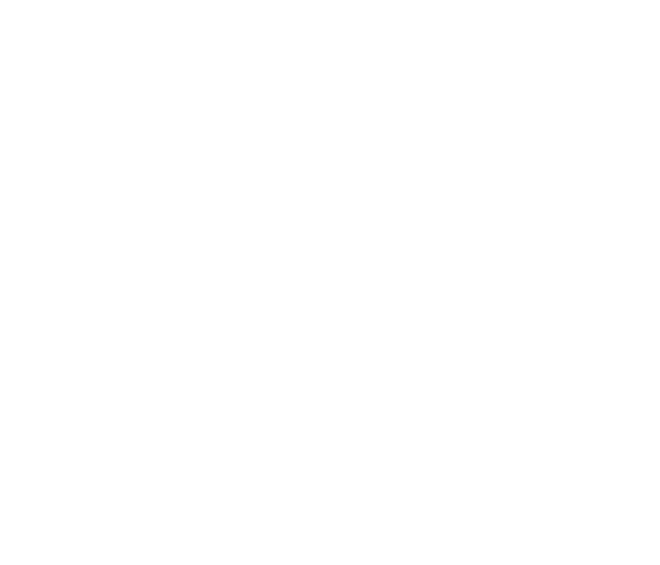
\end{minipage}
\caption{Element of $W_0\OO$ and corresponding element of $B\OO(16;2,3,4,3,4,3,3)$. }\label{fig:bracketing}
    %    \caption{A weighted bracketing $(B, t)$  and its corresponding tree $\mathbb{T}_B$}
\end{figure}

\begin{proof}
Given a reduced element $ (\mathbb{T},f,\lambda,s,p)\in W_0\OO(n;m_1,\dots,m_k)$, we construct \[\Psi(\mathbb{T},f,\lambda,s,p)=(T,\sigma,\tau,B,t)\in B\mathcal{O}(n;m_1,\dots,m_k),\] where 
$(B,t)$ is a weighted bracketing on the labelled tree 
\[(T, \sigma, \tau) = p_0 (\mathbb{T},f,\lambda,s,p) \in \OO (n;m_1,\dots,m_k)\] that is the image of $(\mathbb{T},f,\lambda,s,p)$ under the canonical projection $p_0\colon W_0\OO\to \OO$. 

\medskip
    
The bracketing $B$ is constructed from the set of vertices of $\mathbb{T}$. If $\mathbb{T}$ has at most one vertex, then set $B=\emptyset$ to be the trivial bracketing, in which case there are no weights to chose so $t$ is the empty map.

Otherwise, since $(\mathbb{T},f,\lambda,s,p) $ is reduced, and $\mathbb{T}$ is not a corolla, all its vertices have arity $\geq 2$. Let $v_0$ be the root vertex of $\mathbb{T}$. % with construct the bracketing 
    %of arity at least $2$, denoted $V_{\ge 2}(\mathbb{T})$, 
   %as follows: given  $v\in V(\mathbb{T})$, 
   For each $v\in V(\mathbb{T})\backslash \{v_0\}$, let 
   \[(S_v,\sigma_v,\tau_v) = p_0(\mathbb{T}_v,f|_{\mathbb{T}_v},\lambda|_{\mathbb{T}_v},p|_{\mathbb{T}_v}),
   \] where $\mathbb{T}_v$ is the subtree of $\mathbb{T}$ with $v$ as its root vertex, and containing all the vertices above $v$. Observe, in particular, that, since $v\neq v_0$, the outgoing edge $e_v$ of $v$ -- that is the root of $S_v$ -- is internal in $T$. Since
   the vertices of $\mathbb{T}$ have arity at least 2, 
   each $S_v$ is a large proper subtree of $T$, and because composition in $\OO$ is by substitution, 
  \[ B  = \{S_v \ : \ v\in V(\mathbb{T})\backslash \{v_0\} \}\] is a collection of nested subtrees, and hence a bracketing.

To define the weight function $t$ of $B$, we associate, to each $S_v$ the weight $t_v=s(e_v)$, the length  of $e_v \in \ie(\mathbb{T})$.  %associated $S_v$ is set to be the length of the outgoing edge of $v$, which is necessarily an internal edge of $\mathbb{T}$ since $v\neq v_0$. 
  This completes the definition of $\Psi(\mathbb{T},f,\lambda,s,p)$. 
  
  \medskip

We need to check that the defined bracketing is independent of our choice of (reduced) representative $(\mathbb{T},f,\lambda,s,p)\in  W_0\mathcal{O}(n;m_1,\dots,m_k)$, and continuous. In particular, we must check that it is compatible with the relations (1), %(in the definition of the W-construction),
(2') and (3') in Definition~\ref{def: W construction} and Section~\ref{sec: W construction}.

To prove that $\Psi$ is well defined with respect to relation (1), and hence also continuous, let $s_j$ be the length of an internal edge of $\mathbb{T}$ with end vertices $v,w$, where $v$ is above $w$. Then if $s_j$ goes to $0$ in $W_0 \OO$, the vertices $v,w$ are identified and their labels are composed in $\OO$. Applying $\Psi$, this will precisely have the effect of taking the weight of the bracketing $S_v$ to $0$, which is equivalent to simply forgetting the bracketing $S_v$ in $B\OO$.

Relation (2') allows that  a vertex $v$ with only 1 input in $\mathbb{T}$, labeled by a permutation $\alpha\in \OO(n;n)\cong \Sigma_n$, to be composed to either of the vertices it shares an edge with. So suppose $\mathbb{T}$ is the reduction of a tree $\widetilde{\mathbb{T}}$ with an arity one vertex $v$ attached to two vertices $w$ and $w'$, with $w'$ below $w$. We may assume that $w$ and $w'$ both have arity at least two. We let $\mathbb{T}$ be the tree obtained from $\widetilde{\mathbb{T}}$ by collapsing the edge between $w$ and $v$ %$v--w$ 
and let $\mathbb{T}'$ be the tree obtained from $\widetilde{\mathbb{T}}$ by collapsing the edge between $v$ and $w'$. % $v--w'$ instead. 
We need to check that the brackets $S_w$ and $S_{w'}$ are the same if computed using the representative    $(\mathbb{T},f,\lambda,s,p)$ associated to $\mathbb{T}$ or    $(\mathbb{T}',f',\lambda,s',p')$ associated to $\mathbb{T}'$. This is immediate for the bracket $S_{w'}$ because $w'$ is below $v$ and thus $p_0(\mathbb{T}_{w'},f|_{\mathbb{T}_{w'}},\lambda|_{\mathbb{T}_{w'}},p|_{\mathbb{T}_{w'}})=p_0(\mathbb{T}'_{w'},f'|_{\mathbb{T}'_{w'}},\lambda|_{\mathbb{T}'_{w'}},p'|_{\mathbb{T}'_{w'}})$. For the vertex $w$, the two representatives in general do not have the same image under $p_0$, but if $p_0(\mathbb{T}'_{w},f'|_{\mathbb{T}'_{w}},\lambda|_{\mathbb{T}'_{w}},p'|_{\mathbb{T}'_{w}})=(S'_w,\sigma'_w,\tau'_w)$, we still have that $S'_w=S_w$. In fact, only $\tau'_w$ might differ from $\tau_w$ as the vertex $v$ is a permutation $\alpha\in \OO(n;n)=\Sigma_n$ that acts on a labeling $p\in \OO(n;k_1,\dots,k_l)$ by permuting the leaves of the labeled tree representing $p$. %After applying $p_0$, this will exactly likewise affect the labeling of the leaves of $S'_w$, as recorded by $\tau'_w$.  

 For relation (3') in the definition of $W_0$, the relation gives a unique way to reduce a tree if an arity zero vertex is attached to another vertex, so the representative with no arity $0$ vertices is unique and nothing needs to be checked.

\medskip

Finally, we check that $\Psi$ is a map of operads. Consider a composition $ (\mathbb{T}_1,f_1,\lambda_1,s_1,p_1)\circ_i(\mathbb{T}_2,f_2,\lambda_2,s_2,p_2) $ of reduced representatives in $ W_0 \OO$, and  let $\Psi (\mathbb{T}_j,f_j,\lambda_j,s_j,p_j) = (T_j, \sigma_j, \tau_j, B_j, t_j)$ for $j = 1, 2$. Composition in $W_0\OO$ is induced by grafting a tree $\mathbb{T}_2$ onto the $i$th leaf of $\mathbb{T}_1$, creating a new internal edge of length $1$. If $\mathbb{T}_2$ has at least one vertex of arity $2$,  this corresponds exactly under $\Psi$ to adding a new bracket $T_2$ of weight $1$ in the composed tree $T_1\bullet_{\sigma_1(i), \tau_2} T_2$ where the composition here is by insertion. 
  If not, then, since $(\mathbb{T}_2,f_2,\lambda_2,s_2,p_2)$ is reduced, $\mathbb{T}_2$ has either no vertices or a single arity 1 vertex, so $T_2$ is either the exceptional tree $ \eta$ or a corolla $C_n$. In each case, the newly added edge in the composed tree $\mathbb{T}_1\circ_i\mathbb{T}_2$ will be collapsed when going to a reduced tree, corresponding under $\Psi$ to a composition in $B\OO$ where no extra bracket is added. This finishes the proof. 
\end{proof}

\begin{lemma}\label{lemma:Bijective-map-of-operads}
For every $(n;m_1,\dots,m_k)$ the map $\Psi:W_0\mathcal{O}(n;m_1,\dots,m_k) \longrightarrow B\mathcal{O}(n;m_1,\dots,m_k)$ is a bijection.
\end{lemma}

\begin{proof}%[Proof of Lemma \ref{lemma:Bijective-map-of-operads}]
%    We show $\Psi: W_0\mathcal{O}(m_1,\dots,m_k;n) \longrightarrow B\mathcal{O}(m_1,\dots,m_k;n)$ is bijective by constructing, for each $(T,\sigma,\tau,B,t)\in B\mathcal{O}(m_1,\dots,m_k;n)$ an element in its pre-image (which comprises most of part (b)) and showing it is unique.
%We start by showing the map $\Psi$ is \emph{surjective} by constructing, for each element in $B\mathcal{O}$, an element in its pre-image under $\Psi$. Given a 
We start by checking that $\Psi$ is surjective. 
So let $(T,\sigma,\tau,B,t)$  of $B\mathcal{O}(n;m_1,\dots,m_k)$ %where $(T,\sigma,\tau)\in\mathcal{O}(n;m_1,\dots,m_k)$, 
with $B=\{S_j\}_{j\in J}$ and $t\in[0,1]^J$. %where $(T,\sigma,\tau)\in\mathcal{O}(m_1,\dots,m_k;n)$ a labeled planar rooted tree, and $(B,t)$ a weighted bracketing of $T$, where $B=\{S_j\}_{j\in J}$ and $t\in[0,1]^J$. 
We may always choose 
a representative where all brackets have non-zero weight, so we assume that $t_j\not=0$ for any $j\in J$.
%'s the case. 
We will construct an element $(\mathbb{T},f,\lambda,s,p)\in W_0\mathcal{O}(n;m_1,\dots,m_k)$ in the preimage of $(T,\sigma,\tau,B,t)$. 

If $B = \emptyset$ is the empty bracketing then define $\mathbb{T}$ to be the corolla with $k$ leaves, with $f$ coloring its leaves $m_1,\dots,m_k$ in the ordering given by $\lambda$, and the root by $n$ and $p$ labeling the unique vertex by $(T,\sigma,\tau)$. The weights $s$ are trivial in this case.  By definition, $\Psi$ takes this element to $(T,\sigma,\tau,\emptyset,0)$ as required.  

We now assume that $B = \{S_j\}_{j \in J}$ is non-empty. To encode the leaf labelling $\tau$ on $T$, it is convenient to choose a non-reduced representative of its preimage, using a tree $\mathbb{T}$ with one valence $1$ vertex at its root.  
We define $\mathbb{T}$ as follows: 
we  set  $V(\mathbb{T})=\{v_\tau,v_T\}\cup\{v_j\}_{j\in J}$, where the vertex $v_j$ corresponds to the bracket $S_j\in B$, $v_T$ corresponds to an additional ``trivial bracket" $S_T:=T$, and $v_\tau$ will be associated to the permutation $\tau$. %We also know that if 
To construct the tree, we set $v_i$ above $v_j$ if 
$S_i\subset S_j$, connecting the two vertices by an edge $e_i$ if there is no $k\in J$ such that $S_i\subsetneq S_k \subsetneq S_j$, where we allow $S_j=S_T$. This edge $e_i$ is colored by the number $|L(S_i)|$ of leaves of the smaller tree $S_i$ and we define its length by setting  $s_i = t_i$ is the weight of the corresponding bracket. The nesting condition on the brackets implies that no cycles are formed this way.  We also connect $v_T$ and $v_\tau$ by an edge of length $1$, colored by $n=|L(T)|$, which is also the color of the root of the tree. 

Finally for each vertex $v$ of $T$, we attach a leaf $l_v$ to the vertex $v_i\in V(\mathbb{T})$ if $S_i$ is the smallest tree of the bracketing containing $v$, attaching it to $v_T$ if $v$ is contained in no bracket. This leaf is colored by the arity of $v$ in $T$. 
This defines the tree $\mathbb{T}$, with edge lengths $s$ and edge coloring $f$.

We pick some planar structure for $\mathbb{T}$. 
 (Recall that elements of $W_0\OO$ are only defined up to non-planar isomorphism, which is why there is some freedom here.) %\nwnote{changed. I think it's easier to explain it this way.} 
    Note that the leaves of $\mathbb{T}$ correspond exactly to the vertices of $T$. The ordering $\lambda:\{1,\dots,k\}\to L(\mathbb{T})$ is determined by $\sigma$ and this identification.
 This defines the tuple $(\mathbb{T},f,\lambda,s)$.
    
All that remains is to define the decoration $p$ of the vertices of $\mathbb{T}$ by elements of $\OO$. We need to have that $(T,\sigma,\tau)$ is given by the composition of the elements of the vertices of $\mathbb{T}$ so to determine the decorations in $\mathbb{T}$, what we need is to ``undo'' the compositions in $T$ marked by the bracketings.  

Let $v_j\in V(\mathbb{T})$. We define $p(v_j)$ to be the element $(S_j/\sim,\sigma_j,\tau_j)\in \OO(out(v_j);in(v_j))$ where  $S_j/\sim$ is the planar tree $S_j$ with each subtree $S_i\subsetneq S_j$ collapsed to a corolla with the same set of leaves, $\sigma_j$ orders the vertices according to the above chosen planar ordering of $\mathbb{T}$, where we note that the incoming edges of $v_j$ correspond precisely to the vertices of $S_j/\sim$, 
%where subtrees again inherit a placement by smallest vertex, 
and $\tau_j$ labels the leaves of $S_j/\sim$, which are also the leaves of $S_j$, in the order given by the planar embedding of $T$. 
(Here it is important that the chosen planar structure of $\mathbb{T}$ is compatible with the chosen order $\sigma_j$ of $V(S_j/\sim)$. On the other hand, the chosen order $\tau_j$ of $L(S_j/\sim)$ is not important, as we will fix it below using the vertex $v_\tau$.) 
This determines $p$ uniquely on all vertices $\{v_j\}_{j\in J}\cup\{v_T\}$.
    Finally, the vertex $v_\tau$ is labeled by the permutation $\tau\in \Sigma_n$, considered as an element of $\OO(n;n)$. 
    
    This finishes the construction of $(\mathbb{T},f,\lambda,s,p)$. To compute its image under $\Psi$, we have to pass to a reduced representative, which means collapsing the edge between $v_\tau$ and $v_T$ and composing their labeling. (The length of that edge is forgotten.) 
We have that $\Psi(\mathbb{T},f,\lambda,s,p)=(T,\sigma,\tau,B,t)$, by our choice of $p$ for the tree $T$ and its leaf-labeling $\tau$, our choice of $\lambda$ for the ordering $\sigma$ of the vertices, our choice of vertices of $\mathbb{T}$ for $B$, and with a direct correspondence between the length $s_i$ of the edge $e_i$ and the weight $t_i$ of $B_i$. 

\medskip

To finish the proof, we check that $\Psi$ is injective. We will check that, up to the equivalence relations defining $W_0\OO$, there is a unique reduced  $(\mathbb{T'},f',\lambda',s',p')$ in the preimage of $(T,\sigma,\tau,B,t)$. 
%is another element of $W_0\OO$ whose image under $\Psi$ is $(T,\sigma,\tau,B,s)$. Using the relations in $W_0\OO$, we may assume that $\mathbb{T}'$ is reduced, i.e.~has no vertices of arity zero or one attached to another vertex. 
Note that the number of vertices of such a reduced representative is determined by the tree $T$ and the cardinality of $B$. %We now proceed to show $\mathbb{T}'$ and $\mathbb{T}$ are equivalent, 
We consider first the cases where $\mathbb{T}'$ has 0 or 1 vertex. 

If $\mathbb{T}'$ has no vertices, then $\mathbb{T}'=\eta$ representing the identity element in $\OO(1;1)$, $B=\emptyset$, and, up to the equivalence relations of $W_0\OO$, there is only one possibility for $(\mathbb{T'},f',\lambda',s',p')$. 
%thus $T=\eta$ and $\mathbb{T}$ is necessarily equivalent in $W_0\OO$ to $\mathbb{T}'$. 

Suppose now that $\mathbb{T}'=C_k$ has exactly one vertex of arity $k$. 
The leaves of $\mathbb{T'}$ are in one-to-on correspondence with the vertices of $T$, with $\lambda'$ ordering its leaves, and $f'$ coloring them $m_1,\dots,m_k, n$, with $m_i$ the color of $\lambda'(i)$. We can choose a representative of $(\mathbb{T}',f',\lambda',s',p')$ so that the planar structure of $\mathbb{T'}=C_k$ is given by the ordering $\sigma$ of the vertices of $T$. Then the labeling $p$ of the vertex is necessarily precisely $(T,\sigma,\tau)$. So there is only one possibility for $(\mathbb{T'},f',\lambda',s',p')$. 

Finally, if $\mathbb{T}'$ has at least two vertices, then it must have precisely $|B|+1$ vertices arranged in a tree according to the nested structure of the bracket, and $k$ leaves, with each leaf attached to the vertex corresponding to the appropriate bracket. The root vertex of $\mathbb{T}$ corresponds to the whole tree $T$. The coloring of the edges is determined by the arity of the vertices and brackets in $T$, and the labeling of the leaves $\lambda$ is determined by the ordering $\sigma$. The vertices are decorated by tuples $(T_j,\sigma_j,\tau_j)$, with $T_j$ determined by the bracketing $B$, $\sigma_j$ determined by the nesting of the bracketing once a planar structure for $\mathbb{T}'$ is chosen. Choosing a different planar structure will give an equivalent element of $W_0\OO$ (in fact also of $W\OO$). The ordering $\tau_j$ is likewise not uniquely determined by the situation, but a different choice that does yield the same tuple $(T,\sigma, \tau,B,s)$ under $\Psi$ will be equivalent in $W_0\OO$, using relation $(2')$.  This finishes the proof of injectivity. 

\end{proof}

We are now ready to prove our main result in this appendix, namely that $W_0\OO$ and $B\OO$ are isomorphic as topological operads.

\begin{proof}[Proof of Theorem~\ref{thm:BOWO}]
In  Lemma \ref{lemma:Constructing-the-map-WO-BO} we constructed a map of topological operads
$$\Psi: W_0\mathcal{O} \longrightarrow B\mathcal{O}.$$ 
Combining this with Lemma~\ref{lemma:Bijective-map-of-operads} we know that, for each tuple $(n;m_1,\dots,m_k)$, the map \[\Psi:W_0\mathcal{O}(n;m_1,\dots,m_k) \longrightarrow B\mathcal{O}(n;m_1,\dots,m_k)\] is a continuous bijection. %Hence it is a continuous bijection. 
As the source of this map is a compact space ($\pi_0W_0\mathcal{O}(n;m_1,\dots,m_k)=\OO(n;m_1,\dots,m_k)$ is finite and there are finitely many reduced representatives $(\mathbb{T},f,\lambda,s,p)$ defining a cube in each component), and the target is a Hausdorff space, $\Psi$ is therefore a local homeomorphism and hence an isomorphism of topological operads. 
\end{proof}

\begin{remark}\label{rem:WO as poset}
A corollary of the result we just proved is that $W_0\OO$ is the realization of an operad in posets, namely the operad $B\OO$. %of the poset of bracketed elements of $\OO$. 
The operad $W\OO$ can likewise be seen as the realization of an operad in posets, namely the poset of elements of the free operad $F\OO$, with poset structure generated by edge collapses. 
The map of operads $q:W\OO\to W_0\OO$ is the realization of a map of posets. 
Indeed, the map $q\colon W\OO\to W_0\OO\cong B\OO$ respects the poset structure because collapsing an edge in $\mathbb{T}$, which defines the poset structure underlying $W\OO$, corresponds under the map $q$ to forgetting a bracket, which defines the poset structure underlying $W_0\OO=B\OO$. 
\end{remark}

\section{The explosion category of $\om$}\label{sec:hatomega}
In Section~\ref{sec: Omega W} we introduced an enriched version of the dendroidal category $\widetilde\om_0$ which is closely related to the category of $B\OO$-algebras. As mentioned in the introduction of Section~\ref{sec:thickening-omega}, the idea of the category $\widetilde\om_0$ is to encode homotopy coherent $\om$--diagram, and hence $\widetilde\om_0$ should be connected to the \emph{explosion category} of $\om$, as defined by Leitch \cite{Leitch} and Segal \cite[Appendix B]{Segal74}. 

In this appendix we describe the explosion category of $\om$, denoted $\widetilde\om$, and 
%describe the relationship between this category and $\widetilde\om_0$. Precisely, we will 
show that our topological category $\widetilde\om_0$ sits between $\widetilde\om$ and $\om$ in the sense that there exist equivalences of topological categories $$ \begin{tikzcd} \widetilde\Omega \arrow[r, "q" swap]\arrow[rr, dotted, bend left=25, "\tilde{p}"]& \widetilde\Omega_0 \arrow[r,"p" swap] & \Omega\end{tikzcd}.$$ 

The explosion construction and the $W$--construction are very closely related in spirit. One might thus expect a relationship between Segal $\widetilde\om$--diagrams and $W\OO$--algebras, similar to the relationship between Segal dendroidal spaces ($\om$--diagrams) and $\OO$--algebras, and between Segal homotopy dendroidal spaces ($\widetilde\om_0$--diagrams) and $W_0\OO$-- or $B\OO$--algebras.  Theorem~\ref{thm:WO-Omega} below will show that such a relationship exists, but without being as close as in the other cases: $W\OO$--algebras identify with a full subcategory of the category of reduced strict Segal $\widetilde{\om}$--diagrams.  

\subsection{The explosion of $\om$}
 
For each morphism $g:S\to T$ in $\om$, we define a \emph{poset of paths} $\text{Path}_{\om}(S,T)_g$ whose objects are the 
factorizations of $g:S\rightarrow T$ in $\om$ \[\begin{tikzcd}S\arrow[rrrr,bend left =20,dotted, "g"]\arrow[r,"g_1", swap]& T_1\arrow[r,"g_2", swap]&\dots\arrow[r]& T_{n-1}\arrow[r,"g_n", swap]&T, \end{tikzcd} \] 
where we identify two factorizations if they differ only by identity morphisms. In particular, each such factorisation $(g_1,\dots,g_n)$ has a unique reduced representative containing no identity morphisms unless $n = 1$ and $g$ is the identity on $S$. (Such a factorisation can be thought of as a path in the nerve of $\om$.) The poset structure is by refinement of factorisation: $(g_1,\dots,g_n)\le (g_1',\dots,g_m')$ if $n\le m$ and there is a monotone map $\alpha:\{0,\dots,n\}\to \{0,\dots,m\}$ such that $\alpha(0)=0$, $\alpha(n)=(m)$, and $g_i=g'_{\alpha(i)}\circ \dots\circ g'_{\alpha(i-1)+1}$ for each $1\le i\le n$.

We denote the geometric realization of this poset by
$$K_g:=|\text{Path}_{\om}(S,T)_g|.$$

\begin{definition}\label{def:omega-hat}
The topological category $\widetilde\Omega$ has the same objects as $\Omega$. Morphism spaces in $\widetilde\Omega$ are defined as \[\Hom_{\widetilde\Omega}(S,T)=\coprod_{g\in \Hom_\om(S,T)} K_g=\coprod_{g\in \Hom_\om(S,T)}|\text{Path}_{\om}(S,T)_g|.\] 
Composition of morphisms of $\widetilde\Omega$ is given by concatenation of factorizations. 
\end{definition}

\begin{example}\label{example maps in hat omega from corolla}
Fix a tree $T$ with $|L(T)|=n$ leaves and three inner edges: $e_1$, $e_2$, $e_3$.  Recall that $C_n$ denotes the corolla with $n$ leaves. Let $\partial_{e_1}$, $\partial_{e_2}$, $\partial_{e_3}$ denote the inner face maps in $\om$ associated to each inner edge, and let $g=\partial_{e_1}\partial_{e_2}\partial_{e_3}:C_{n}\rightarrow T$ 
be their composition. 
Then $g$ admits a factorization \[\begin{tikzcd} C_n\arrow[r] \arrow[rrr,bend left =20,dotted, "g"] &T_1 \arrow[r] & T_2 \arrow[r]& T.  \end{tikzcd}\] 
as a composition of three inner face maps for each permutation of $\{1,2,3\}$.  %$g=\partial_{e_1}\partial_{e_2}\partial_{e_3}$. 
The elements of  $\text{Path}_{\Omega}(C_n,T)_g$ that involve only these three inner face maps form a subposet with 
    $ ([1],\xrightarrow{g}) $ as minimum, 
and for each permutation $\sigma\in\Sigma_3$ the elements
    \begin{align*}
        ([3],\xrightarrow{\partial_{\sigma(1)}}\xrightarrow{\partial_{\sigma(2)}}\xrightarrow{\partial_{\sigma(3)}}) \quad \quad ([2],\xrightarrow{\partial_{\sigma(1)}}\xrightarrow{\partial_{\sigma(2)}\partial_{\sigma(3)}})\quad \quad
        ([2],\xrightarrow{\partial_{\sigma(1)}\partial_{\sigma(2)}}\xrightarrow{\partial_{\sigma(3)}}). 
    \end{align*}

Each permutation $\sigma$ this way contributes to a square
    \[\begin{tikzcd}
        (\xrightarrow{\partial_{\sigma(1)}\partial_{\sigma(2)}}\xrightarrow{\partial_{\sigma(3)}}) \arrow[r] & (\xrightarrow{\partial_{\sigma(1)}}\xrightarrow{\partial_{\sigma(2)}}\xrightarrow{\partial_{\sigma(3)}})\\
        (\xrightarrow{g}) \arrow[u] \arrow[r] \arrow[ur] & (\xrightarrow{\partial_{\sigma(1)}}\xrightarrow{\partial_{\sigma(2)}\partial_{\sigma(3)}}) \arrow[u]\\
    \end{tikzcd}\]
in this subposet, and the dendroidal identities tell us that these squares together form the following hexagon inside  $|\text{Path}_{\Omega}(C_n,T)_g|$:
\[\includegraphics[width=0.7\textwidth]{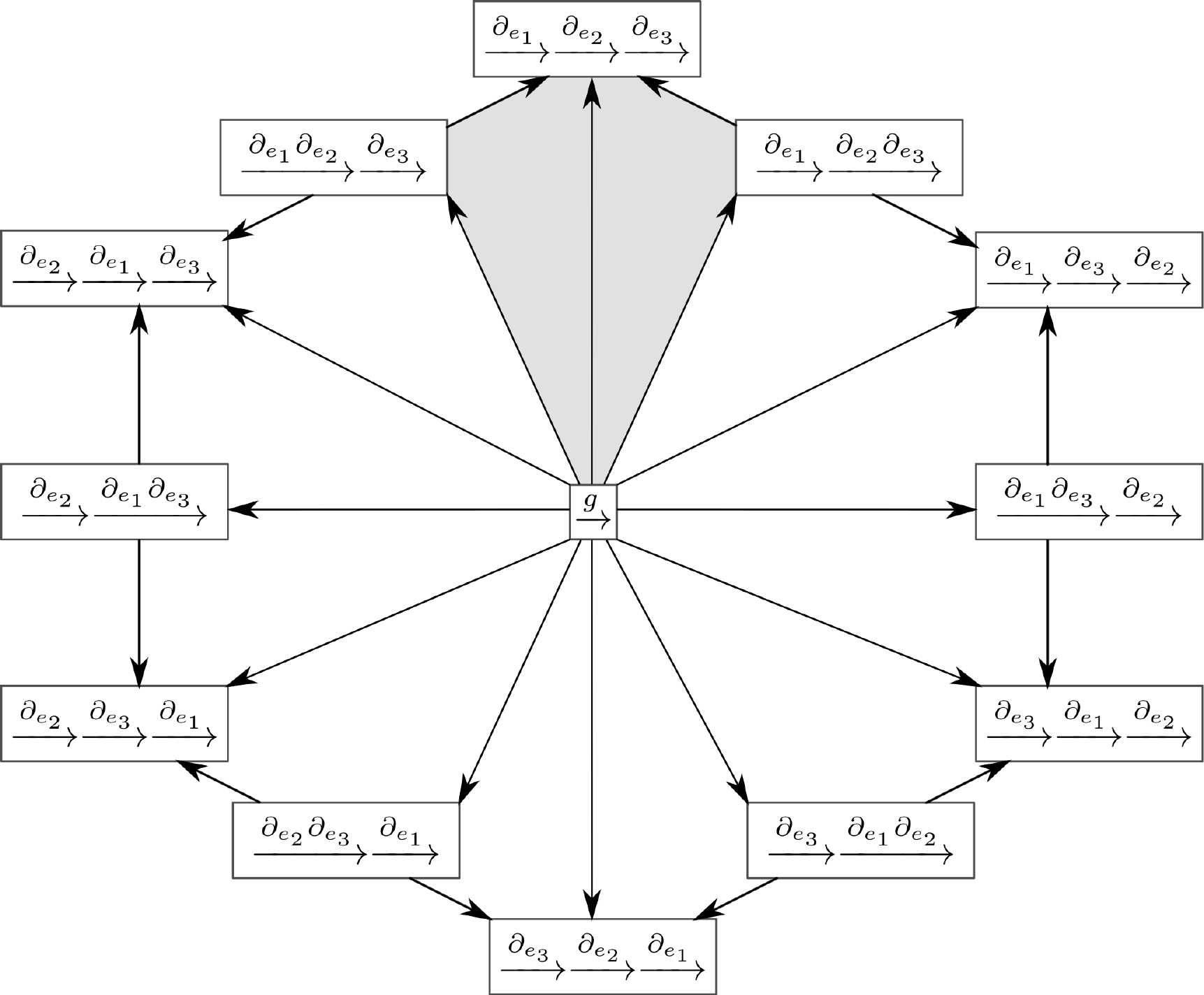}\]
Additional elements of $\text{Path}_{\Omega}(C_n,T)_g$ can be obtained by inserting tree isomorphisms. This example should be compared to Examples~\ref{example: pentagon} and~\ref{example: hexagon-trivalent} which can be interpreted as computing morphism spaces in the category $\widetilde\om_0$ likewise associated to trees with three internal edges, where in one case a pentagon occurs, and in the other it is a hexagon. 
%For example, the category $\text{Path}_{\Omega}(C_n,C_n)$ has one object $([1],\sigma)$ for each permutation $\sigma\in \Sigma_n$, and these can be concatenated to the above described elements of $\text{Path}_{\Omega}(C_n,T)$. \nwnote{ok??? --changed...}
\end{example}

\begin{lemma}
For each $g\in\Hom_{\om}(S,T)$  the space  $K_g=|\text{Path}_{\om}(S,T)_g|$ is contractible. 
\end{lemma}

\begin{proof}
The poset $\text{Path}_{\om}(S,T)_g$ has the trivial factorisation $S\xrightarrow{g} T$ as a minimal element. 
\end{proof}

Let $\tilde{p}:\widetilde\Omega\rightarrow\Omega$ be the functor that is the identity on objects and projects each morphism space $K_g$ to $g$. 
Considering $\om$ as a discrete topological category, the lemma immediately gives the following proposition. 
\begin{prop}\label{prop:pi}
The functor $\tilde{p}:\widetilde\Omega\rightarrow\Omega$ induces a homotopy equivalence on morphism spaces. 
\end{prop}

Note that the proposition identifies $\om$ with the ``path component category'' $\pi_0\widetilde\om$, which has the same objects as $\widetilde\om$ and $\Hom_{\pi_0\widetilde\om}(S,T):=\pi_0(\Hom_{\widetilde\om}(S,T))$.

\subsection{The relationship between $\widetilde{\om}$ and $\widetilde\om_0$}

The category $\widetilde\om_0$ sits between $\widetilde\om$ and $\om$ in the sense of the following proposition.

\begin{prop}
There is a functor $q:\widetilde\Omega\rightarrow\widetilde\Omega_0$, which is the identity on objects and induces a homotopy equivalence on each morphism space. Moreover, the composition 
$p\circ q=\tilde{p}:\widetilde\om \to\om$ is the projection functor of Proposition~\ref{prop:pi}. 
\end{prop}

\begin{proof}
    Fix two objects $S,T\in\widetilde\om$. Recall from Definition~\ref{def:omega-hat} that 
\[\Hom_{\widetilde{\om}}(S,T)=\coprod\limits_{g\in\Hom_{\om}(S,T)}K_g\]
for $K_g=|\text{Path}_{\om}(S,T)_g|$ is the realization of the poset of factorizations of $g$, and $K_g$ is contractible. Likewise by Definition~\ref{def: Omega W} \[\Hom_{\widetilde{\om}_0}(S,T)=\coprod\limits_{g\in\Hom_{\om}(S,T)}L_g\]
and $L_g=\prod\limits_{v\in V(S)}|\B(g(C_v))|$ is the realization of the poset $\mathcal{L}_g$ of bracketings of the trees $g(C_v)$, with $L_g$ likewise contractible. So to prove the proposition, it is enough to produce a functor $q$ which is the identity on objects and takes $K_g$ to $L_g$ for each $g$.  We will define the functor by defining a poset map 
$$q_g: \text{Path}_{\om}(S,T)_g\to \mathcal{L}_g$$
and show that it is compatible with composition. 

\medskip
    
Fix a map $g:S\rightarrow T$ in $\om$. An object of $\Path_{\om}(S,T)_g$ is a factorization $(g_1,\dots, g_n)$ of $g$ and to such a factorization of $g$, for each $v\in V(S)$, we associate a  bracketing of $g(C_v)$ as follows: set 
$$B_v=\{S_w=g_n\circ\dots\circ g_{i+1}(C_w)\}_{\begin{subarray}{l}1\le i\le n-1\\ w\in V(g_{i}\circ\dots\circ g_1(C_v))\\ S_w\subsetneq g(C_v) \ \textrm{large} \end{subarray}}$$
This is a (possibly empty) bracketing as these sets are by definition  nested. We then define 
$q_g(g_1,\dots,g_n)=(B_v)_{v\in V(S)}$. Note that this association is a map of posets as refining a factorization will correspond under $q_g$ to an inclusion of bracketings.

We are left to check that the maps $q_g$ assemble to define a functor, i.e. that they are compatible with composition in $\widetilde\om$ and $\widetilde\om_0$. Let $f:R\to S$ be another morphism in $\om$. We need to check that 
$$\xymatrix{\Path_{\om}(S,T)_g\times \Path_{\om}(R,S)_f \ar[d]_{q_g\times q_f} \ar[r] & \Path_{\om}(R,T)_{g\circ f} \ar[d]^{q_{g\circ f}}\\
\mathcal{L}^{op}_g\times \mathcal{L}^{op}_f \ar[r] & \mathcal{L}^{op}_{g\circ f}}$$
commutes. Because the target is a poset, it is enough to check that it commutes on objects. 
Let 
%$([n],\theta)$ be an object of $\Path_{\om}(S,T)_g$ and $([m],\psi)$ and object of $\Path_{\om}(R,S)_f$. They correspond to factorisations 
$(g_1,\dots,g_n)$ and $(f_1,\dots,f_m)$ be objects of  $\Path_{\om}(S,T)_g$ and $\Path_{\om}(R,S)_f$. By definition, their composition is $(f_1,\dots, f_m,g_1,\dots,g_n)\in \Path_{\om}(R,T)_{g\circ f}$. We have 
$q_f(f_1,\dots,f_m)=(B^f_x)_{x\in V(R)}$ and $q_g(g_1,\dots,g_n)=(B^g_v)_{v\in V(S)}$ with 
$$B^f_x=\{S_y=f_m\circ\dots\circ f_{i+1}(C_y)\}_{\begin{subarray}{l}1\le i\le m-1\\ y\in f_{i}\circ\dots\circ f_1(C_x)\\ S_y\subsetneq f(C_x) \ \textrm{large} \end{subarray}}$$
and bracketing of $f(C_x)\subset S$, and 
$$B^g_v=\{S_w=g_n\circ\dots\circ g_{i+1}(C_w)\}_{\begin{subarray}{l}1\le i\le n-1\\ w\in g_{i}\circ\dots\circ g_1(C_v)\\ S_w\subsetneq g(C_v) \ \textrm{large} \end{subarray}}$$
a bracketing of $g(C_v)\subset T$. By definition, $q_g(g_1,\dots,g_n)\circ q_f(f_1,\dots,f_m)$ is the collection $(\bar B_x)_{x\in V(R)}$ of bracketings of each tree $g\circ f(C_x)\subset T$ defined by 
$$\bar B_x= \Big(\bigcup_{v\in f(C_x)} B^g_v\Big) \cup \Big(\bigcup_{\begin{subarray}{c}v\in f(C_x)\\ g(C_v)\subsetneq g\circ f(C_x) \ \textrm{large}\end{subarray}} \{g(C_v)\}\Big) \cup 
\Big(\bigcup_{\begin{subarray}{c}S_y\in B^f_x \\ g(S_y)  \ \textrm{large}\end{subarray}} \{g(S_y)\}\Big),$$
where %$g(B_x)=\{g(S_y)\}_{S_y\in B_x}$ and 
$B^g_v$ is considered as a bracketing of $g\circ f(C_x)$ via the inclusion $g(C_v)\subset g\circ f(C_x)$.  
Now we see that this is exactly the bracketing of $g\circ f(C_x)$ defined by the factorization $(f_1,\dots,f_m,g_1,\dots,g_m)$, which indeed is the union of the sets 
\begin{align*}
\{g(S_y)=g\circ f_m\circ\dots\circ f_{i+1}(C_y)\}_{\begin{subarray}{l}1\le i\le m-1\\ y\in f_{i}\circ\dots\circ f_1(C_x)\\ g(S_y)\subsetneq g\circ f(C_x) \ \textrm{large} \end{subarray}}
&\cup \{S_v=g(C_v)\}_{\begin{subarray}{l} v\in f(C_x)\\ S_v\subsetneq g\circ f(C_x) \ \textrm{large} \end{subarray}}\\
&\cup \{S_w=g_n\circ\dots\circ g_{i+1}(C_w)\}_{\begin{subarray}{l}1\le i\le n-1\\ w\in g_{i}\circ\dots\circ g_1\circ f(C_x)\\ S_w\subsetneq g\circ  f(C_x) \ \textrm{large} \end{subarray}}.
\end{align*}
Hence the poset maps $q_g$ assemble to define a functor $q:\widetilde\om\to \widetilde\om_0$ as claimed. Moreover, one readily checks that the composition with the projection $p:\widetilde\om_0\to \om$ is the canonical projection $\tilde{p}:\widetilde\om\to \om$. 
\end{proof}

\subsection{$W\OO$--algebras as $\widetilde\om$--diagrams}

In Section~\ref{sec: homotopy dendroidal spaces} we showed that $B\OO$-algebras describe dendroidal Segal spaces. For completeness, we now show how homotopy dendroidal spaces $\mathcal{S}^{\widetilde\om^{op}}$ are related to $W\OO$-algebras. 

We will only need to consider $\widetilde\om$--diagrams $X:\widetilde\om^{op}\to \mathcal{S}$ that are reduced in the strict sense, i.e.~such that $X(\eta)=*$. Recall that in this case, for $X:\om\to \mathcal{S}$, the Segal map becomes the map 
$$X(T) \xrightarrow{\chi} \prod_{v\in V(T)}X(C_v)$$
induced by the restriction maps $T\to C_v$ in $\om^{op}$. Considering these morphisms as morphisms of $\widetilde\om$, we likewise have a Segal map for $X:\widetilde\om\to \mathcal{S}$ in this strictly reduced case.

In analogy to the case of dendroidal and homotopy dendroidal spaces, let $(\mathcal{S}^{\widetilde\om^{op}})_{strict}$ denote the full subcategory of $\mathcal{S}^{\widetilde\om^{op}}$ of $\widetilde\om$--diagrams $X:\widetilde\om^{op}\to \mathcal{S}$ such that $X(\eta)=*$ and such that the Segal map $\chi$ as above is an isomorphism for every $T\neq \eta$. We have the following:

\begin{theorem}\label{thm:WO-Omega}
There exists a functor \[\begin{tikzcd}\Psi:W\OO\mathrm{-Alg}\rightarrow (\mathcal{S}^{\widetilde\om^{op}})_{strict} \end{tikzcd}\] that embeds the category of $W\OO$-algebras as a full subcategory of the category of strictly reduced $\widetilde\om$--diagrams satisfying the strict Segal condition.
\end{theorem}

As we will see in the proof, $\widetilde\om$--diagrams are governed by a version of $W\OO$ where the trees $\mathbb{T}$ have an additional level structure, and $W\OO$--algebras identify then as the subcategory of diagrams where this level structure does not matter. If one wished to describe a category of homotopy dendroidal spaces which is isomorphic to $W\OO$-algebras, one could use this observation to take an appropriate quotient of $\widetilde\om$. As this is particularly messy, and not the main focus of this article, we have elected not to include such a construction. 

\begin{proof}
The proof is similar to that of  Theorem~\ref{thm:W_0O-Omega_0} treating the case of $B\OO$--algebras.  We start with the definition of the  functor $\Psi$. Let $\pp=\{\pp(n)\}_{n\ge 0}$ be a $W\OO$--algebra with structure maps
$$\alpha_\pp:W\OO(n;m_1,\dots,m_k)\times \pp(m_1)\times\dots\times \pp(m_k)\longrightarrow \pp(n).$$ We associate to this data an $\widetilde\om$--diagram 
$$\Psi(\pp)=\Psi(\pp,\alpha_\pp):\widetilde\om^{op} \rightarrow \mathcal{S}$$
as follows. Set $\Psi(\pp)(\eta)=*$ and, for $T\not=\eta$ in $\widetilde\om$, set
$$\Psi(\pp)(T)=\prod_{w\in V(T)}\pp(|w|).$$
For every morphism $g:S\to T$ in $\om$, we need to define maps
$$\Psi(\pp)(g):K_g\times \prod_{w\in V(T)}\pp(|w|) \longrightarrow \prod_{v\in V(S)}\pp(|v|).$$
As in Theorem~\ref{thm:W_0O-Omega_0}, we do this one vertex of $S$ at a time. 

Recall that $K_g$ is the realisation of the poset $\Path_{\om}(S,T)_g$
of factorizations 
$$S\xrightarrow{g_1} T_1\xrightarrow{g_2}\dots \xrightarrow{g_{n-1}} T_{n-1}\xrightarrow{g_n}T$$
of $g$ in $\om$. For each $v\in V(S)$, we consider the restriction of these maps to $C_v\in S$: 
\begin{equation}\label{equ:gCv}
C_v\xrightarrow{g_1} g_1(C_v)\xrightarrow{g_2}\dots \xrightarrow{g_{n-1}} g_{n-1}\circ\dots\circ g_1(C_v) \xrightarrow{g_n}
g(C_v)\subset T.\tag{$\ast$}
\end{equation}
Recall from Remark~\ref{rem:WO as poset} that $W\OO$ is the realization of an operad in posets, whose elements are those of the free operad $F\OO$ (identifying elements of $F\OO$ with elements of $W\OO$ in which all weights of internal edges are $1$). 
We will now use the restriction (\ref{equ:gCv}) of $(g_1,\dots,g_n)$ to $C_v$ to construct a labeled planar tree $$(\mathbb{T},f,\lambda,p)\in F\OO(|v|;(|w|)_{w\in V(g(C_v))})$$
by induction on the height of the tree:
\medskip

Starting at the root, we attach a vertex $\bar v$ of valence $|V(g_1(C_v))|$. The incoming edges of $\bar v$ are labelled in accordance with $(g_1(C_v),\sigma_v,\tau_v)$,  where $\sigma_v$ is a chosen ordering of the vertices of the tree $g_1(C_v)$, and $\tau_v$ is induced by the planar structure of $g_1(C_v)\subset T_1$. Specifically, the incoming edges of $\bar v$ are labeled by the vertices of $g_1(C_v)$ and ordered via the map $\sigma_v$.
\medskip

For each vertex $w\in g_1(C_v)$, which is now an incoming edge of $\bar v$, we can attach a vertex $\bar w$ of valence $|V(g_2(C_w))|$. These incoming edges are labeled with the tuple $(g_2(C_w),\sigma_w,\tau_w)$, as in the previous case.

\medskip

More generally, for vertices with height $2\le i\le n$, the tree $\mathbb{T}$ has a vertex $\bar y$ for every vertex $y$ in $(g_{i-1}\circ \dots\circ g_1)(C_v)$, attached to the previously constructed vertex $\bar x$ associated to the vertex $x\in (g_{i-2}\circ \dots\circ g_1)(C_v)$ satisfying that $y\in g_{i-1}(C_x)$. We label $\bar y$ by the tuple $(g_i(\bar y),\sigma_y,\tau_y)$ with $\tau_y$ induced by the planar structure of $T_i$, giving $C_{\bar y}$ the planar structure dictated by the chosen $\sigma_y$.

\medskip

We now set $f:E(\mathbb{T})\rightarrow\mathbb{N}$ to be the unique meaningful colouring which makes $\mathbb{T}$ an element of $F\OO(|v|;|w_1|,\dots,|w_k|)$. We set the ordering $\sigma$ of the vertices $w_1,\dots,w_k$ of $g(C_v)$ in accordance to the resulting planar structure on $\mathbb{T}$.  As the set of vertices of $g(C_v)$ is also the set of leaves of the tree $\mathbb{T}$, this also defines $\lambda$. We note that the tree constructed this way is in no way reduced and will, a priori, have many arity one vertices labelled by identities. We can use relations defining $F\OO$, however, to remove such vertices and give an equivalent element in $F\OO$.

\medskip
This assignment of the restriction of a factorisation (\ref{equ:gCv}) to a labelled tree $\mathbb{T}$ respects the poset structure of 
$\Path_{\om}(S,T)_g$ and $W\OO$ as refining the factorisation 
corresponds to undoing the collapse of edges, namely if 
$(g_1,\dots,g_n)\le (g_1',\dots,g_m')$, then the image 
$(\mathbb{T},f,\lambda,p)$ of the first factorisation can be obtained from the image $(\mathbb{T}',f',\lambda',p')$ of the second by 
collapsing the edges corresponding to the added levels, as collapsing level in the tree correspond in this construction to composing 
consecutive maps $g_i$.

In this way we can apply the structure map $\alpha_\pp$ one vertex at a time and define a map
\begin{equation}\label{equ:la2}
\alpha_v: |\Path_{\om}(S,T)_g|\times \prod_{w_i\in V(g(C_v))}\pp(|w_i|) \longrightarrow \pp(|v|)\end{equation} 
and we can define $$\Psi(\pp)(g)=(\alpha_v)_{v\in V(S)}.$$ By construction, the action of $\Psi(P)$ on morphisms commutes with composition in $\widetilde\om$, and thus $\Psi(\pp):\widetilde\om^{op}\rightarrow\mathcal{S}$ defines a functor. That the Segal map for $\Psi(\pp)$  
%(\ref{eq: hSegal}) 
is 
an isomorphism for every $T\neq \eta$ follows immediately from our definition of $\Psi(\pp)$.

\medskip

The assignment $\pp\mapsto \Psi(\pp)$ requires only the data of underlying symmetric sequence of $\pp$ and the algebra structure maps $\alpha_\pp$. 
% and the decoration of vertices of $\mathbb{T}$.  
 This data is natural under maps of $W\OO$-algebras and thus
\[\Psi:W\mathcal{O}\mathrm{-Alg}_{\mathcal{S}} \longrightarrow \mathcal{S}^{\widetilde\Omega^{op}}\] is a functor.

\medskip

It remains to check that $\Psi$ is an embedding of a full subcategory. Injectivity on objects follows from the fact that if $\pp$ and $\mathcal{Q}$ satisfy that $\Psi(\pp)=\Psi(\mathcal{Q})$, then we necessarily have that $\pp(n)=\mathcal{Q}(n)$ for each $n$, as given by the value at the corolla, with agreeing symmetric group actions as given by the isomorphisms of corollas, and the structure maps $\alpha_\pp$ and $\alpha_{\mathcal{Q}}$ likewise must agree as the value of the structure map on every element of $W\OO$ is the value of the functor $\Psi(\pp)=\Psi(\mathcal{Q})$ on an associated morphism of $\widetilde\om$ obtained by choosing a level structure on the tree and interpreting the collapse of each level of the tree as a morphism in $\om$. As morphisms of $W\OO$--algebras are determined by what they do on spaces $\pp(n)$, we see that the functor is faithful. It is also full as natural transformations between diagrams originating from $W\OO$--algebras, will necessarily respect the $W\OO$-algebra structure of their values at the corollas. 

\end{proof}

\begin{remark}\label{remark: hcn}
The reader might be tempted to compare the functor $\Psi(-)$ from Theorem~\ref{thm:WO-Omega} with the \emph{homotopy coherent nerve} of a topological operad $\pp$. This is a functor $$w^*:\OO\mathrm{-Alg}\rightarrow \set^{\om^{op}}$$ defined by $$(w^*\pp)(T)=\Hom_{\mathsf{Op}}(W\Omega(T),\pp),$$ where $W\Omega(T)$ denotes the Boardman-Vogt $W$-construction applied to the free operad generated by a tree $T$ (Example~\ref{example: Omega(T)}). The functors $\Psi$ and $w^*$ are not equivalent on operads, though if one has a $W\OO$-algebra $\pp$ which happens to be an operad then one can define a dendroidal space $X_\pp\in\mathcal{S}^{\om^{op}}/w^*\pp$, where the later denotes the slice category.  For more on this point of view, see \cite[Remark 6.2]{mw07} or \cite[Corollary 1.7]{BdBM}. 
\end{remark}

\bibliographystyle{plain}
\bibliography{bibliography.bib}

\end{document}